\documentclass[a4paper, 11pt, reqno]{amsart}

\usepackage{amsmath,amssymb,amsthm}
\usepackage{enumerate}
\usepackage{cases}
\usepackage{color}
\usepackage[mathscr]{euscript}
\usepackage{url}
\usepackage[all]{xy}
\usepackage{booktabs}
\usepackage{graphicx}
\usepackage{epstopdf}
\usepackage{algorithmic}
\usepackage[caption=false]{subfig}
\usepackage{mathtools}
\usepackage{multirow}
\usepackage{makecell}
\usepackage{bigdelim}
\usepackage[ruled,vlined]{algorithm2e}
\usepackage{threeparttable}
\usepackage{hyperref}
\usepackage{xcolor}
\definecolor{cobalt}{rgb}{0.0, 0.28, 0.67}
\definecolor{darkblue}{rgb}{0.0, 0.0, 0.55}
\definecolor{color1}{RGB}{246, 189, 96}
\definecolor{color2}{RGB}{247, 237, 226}
\definecolor{color3}{RGB}{245, 202, 195}
\definecolor{color4}{RGB}{132, 165, 157}
\definecolor{color5}{RGB}{242, 132, 130}
\definecolor{color6}{RGB}{132, 220, 198}
\definecolor{color7}{RGB}{165, 255, 214}
\usepackage{mathdots}

\usepackage{tikz}
\usetikzlibrary{shapes}
\usetikzlibrary{arrows}
\usetikzlibrary{matrix}
\usetikzlibrary{cd}

\newcommand{\GL}{\mathbf{GL}}
\newcommand{\SL}{\mathbf{SL}}

\newcommand{\St}{\mathbf{St}}
\newcommand{\RSL}{\mathbf{RSL}}

\newcommand{\TSDO}{{$\mathbf{T}_{m,n}$-SDO}}

\hypersetup
{
bookmarksopen=true,
pdftitle={Nonsymmetric Jacobi-G algorithm},
pdfauthor={SL},
pdfsubject={15A12},
pdfmenubar=true,
pdfhighlight=/O,
colorlinks=true,
pdfpagelayout=SinglePage,
pdffitwindow=true,
linkcolor=cobalt,
citecolor=cobalt,
urlcolor=cobalt
}

\usepackage[capitalize,nameinlink,noabbrev]{cleveref}

\usepackage[scale=0.75]{geometry}
\theoremstyle{definition}

\newtheorem{theorem}{Theorem}[section]
\newtheorem{corollary}[theorem]{Corollary}
\newtheorem{lemma}[theorem]{Lemma}
\newtheorem{definition}[theorem]{Definition}

\newtheorem{remark}[theorem]{Remark}

\newtheorem{example}[theorem]{Example}

\definecolor{darkred}{rgb}{0.7,0,0}
\definecolor{darkgreen}{rgb}{0,0.46,0}
\definecolor{purple}{rgb}{0.6,0,0.5}

\newcommand{\tens}[1]{\boldsymbol{\mathcal{#1}}}

\newcommand{\matr}[1]{\boldsymbol{#1}}

\newcommand{\matrelem}[1]{\mathrm{#1}}

\newcommand{\vect}[1]{\boldsymbol{#1}}

\newcommand{\set}[1]{\mathcal{#1}}

\newcommand{\T}{{\sf T}}        % transposition
\renewcommand{\top}{{\sf T}}        % transposition
\makeatletter
\newcommand{\rank}[1]{\mathop{\operator@font rank}(#1)}
\newcommand{\colrank}[1]{\mathop{\operator@font colrank}\{#1\}}
\newcommand{\krank}[1]{\mathop{\operator@font krank}\{#1\}}
\newcommand{\trace}[1]{\mathop{\operator@font tr}\left(#1\right)}
\newcommand{\symmm}[1]{\mathop{\operator@font sym}\left(#1\right)}
\newcommand{\skeww}[1]{\mathop{\operator@font skew}(#1)}
\newcommand{\Diag}[1]{\mathop{\operator@font \textbf{Diag}}\{#1\}}    % a diagonal matrix
\newcommand{\diag}[1]{\mathop{\operator@font diag}\{#1\}}    % a vector
\newcommand{\Span}[1]{\mathop{\operator@font Span}\{#1\}}    % a space
\newcommand{\argmin}{\mathop{\operator@font argmin}}

\newcommand{\offdiag}[1]{\mathop{\operator@font offdiag}\{#1\}}    % a vector
\newcommand{\Proj}[2]{\mathop{\operator@font Proj_{#1}}{#2}}
\newcommand{\ProjGrad}[2]{\mathop{{\operator@font grad} }#1(#2)}

\newcommand{\expp}[1]{\mathop{\operator@font exp}\left(#1\right)}
\makeatother

\newcommand{\eqdef}{\stackrel{\sf def}{=}}
\newcommand{\RR}{\mathbb{R}}
\newcommand{\CC}{\mathbb{C}}

\newcommand{\SON}{\mathbf{SO}}

\newcommand{\contr}[1]{\mathop{\bullet_{#1}}}   % Contraction

\makeatletter
\def\bbordermatrix#1{\begingroup \m@th
\@tempdima 4.75\p@
\setbox\z@\vbox{%
\def\cr{\crcr\noalign{\kern2\p@\global\let\cr\endline}}%
\ialign{$##$\hfil\kern2\p@\kern\@tempdima&\thinspace\hfil$##$\hfil
	&&\quad\hfil$##$\hfil\crcr
	\omit\strut\hfil\crcr\noalign{\kern-\baselineskip}%
	#1\crcr\omit\strut\cr}}%
\setbox\tw@\vbox{\unvcopy\z@\global\setbox\@ne\lastbox}%
\setbox\tw@\hbox{\unhbox\@ne\unskip\global\setbox\@ne\lastbox}%
\setbox\tw@\hbox{$\kern\wd\@ne\kern-\@tempdima\left[\kern-\wd\@ne
\global\setbox\@ne\vbox{\box\@ne\kern2\p@}%
\vcenter{\kern-\ht\@ne\unvbox\z@\kern-\baselineskip}\,\right]$}%
\null\;\vbox{\kern\ht\@ne\box\tw@}\endgroup}
\makeatother

\begin{document}

\title[Projectively and weakly simultaneously diagonalizable matrices]{Projectively and weakly simultaneously diagonalizable matrices and their applications}

\author[Wentao Ding]{Wentao Ding$^{\dagger}$}
\thanks{$\dagger$ School of Data Science, The Chinese University of Hong Kong, Shenzhen, Guangdong, China (wentaoding@link.cuhk.edu.cn)}

\author[Jianze Li]{Jianze Li$^{\ddagger}$}
\thanks{$\ddagger$ Shenzhen Research Institute of Big Data, The Chinese University of Hong Kong, Shenzhen, Guangdong, China (lijianze@gmail.com)}

\author[Shuzhong Zhang]{Shuzhong Zhang$^{\natural}$}
\thanks{$\natural$ Department of Industrial and Systems Engineering, University of Minnesota, Minneapolis, MN 55455, USA (zhangs@umn.edu)}

\date{\today}

\keywords{simultaneous diagonalization, weak simultaneous diagonalization, projective simultaneous diagonalization, canonical form, quadratically constrained quadratic programming, independent component analysis}

\thanks{This work was supported in part by the National Natural Science Foundation of China (No. 11601371) and the GuangDong Basic and Applied Basic Research Foundation (No. 2021A1515010232).}

\subjclass[2020]{15A20, 15A21, 15A22, 90C20, 90C30}

\begin{abstract}
 Characterizing simultaneously diagonalizable (SD) matrices has been receiving considerable attention in the recent decades due to its wide applications and its role in matrix analysis. 
 However, the notion of SD matrices is arguably still restrictive for wider applications. 
 In this paper, we consider two error measures related to the simultaneous diagonalization of matrices, and propose several new variants of SD thereof; in particular, TWSD, TWSD-B, \(\textbf{T}_{m,n}\)-SD (SDO), DWSD and \( \textbf{D}_{m,n} \)-SD (SDO). Those are all weaker forms of SD. We derive various sufficient and/or necessary conditions of them under different assumptions, and show the relationships between these new notions. Finally, we discuss the applications of these new notions in, e.g., quadratically constrained quadratic programming (QCQP) and %theoretical study of 
 independent component analysis (ICA).
\end{abstract}

\maketitle
\section{Introduction}\label{sec:introduction}

% \jl
% two matrices SDO and SD

% a finite number of matrices SDO and SD

% applications

% two matrices WSD (Jiang's work)

% applications

% In this paper, we

% first applications

% second applications

% the structure of this paper

% \fin

Let $\SL_{m}(\RR)\eqdef\{\matr{P}\in\RR^{m\times m}, {\rm det}(\matr{P})=1\}$ be the \emph{special linear group} with $m\geq 1$, and $\SON_{m}\subseteq\SL_{m}(\RR)$ be the \emph{special orthogonal group}. 
Let $\textbf{symm}(\RR^{m\times m})\subseteq\RR^{m\times m}$ be the set of all \emph{symmetric} matrices, and 
\begin{equation}\label{set_C}
\mathcal{C}\eqdef\{\matr{A}_{i}\}_{1\leq i\leq L} \subseteq \textbf{symm}(\RR^{m\times m})
\end{equation}
be a set of $L$ symmetric matrices. 
Then the set $\mathcal{C}$ is said to be {\it simultaneously diagonalizable on $\SL_{m}(\RR)$} (SD) \cite{horn2012matrix} if there exists a matrix $\matr{P}\in \SL_{m}(\RR)$ such that
$\matr{P}^\T \matr{A}_{i} \matr{P}$ 
is diagonal for all $1\leq i\leq L$. 
In particular, it is said to be {\it simultaneously diagonalizable on $\SON_{m}$} (SDO) \cite{horn2012matrix} if there exists a matrix $\matr{P}\in \SON_{m}$ such that
$\matr{P}^\T \matr{A}_{i} \matr{P}$ is diagonal for all $1\leq i\leq L$.

% Given a set of symmetric matrices \( \mathcal{C} \subseteq \mathbb{R}^{m \times m}\), if there exists a nonsingular matrix \( \matr{P} \) such that \( \matr{P}^{\T}\matr{A}\matr{P} \) is diagonal for every \( \matr{A} \in \mathcal{C} \), then the set \( \mathcal{C} \) is said to be {\it simultaneously diagonalizable} (SD). In addition, if \( \matr{P} \) is an orthogonal matrix, then it is said to be {\it simultaneously diagonalizable on orthogonal group} (SDO).

The notions of SDO and SD are closely related to many intrinsic properties of matrices, including the \emph{community} and \emph{eigenvalue} \cite{weierstrass1868theorie,uhlig1979recurring}. For example, a well-known theorem, which can be dated back to 1868 \cite{weierstrass1868theorie}, states that the set \( \mathcal{C} \) in \eqref{set_C} is SDO if and only if all of the matrices in \( \mathcal{C} \) commute with each other.
%In contrast, it is much more difficult to find conditions that are easy to check for SD. 
Throughout the last few decades, a variety of characterizations of SD have also emerged under different conditions.
When the set \( \mathcal{C} \) in \eqref{set_C} is a nonsingular symmetric matrix pair (\emph{i.e.}, $L=2$, $\matr{A}_1$ or $\matr{A}_2$ is nonsingular), the study of its canonical form suggests a sufficient and necessary condition for SD by the real Jordan normal form \cite{uhlig1973SimultaneousBlockDiagonalizationa, uhlig1976canonical}; see also \cite{uhlig1979recurring} for a historical review.
In \cite{jiang2016simultaneous}, this result was further extended to the case of a singular symmetric matrix pair.
For the set \( \mathcal{C} \) in \eqref{set_C} with $L > 2$, the procedures to check whether \( \mathcal{C} \) is SD were also derived in \cite{jiang2016simultaneous} if \( \mathcal{C} \) has a positive semi-definite pencil, and in \cite{nguyen2020simultaneous, le2020simultaneous, bustamante2020SolvingProblemSimultaneousa} for real and complex cases without any assumption.

Apart from their theoretical importance, SDO and SD also have practical applications; see e.g.~\cite{luo2020effective,vollgraf2006quadratic,wang2021new, wang2021fast}, including \emph{signal processing} \cite{Cardoso93:JADE,cardoso1996jacobi, Como10:book, yeredor2002NonorthogonalJointDiagonalization} and \emph{quadratically constrained quadratic programming} (QCQP) \cite{ben2014hidden,jiang2016simultaneous, nguyen2020simultaneous}. To be more specific, if the quadratic forms in the objective function and constraints are SD, then there exist efficient algorithms to solve those problem via reformulations \cite{ben2014hidden,jiang2016simultaneous,zhou2020simultaneous, zhou2019simultaneous}. Also, if the SDP relaxation is employed to this kind of problem \cite{luo2010semidefinite,sturm2003cones,todd2001semidefinite,ye2003new}, then more theoretical results can be obtained; see \cite{burer2020ExactSemidefiniteFormulations,wang2020generalized,wang2020geometric,wang2021TightnessSDPRelaxations} and the references therein.

In spite of these remarkable properties and wide applications of SDO and SD, the sets of matrices that are SD or SDO are still limited.
To overcome this drawback, Wang and Jiang \cite{wang2021new} proposed two weaker notions: \emph{almost SDC} (ASDC) and \( d \)-\emph{restricted SDC} (\( d\)-RSDC). They derived full characterizations of the ASDC pairs and nonsingular ASDC triples, and proved that all singular pairs are ASDC and almost all pairs are 1-RSDC over complex field $\CC$ and real field $\RR$.
As an application, they discussed how to use these two properties to solve the QCQP models with a single quadratic constraint over a polytope.

% In the context of the approximate diagonalization of matrices, the following optimization problem is considered:
% \begin{equation}
%   \label{eq:approximate-matrix-diag}
% \min_{\matr{P}} \sum_{\matr{A} \in \mathcal{C}} \|\operatorname{offdiag} (\matr{P}^{\T}\matr{A}\matr{P}) \|^2,
% \end{equation}
% with different kinds of contraints of \( \matr{P} \). When the feasible set of \( \matr{P} \) is compact, for instance, orthogonal group, the problem has been extensively studied, see for example \cite{li2020gradient,maurandi2014jacobi,ULC2019}. However, when the feasible set of \( \matr{P} \) is not compact, such as the set of nonsingular matrix in \cite{afsari2004gradient,afsari2006simple}, there might be the case when the minimum of the optimization problem is unattinable. Especially, when the minimum is \( 0 \) but unattinable, the set \( \mathcal{C} \) is not SD but there exists a sequence \( \{\matr{P}_k\}_{k \geq 1} \) such that the off-diagonal elements of \( \matr{P}_k^{\T} \matr{A} \matr{P}_k \) converge to 0. To the best of our knowledge, there is no study about this case before.
%\subsection{Contribution}

In this paper, from an even broader perspective, we shall consider more weaker versions of SDO and SD, and study their theoretical characterizations and applications. 
%Before that, we first need to 
To this end, let us first introduce some definitions and notations. 
Let 
$\St(m,n) \eqdef\{\matr{P}\in\RR^{n\times m}, \matr{P}^{\T}\matr{P}= \matr{I}_m\}$ 
be the \emph{Stiefel manifold} with $n\geq m$. 
Define the \emph{rectangular special linear set} \cite{li2020gradient} as
\begin{align*}
\textbf{RSL}(m,n) \eqdef \{\matr{P}\in\RR^{n\times m}, \matr{P}^{\T}\matr{P}\in\SL_{m}(\RR)\},
\end{align*}
which can be seen as a non-orthogonal analogue of $\St(m,n)$. It is easy to verify that\footnote{This equation means that a matrix $\matr{Z}\in\textbf{RSL}(m,n)$ if and only if there exist $\matr{Y}\in\St(m,n)$ and $\matr{X}\in\SL_{m}(\RR)$ such that $\matr{Z}=\matr{Y}\matr{X}$.}:
\begin{align}
\textbf{RSL}(m,n) = \St(m,n)\times\SL_{m}(\RR).
\label{eq:rela-equivalent}
\end{align}
Let $\mathbf{D}_{n}\subseteq\RR^{n\times n}$ be the set of all \emph{diagonal} matrices.
Let the set $\mathcal{C}$ be as in \eqref{set_C}.
We define two error-measuring functions
\begin{align}\label{eq:func_T}
\varphi_{\rm T}(\matr{P},\mathcal{D}) &\eqdef \sum_{i=1}^{L} \|\matr{P}\matr{A}_i\matr{P}^\T-\matr{D}^{(i)}\|^2,\\
\varphi_{\rm D}(\matr{P},\mathcal{D}) &\eqdef \sum_{i=1}^{L} \|\matr{A}_i-\matr{P}^\T\matr{D}^{(i)}\matr{P}\|^2,\label{eq:func_D}
\end{align}
for $\matr{P}\in\textbf{RSL}(m,n)$ ($\matr{P}\in\St(m,n)$), and  $\mathcal{D}=\{\matr{D}^{(i)}\}_{1\leq i\leq L}\subseteq\mathbf{D}_{n}$. 
It is clear that the set $\mathcal{C}$ is SD (SDO) if and only if $n=m$, and there exists $\matr{P}\in\SL_{m}(\RR)$ ($\matr{P}\in\SON_{m}$) and $\mathcal{D}\subseteq\mathbf{D}_{m}$ such that $\varphi_{\rm T}(\matr{P},\mathcal{D})=0$ (equivalently, $\varphi_{\rm D}(\matr{P},\mathcal{D})=0$).

In this paper, using the function $\varphi_{\rm T}$ in \eqref{eq:func_T}, we will extend the SDO and SD from two different angles. %perspectives.
The first one is to allow \( \matr{P} \) to be %is not necessarily 
non-square, {\it i.e.}, \( n \geq m \), 
%This type of notions will be named by 
which we shall term as ``projectively'' SD, including $\textbf{T}_{m,n}$-SDO and $\textbf{T}_{m,n}$-SD for \( \matr{P} \in \St(m,n) \) and \( \matr{P} \in\RSL(m,n) \), respectively.
The second one is to allow that the off-diagonal elements of \( \matr{P}^{\T}\matr{A}_i\matr{P} \) to be not necessarily exactly equal to \( 0 \), but asymptotically approaches \( 0\), %very close to \( 0 \).
%This type of notions will be named by
which we shall term as ``weakly'' SD, including the TWSD-B and TWSD.
Similarly, using the function $\varphi_{\rm D}$ in \eqref{eq:func_D}, we will also introduce several definitions weaker than SDO and SD, including the $\textbf{D}_{m,n}$-SDO, $\textbf{D}_{m,n}$-SD, and DWSD. 
These new notions and definitions are summarized in \Cref{table-example-3-0}.
%More details can be found in ?.

{\renewcommand{\arraystretch}{1.4}
\begin{table}[h!]
\centering
\caption{A summary of the new weaker notions}
\label{table-example-3-0}
\scalebox{0.8}{
\begin{tabular}{|c|c|c|p{7cm}|c|}
\hline
Classical notions        & Functions                                  & Types                    & \makecell[c]{Full name}         & Short name            \\ \hline
\multirow{7}{*}[-4em]{\makecell[c]{SD, SDO}} & \multirow{4}{*}[-3em]{\(\varphi_{\rm T}\) in \eqref{eq:func_T}} & \multirow{2}{*}[-1.2em]{Projectively} &  Transformation based projectively simultaneously diagonalizable on \( \St(m,n) \) & \multirow{1}{*}[-0.7em]{\( \mathbf{T}_{m,n} \)-SDO (\cref{def:TSDO-SD})}    \\ \cline{4-5}
                         &                                            &                          &    Transformation based projectively simultaneously diagonalizable on \( \RSL(m,n) \) & \multirow{1}{*}[-0.7em]{\( \mathbf{T}_{m,n} \)-SD (\cref{def:TSDO-SD})}       \\ \cline{3-5}
                         &                                            & \multirow{2}{*}[-1.2em]{Weakly}  &    Transformation based weakly simultaneously diagonalizable & \multirow{1}{*}[-0.7em]{TWSD (\cref{def:twsd})}   \\ \cline{4-5}
                         &                                            &                          &     Bounded transformation based weakly simultaneously diagonalizable & \multirow{1}{*}[-0.7em]{TWSD-B (\cref{def:twsd})} \\ \cline{2-5}
                         & \multirow{4}{*}[-1em]{\(\varphi_{\rm D}\) in \eqref{eq:func_D}} & \multirow{2}{*}[-1.2em]{Projectively}   &    Decomposition based projectively simultaneously diagonalizable on \( \St(m,n) \) & \multirow{1}{*}[-0.7em]{\( \mathbf{D}_{m,n} \)-SDO (\cref{def:D-SDO-SD})}    \\ \cline{4-5}
                         &                                            &                          &    Decomposition based projectively simultaneously diagonalizable on \( \RSL(m,n) \) & \multirow{1}{*}[-0.7em]{\( \mathbf{D}_{m,n} \)-SD (\cref{def:D-SDO-SD})}    \\ \cline{3-5}
                         &                                            & Weakly    &Decomposition based weakly simultaneously diagonalizable & \multirow{1}{*}[-0.7em]{DWSD (\cref{def:DWSD})} \\ \hline
\end{tabular}}
\end{table}}

For these new notions in \Cref{table-example-3-0}, we will develop different characterizations of them under different assumptions. 
For the function $\varphi_{\rm T}$ based notions, we will first prove that there is no difference between \(\mathbf{T}_{m,n}\)-SD (or \( \mathbf{T}_{m,n} \)-SDO) and SD (or SDO) for all sets of matrices. 
Then we will focus on TWSD-B and TWSD. 
For a pair of matrices, we will derive sufficient and necessary conditions for them to be TWSD-B, and develop algorithms to check whether they are TWSD-B or not. 
In particular, it will be shown that any singular pair is TWSD-B. 
For a set of nonsigular matrices ($L\geq 2$), we will also provide a necessary condition and two sufficient conditions to ensure that they are TWSD-B, which will be helpful if one needs to verify whether this %nonsingular 
set is TWSD-B or not. 
In particular, if this set is positive definite, we will show %it will be shown 
that TWSD-B is essentially equivalent to SD, while TWSD is not. 
For the function $\varphi_{\rm D}$ based notions, we will prove that any set of $L$ matrices is \(\mathbf{D}_{m,Lm}\)-SDO.
It will be shown that the ASDC and \(d\)-RSDC proposed in \cite{wang2021new} are exactly the DWSD and $\textbf{D}_{m,n}$-SD, respectively. 
Based on the theoretical results we are going to develop in this paper, as well as  
the characterizations from \cite{wang2021new} about DWSD and $\textbf{D}_{m,n}$-SD (ASDC and \(d\)-RSDC), %we will systematically 
the relationships between these notions from different perspectives are schematically shown in \Cref{fig:relationship-WSD-PSD,figure-relathionships-WSD}.
% Note that \( \varphi_T(\matr{P}, \mathcal{D})= 0 \) is equivalent to \( \matr{P}\matr{A}^{(i)}\matr{P}^{\T} = \matr{D}^{(i)}\), \( \varphi_D(\matr{P}, \mathcal{D})= 0 \) is equivalent to \( \matr{A}^{(i)} = \matr{P}^{\T}\matr{D}^{(i)}\matr{P}\) for all \( 1 \leq i \leq L \). It is clear that $\mathcal{C}$ is SDO (or SD) if and only if $m=n$, and there exist $\matr{P}\in\SON_{m}$ (or $\matr{P}\in\SL_{m}(\RR)$) and $\matr{D}^{(i)}\in\mathbf{D}_{m}$, such that $\varphi_{\rm T}(\matr{P}, \mathcal{D}) =0  $ ($\varphi_{\rm D}(\matr{P}, \mathcal{D})=0$).

% When \( m = n \), \( \varphi_T(\matr{P}, \mathcal{D}) = 0  \) if and only if \( \varphi_D(\matr{P}^{-\T}, \mathcal{D}) = 0 \) if and only if \( \mathcal{C} \) is SD. 

We then consider applications based on these newly introduced concepts. 
As a first application, different from the approximation method used in \cite{wang2021new} to explore the application of DWSD (ASDC) in QCQP problem, we will prove a theoretical result about the QCQP problem with a single constraint. If the quadratic forms in the objective function and constraint are TWSD-B and the constraint is nonsingular, then we can reformulate it as a linear programming problem.
As the second application, using \(\mathbf{D}_{m,n}\)-SDO, we prove a theoretical result about the \emph{independent component analysis}(ICA), which is helpful to solving the \emph{blind source separation}(BSS) problem.

% There are two problems related to them. The first one is QCQP, which has been explored in detail in \cite{wang2021new}. Similar to the approximation method demonstrated there, we also show it is possible to obtain approximate solutions of QCQP if the quadratic forms of it are TWSD-B. The other one is ICA. Actually, both of the two functions \( \varphi_T \) and \( \varphi_D \) are the objective functions to minimize in the context of ICA. By the characterizations of our new notions, we give a theoretical analysis of the cases when the infimum of the objective function is 0 and when the objective function can attain 0 if the congruent matrix is allowed to be nonsquare.

%\subsection{Organization}
The organization of this paper is as follows. In \cref{sec:preliminary}, we recall several characterizations of the classical SDO and SD notions, as well as some results that will be frequently used in the next sections. 
In \cref{sec:trans_based_SD}, we propose the function $\varphi_{\rm T}$ based notions in \Cref{table-example-3-0}, and show that \( \mathbf{T}_{m,n} \)-SD (\( \mathbf{T}_{m,n} \)-SDO) is the same as SD (SDO). Then we derive several necessary and/or sufficient conditions for TWSD-B of a pair of matrices under different assumptions. 
%, as well as the algorithms to check whether a set of matrices is TWSD and obtain congruent matrices if it is. 
In \cref{sec:TWSD}, we study the relationship between TWSD and TWSD-B, through which we also obtain some useful sufficient conditions for TWSD-B. 
In \cref{sec:decom_based_SD}, we propose the $\varphi_{\rm D}$-based analogs in \Cref{table-example-3-0}, and show that the ASDC and \(d\)-RSDC proposed in \cite{wang2021new} are exactly the DWSD and $\textbf{D}_{m,n}$-SD, respectively. 
In particular, we observe that any set of matrices is \(\mathbf{D}_{m,n}\)-SDO when $n$ is large enough. 
In \cref{sec:relationship}, based on the theoretical results we obtain and the characterizations from \cite{wang2021new}, we present the relationships among all these new notions. In \cref{sec:application-QCQP,sec:application-approximate-diagonalization}, we consider the applications of these new notions to the QCQP and BSS problems, respectively.
In \cref{sec:conclusi}, we conclude this paper with some discussions and possible future work. 
%In \cref{sec:experiments}, we present the numerical experiments to show the effectiveness of the proposed algorithms.

% Finally, we close this section with the list of new weaker definitions in this paper for the convenience of readers.
% {\renewcommand{\arraystretch}{1.4}
% \begin{table}[h!]
% \centering
% \caption{A list of the new weaker notions}
% \label{table-list-new-notions}
% \scalebox{0.8}{
% \begin{tabular}{| c | c|}
% \toprule
%   Full name & Short name\\
%   \hline
% Transformation based projectively simultaneously diagonalizable on \( \St(m,n) \) & \( \mathbf{T}_{m,n} \)-SDO (\cref{def:TSDO-SD})\\
%   \hline
% Transformation based projectively simultaneously diagonalizable on \( \RSL(m,n) \) & \( \mathbf{T}_{m,n} \)-SD (\cref{def:TSDO-SD})\\
%   \hline
%   Transformation based weakly simultaneously diagonalizable & TWSD (\cref{def:twsd})\\
%   \hline
%   Bounded transformation based weakly simultaneously diagonalizable & TWSD-B (\cref{def:twsd})\\
%   \hline
% Decomposition based projectively simultaneously diagonalizable on \( \St(m,n) \) & \( \mathbf{D}_{m,n} \)-SDO (\cref{def:D-SDO-SD})\\
%   \hline
% Decomposition based projectively simultaneously diagonalizable on \( \RSL(m,n) \) & \( \mathbf{D}_{m,n} \)-SD (\cref{def:D-SDO-SD})\\
%   \hline
%   Decomposition based weakly simultaneously diagonalizable & DWSD (\cref{def:DWSD})\\
% \bottomrule
% \end{tabular}}
% \end{table}}

\section{Preliminaries}
\label{sec:preliminary}

\subsection{Notations}
Let $\GL_{m}(\RR)\eqdef \{\matr{X}\in\RR^{m\times m}, {\rm det}(\matr{X})\neq 0\}$ be the \emph{general linear
group}. 
Matrices and vectors will be respectively denoted by  bold uppercase letters, \textit{e.g.}, $\matr{A}$, and by bold lowercase letters, \textit{e.g.}, $\vect{u}$; corresponding entries will be denoted by $\matrelem{A}_{ij}$ and $u_i$.
We denote by $\|\cdot\|$ the Frobenius norm of a matrix,
or the Euclidean norm of a vector.
Denote $\matr{I}_{m,n}=[\vect{e}_1,\vect{e}_2,\cdots,\vect{e}_m]\in\RR^{n\times m}$ and  $\matr{I}_{m}=\matr{I}_{m,m}$. Let \( \matr{0}_{m \times n} \)  denote a zero matrix in \( \mathbb{R}^{m \times n} \).
% Let $\GL_{m}(\RR)\subseteq\RR^{m\times m}$ be the set of nonsingular matrices, and $\SL_{m}(\RR)\subseteq\RR^{m\times m}$ be the set of nonsingular matrices with determinant equal to 1. 
For a matrix $\matr{A}\in\RR^{m\times m}$, we denote by $\textbf{diag}(\matr{A})$ the matrix $\matr{A}$ with the offdiagonal elements being set to 0, and by $\textbf{offdiag}(\matr{A})$ the matrix $\matr{A}$ with the diagonal elements being set to 0. 
We denote \( \matr{A} \succ 0 \) (\( \matr{A} \succeq 0 \)), if the matrix \( \matr{A} \) is positive definite (positive semi-definite). 
For multiple square matrices \( \matr{X_1}, \matr{X}_2 , \ldots , \matr{X}_p \), we denote the square block diagonal matrix consisting of them by \( \Diag{\matr{X}_1 , \ldots , \matr{X}_p} \).
We denote by $\set{SD}$($\set{SDO}$) the class of SD (SDO) sets.

We now define several special kinds of matrices which will be used to in the canonical form of a matrix pair.
Let $\matr{E}(m), \matr{F}(m), \matr{H}(m) \in \mathbb{R}^{m \times m}$ be defined as:
\begin{align}\label{def:E-F-H}{\footnotesize
\matr{E}(m) \eqdef \begin{bmatrix}
0 &   &   &   & 1 \\
&   &   & \iddots &   \\
&   & \iddots &   &   \\
& \iddots &   &   &   \\
1 &   &   &   & 0 \\
\end{bmatrix},\ 
\matr{F}(m) \eqdef \begin{bmatrix}
0       &   &   &   & 0 \\
  &   &   & \iddots & 1       \\
  &   & \iddots & \iddots &         \\
  & \iddots & \iddots &   &         \\
0 & 1 &   &   & 0       \\
\end{bmatrix},\ 
\matr{H}(m) \eqdef \begin{bmatrix}
0       &   &   & 1  & 0 \\
  &   & \iddots  & \iddots & -1       \\
  &  \iddots & \iddots & \iddots &         \\
1  & \iddots & \iddots &   &         \\
0 & -1 &   &   & 0       \\
\end{bmatrix}.
% \matr{H}(2m) \eqdef \begin{bmatrix}
%   0      & 0       &         & \cdots  & \cdots  & 1       & 0      \\
%   0      &         &         &         & \iddots & 0       & -1     \\
%   \vdots &         &         & 1       & 0       & \iddots &  0      \\
%   \vdots &         & \iddots & 0       & -1      &         &        \\
%   0     & \iddots &         & \iddots &         &         & \vdots \\
%   1      & 0       & \iddots &         &         &         & 0      \\
%   0      & -1      &         &         & \cdots  & 0       & 0
% \end{bmatrix}.
}
\end{align}
Denote by \( \matr{J}(\lambda, m) \) the real \emph{Jordan block} associated with eigenvalue \(\lambda\) and size \( m \), {\it i.e.},
\[{\footnotesize \matr{J}(\lambda, m) \eqdef
\left[ \begin{array}{ccccc}
c & e &  &   &    \\
& c & e &  &  \\
&  & \ddots & \ddots &  \\
&  &  & \ddots & e \\
&  &  &  & c
     \end{array} \right]
\in\RR^{m\times m}.}
\]
If the eigenvalue $\lambda \in \mathbb{R}$, then $c=\lambda$ and \( e = 1 \). If the eigenvalue pair \( \lambda = a\pm bi \in \mathbb{C}\backslash\mathbb{R} \) with $b \neq 0$, then \( c = \left[ \begin{array}{cc} a & -b\\ b & a \end{array} \right] \) and \(e = \matr{I}_2\).
Define the matrices
\( \matr{G}(m), \matr{R}_k(m) \in \mathbb{R}^{m \times m}\) as
\begin{align}\label{def:matrix-G-R}
\matr{G}(m) &\eqdef \left\{
\begin{array}{ll}
\Diag{\matr{I}_{m/2}, -\matr{I}_{m/2}}, & \text{if } m \text{ is even};\\
\Diag{\matr{I}_{(m+1)/2}, -\matr{I}_{(m-1)/2}}, & \text{if } m \text{ is odd},
\end{array}
\right.\\
\matr{R}_{k}(m) &\eqdef \Diag{k^{\delta_1}, k^{\delta_2}, \ldots ,k^{\delta_m}},\label{def:matrix-G-R-k}
\end{align}
for $k\geq 1$, where \(\delta_s \eqdef (m+1)/2 - s\) for \(1 \leq s \leq m\).
Note that \( \matr{G}(m) \) is the real Jordan normal form of \( \matr{E}(m) \). 
There exists a matrix \( \matr{Q}\in\SON_{m} \) such that \( \matr{G}(m) = \matr{Q}^{\T} \matr{E}(m) \matr{Q} \).
It can be also seen that \( \matr{R}_{k}(m) \in \SL_m(\mathbb{R}) \), and
\begin{equation*}
\label{eq:R-sequence-J}{\small 
\matr{R}_{k}(m)^{-1}\matr{J}(\lambda, m)\matr{R}_{k}(m)= \begin{bmatrix}
\lambda & k^{-1} &  &   &    \\
& \lambda & k^{-1} &  &  \\
&  & \ddots & \ddots &  \\
&  &  & \ddots & k^{-1} \\
&  &  &  & \lambda
\end{bmatrix}}
\end{equation*}
for any Jordan block \( J(\lambda, m) \) with real eigenvalue \(\lambda\), and thus
\begin{equation}\label{eq:R-sequence-J-limit}
\lim_{k \to \infty} \matr{R}_k(m)^{-1} \matr{J}(\lambda, m) \matr{R}_k(m) = \lambda \matr{I}_{m}.
\end{equation}

Let \( \mathcal{C}\) be as in \eqref{set_C}. The \emph{linear span} of \( \mathcal{C} \) is denoted by
\begin{equation}\label{eq:span-C}
\operatorname{span}(\mathcal{C})\eqdef\left\{ \sum_{i=1}^{L} \alpha_{i} \matr{A}_{i}, \alpha_{i} \in \mathbb{R}, 1 \leq i \leq L \right\}.
\end{equation}
Every element in \eqref{eq:span-C} is called a {\it pencil} of \( \mathcal{C} \). 
If \( \mathcal{C} \) has a nonsingular pencil, we say \( \mathcal{C}\) is {\it nonsingular}; otherwise, it is {\it singular}. 
If \( \mathcal{C} \) has a positive definite pencil, we say \( \mathcal{C}\) is {\it positive definite}.
Let $\matr{A}, \matr{B}, \matr{S}\in\textbf{symm}(\RR^{m\times m})$, and $\matr{S}$ be nonsingular.
We define the $\matr{S}$-\emph{commutator} of $\matr{A}$ and $\matr{B}$ as 
\begin{equation*}\label{eq:commu_AB}
[\matr{A}, \matr{B}]_{\matr{S}}\eqdef\matr{S}^{-1}\matr{A} \matr{S}^{-1}\matr{B} - \matr{S}^{-1}\matr{B}\matr{S}^{-1}\matr{A},
\end{equation*}
which will be frequently used in this paper. 
In particular, we denote $[\matr{A}, \matr{B}]\eqdef [\matr{A}, \matr{B}]_{\matr{I}_{m}}$ for simplicity.

\subsection{Simultaneous diagonalization on $\SON_{m}$ and $\SL_{m}(\RR)$}
Let the set $\mathcal{C}$ be as in \eqref{set_C}. 
It is well-known that $\mathcal{C}$ is SDO if and only if $[\matr{A}_{i}, \matr{A}_{j}]=\matr{0}$ for all \( 1 \leq i\neq j \leq L \), \emph{i.e.,} they commute with each other \cite{horn2012matrix}. A generalization of this result for multiple matrices is as follows.

\begin{lemma}[{\cite[Theorem 9]{jiang2016simultaneous}}]\label{lem:SDO-community}
Let the set $\mathcal{C}$ be as in \eqref{set_C}. Then \( \mathcal{C} \cup \{ \matr{I}_m \} \) is SD if and only if $[\matr{A}_{i}, \matr{A}_{j}]=\matr{0}$ for all \( 1 \leq i\neq j \leq L \).
\end{lemma}

Moreover, the proof therein also gives a procedure to obtain the nonsingular matrix diagonalizing them. A similar result and an algorithm for Hermitian matrices are given in \cite[Theorem 3]{le2020simultaneous} and \cite[Algorithm 1]{le2020simultaneous}.
A direct consequence of the above \cref{lem:SDO-community} is that \( \mathcal{C} \cup \{ \matr{I}_m \} \) is SD if and only if $\mathcal{C}$ is SDO. Above all, we have the following result.

% Although the above lemma is about the SD property of \( \mathcal{C} \cup \{ \matr{I}_m \} \), in fact, it is SD is equivalent to \( \mathcal{C} \) is SDO. Thus, it states the condition for the SDO property of \( \mathcal{C} \) as well. To make it more clear, we present a lemma formally:
\begin{corollary} Let the set $\mathcal{C}$ be as in \eqref{set_C}. Then the following statements are equivalent:\\
(i) \( \mathcal{C} \) is SDO; \\
(ii) \( \mathcal{C} \cup \{ \matr{I}_m \} \) is SD; \\
(iii) \( \mathcal{C} \cup \{  \matr{I}_m\}  \) is SDO.
\end{corollary}

% \begin{proof}
% We prove the lemma by showing (i) \( \implies \) (ii) \( \implies \) (iii). Obviously, (iii) implies (i). So we just need to prove the rest two directions.\\
%   (i) \( \implies \)(ii): When \( \mathcal{C} \) is SDO, there exists an orthogonal matrix \( \matr{Q} \) such that \( \matr{Q}^{\top}\matr{A}_{i}\matr{Q} \) is diagonal for all \( 1 \leq i \leq L \). Since \( \matr{Q}^{\top}\matr{I}_m\matr{Q} = \matr{I}_m \), \( \mathcal{C} \cup \{ \matr{I} \}_m \) is SD via \( \matr{Q}\).\\
%   (ii) \( \implies \) (iii): Suppose \( \mathcal{C} \cup \{ \matr{I}_m \} \) is SD via \( \matr{P} \). Then \( \matr{P}^{\top}\matr{I}_m\matr{P} \) is diagonal. Assume \( \matr{P}^{\top}\matr{I}_m\matr{P} = \Diag{s_1, s_2 , \ldots , s_m} \). Since \(\matr{I}_m \succ 0\), \(s_i > 0\) for \( 1 \geq i \geq m \). Let
%   \[
%     \matr{V} = \Diag{\frac{1}{\sqrt{s_1}}, \frac{1}{\sqrt{s_2}} , \ldots , \frac{1}{\sqrt{s_m}}}.
%   \]
%   Then \( \matr{V}^{\top}\matr{P}^{\top}\matr{P}\matr{V} = \matr{I}_m \). So \( \matr{P}\matr{V} \) is an orthogonal matrix. Moreover, since \( \matr{V} \) and \( \matr{P}^{\top}\matr{A}_{i}\matr{P} \) are diagonal, \( \matr{V}^{\top}\matr{P}^{\top}\matr{A}_{i}\matr{P}\matr{V} \) is also diagonal for all \( 1 \leq i \leq L \). Thus, let \( \matr{Q} = \matr{P}\matr{V} \), then \( \matr{Q} \) is an orthogonal matrix and \( \mathcal{C} \cup \{ \matr{I}_m \} \) is SDO via \( \matr{Q} \).

% \end{proof}

We now recall an equivalent characterization of SD for a nonsingular matrix pair by the Jordan normal form of a matrix. 
%, which will be compared with the results about our new notions later.

\begin{lemma}[{\cite[Corollary 1.3]{uhlig1973SimultaneousBlockDiagonalizationa}}]\label{lem:nonsingular-pair-SD}
Let $\mathcal{C} = \{\matr{A},\matr{B} \} \subseteq \textbf{symm}(\RR^{m\times m})$, and \( \matr{A} \) be nonsingular.
Then \( \mathcal{C} \) is SD if and only if the real Jordan normal form of \( \matr{A}^{-1}\matr{B} \) is diagonal.
\end{lemma}

The above \cref{lem:nonsingular-pair-SD} can be extended to the case that $\mathcal{C}$ in \eqref{set_C} is nonsingular with $L > 2$.
%multiple matrices with a nonsingular pencil as follows.
\begin{lemma}[{\cite[Proposition 1]{wang2021new}}]\label{lem:SD-multiple-nonsingular}
Let the set $\mathcal{C}$ be as in \eqref{set_C}, and \( \matr{S} \in \operatorname{span}(\mathcal{C}) \) be nonsingular.
Then \( \mathcal{C} \) is SD if and only if the real Jordan normal form of \( \matr{S}^{-1}\matr{A}_i \) is diagonal for all \( 1 \leq i \leq L\), and
$[\matr{A_{i}}, \matr{A}_{j}]_{\matr{S}}=\matr{0}$ for all \( 1 \leq i\neq j \leq L\).
\end{lemma}

If the set $\mathcal{C}$ in \eqref{set_C} is positive definite, then we have the following result. 
%case when there is an positive definite pencil by first applying orthogonal transformation to all of the matrices simultaneously.

\begin{lemma}[{\cite[Theorem 10]{jiang2016simultaneous}}]
Let the set $\mathcal{C}$ be as in \eqref{set_C} and suppose there exist \(\alpha_{i} \in \mathbb{R}, 1 \leq i \leq L\) such that \( \matr{S}=\sum_{i=1}^L \alpha_{i} \matr{A}_{i} \succ 0 \). We assume that \(\alpha_{L}\neq 0\) without loss of generality, and choose $\matr{P}\in\GL_{m}(\RR)$ such that  \(\matr{P}^{\top}\matr{S} \matr{P} = \matr{I}_m \).
Then \( \mathcal{C} \) is SD if and only if 
$[\matr{P}^{\top} \matr{A}_i \matr{P}, \matr{P}^{\top} \matr{A}_j \matr{P}]=\matr{0}$ for all \( 1 \leq i\neq j \leq L \).
\end{lemma}

\subsection{Canonical form of a symmetric matrix pair}
Let $\matr{A},\matr{B}\in\textbf{symm}(\RR^{m\times m})$. 
The canonical form of the symmetric matrix pair $\{\matr{A},\matr{B}\}$ can be regarded as the simplest form of \(\matr{P}^{\T}\matr{A}\matr{P}\) and \( \matr{P}^{\T}\matr{B}\matr{P} \) among all \(\matr{P}\in\GL_{m}(\RR)\) \cite{uhlig1976canonical,lancaster2005CanonicalFormsHermitian, thompson1991PencilsComplexReal}.
It is an important tool by which we can reduce all the pairs of symmetric matrices to a special class, especially when the property we are studying is invariant under any congruence transformation.
We now recall the following two theorems about the canonical form, where \cref{lem:Uhlig-canonical-nonsingular} is for a nonsingular pair, and \cref{lem:lancaster-canonical-general-pair} is for a general pair.
Although \cref{lem:lancaster-canonical-general-pair} actually covers \cref{lem:Uhlig-canonical-nonsingular}, we still present both of them here for the convenience of proof later.

\begin{lemma}[{\cite[Theorem 1]{uhlig1976canonical}}]\label{lem:Uhlig-canonical-nonsingular}
Let $\matr{A},\matr{B}\in\textbf{symm}(\RR^{m\times m})$, and \( \matr{A} \) be nonsingular. 
Suppose the real Jordan normal form of \( \matr{A}^{-1} \matr{B} \) is
\begin{equation}\label{eq:Jord_nonsin-AB}
\Diag{\matr{J}(\lambda_1, m_1), \matr{J}(\lambda_2, m_2), \ldots , \matr{J}(\lambda_r, m_r), \matr{J}(\lambda_{r+1}, m_{r+1}), \ldots, \matr{J}(\lambda_p, m_p)}, 
\end{equation}
where \( \lambda_1, \lambda_2 , \ldots , \lambda_r \in \mathbb{R} \) and \( \lambda_{r+1} , \ldots , \lambda_p \in \mathbb{C}\backslash\mathbb{R}\). Then there exists \(\matr{P}\in\GL_{m}(\RR)\) such that
\begin{align*}%\label{eq:cano_form_AB_1}
{\small \matr{P}^{\T}\matr{A}\matr{P}} &{\small= \Diag{\sigma_1 \matr{E}(m_{1}), \ldots, \sigma_r\matr{E}(m_{r}), \matr{E}(m_{r+1}), \ldots \matr{E}(m_{p})},}\\
{\small\matr{P}^{\T}\matr{B}\matr{P}} &{\small= \textbf{Diag}\{\sigma_1 \matr{E}(m_{1})\matr{J}(\lambda_{1}, m_{1}), \ldots, \sigma_r\matr{E}(m_{r})\matr{J}(\lambda_{r}, m_{r})},\notag\\
&\ \ \ \ \ \ \ \ \ \   {\small\matr{E}(m_{r+1})\matr{J}(\lambda_{r+1}, m_{r+1}), \ldots \matr{E}(m_{p})\matr{J}(\lambda_{p}, m_{p})\}},%\label{eq:cano_form_AB_2}
\end{align*}
where the sign \( \sigma_s = \pm 1 \) for \( 1 \leq s \leq m \). They are unique (up to permutations) for each set of indices \( s \) that are associated with a set of identical Jordan blocks.
\end{lemma}

\begin{lemma}[{\cite[Theorem 9.2]{lancaster2005CanonicalFormsHermitian}}]\label{lem:lancaster-canonical-general-pair}
Let \( \matr{A}, \matr{B} \in \mathbf{symm}(\mathbb{R}^{m \times m}) \). Then there exists \( \matr{P} \in\GL_{m}(\RR)\) such that
\begin{equation}\label{eq:blocks_p}
\matr{P}^{\T}\matr{A}\matr{P} = \Diag{\matr{X}_1, \ldots, \matr{X}_p}\ \ \text{and}\ \ \ \matr{P}^{\T}\matr{B}\matr{P} = \Diag{\matr{Y}_1, \ldots, \matr{Y}_p} 
\end{equation}
are both block diagonal matrices with compatible block structure. 
Here, corresponding to \(p = p_1+p_2+p_3+p_4 +p_5 \) with $p_1,p_2,p_3, p_4\in\mathbb{N}_0$ and $p_5 \in \{ 0,1 \}$, the $p$ blocks in \eqref{eq:blocks_p} can be divided to five different types.
\begin{itemize}
\item The first \( p_1 \)-many blocks of \eqref{eq:blocks_p} have the form
\[
\matr{X}_s = \sigma_s\matr{E}(m_s),\ \ \matr{Y}_s = \sigma_s(\lambda_s\matr{E}(m_s)+\matr{F}(m_s)),
\]
where $ 1 \leq s \leq p_1, m_s \in \mathbb{N}$, $\sigma_s \in \{ \pm 1 \}$, and $\lambda_s \in \mathbb{R}$.
\item The next \( p_2 \)-many blocks of \eqref{eq:blocks_p} have the form
\[
\matr{X}_s = \eta_s\matr{F}(m_s),\ \ \matr{Y}_s = \eta_s\matr{E}(m_s),
\]
where $p_1+ 1 \leq s \leq p_1+p_2, m_s \in \mathbb{N}$ and $\eta_s \in \{ \pm 1 \}$.
\item The next $p_3$-many blocks of \eqref{eq:blocks_p} have the form
\[
\matr{X}_s = \matr{E}(2m_s),\ \ \matr{Y}_s = \mu_s\matr{E}(2m_s)+v_s\matr{H}(2m_s) + \Diag{\matr{E}(2m_s-2), \matr{0}_{2\times 2}},
\]
where $ p_1+p_2 + 1 \leq s \leq p_1 + p_2 + p_3, m_s \in \mathbb{N}$ and \( \mu_s, v_s \in \mathbb{R} \), $v_s \neq 0$.
\item The next $p_4$-blocks of \eqref{eq:blocks_p} have the form
\[
\matr{X}_s = \begin{pmatrix}
& & \matr{E}(m_s) \\
& 0 & \\
\matr{E}(m_s) & &
\end{pmatrix},\ \ 
\matr{Y}_s = \matr{F}(2m_s+1)
\]
where \( p_1+ p_2 + p_3 + 1 \leq s \leq p_1 + p_2 + p_3 + p_4, m_s \in \mathbb{N} \).
\item If $p_5 = 1$, then the last block of \eqref{eq:blocks_p} has the form \( \matr{X}_p = \matr{Y}_p = \matr{0}_{m_p\times m_p} \) for some $m_p \in \mathbb{N}$.
\end{itemize}
\end{lemma}

\section{Transformation based simultaneously diagonalizable matrices}\label{sec:trans_based_SD}
In this section, we will extend the notions SDO and SD using the function $\varphi_{\rm T}$ in \eqref{eq:func_T}, and propose the notions \(\mathbf{T}_{m,n}\)-SDO, \(\mathbf{T}_{m,n}\)-SD, TWSD and TWSD-B in \Cref{table-example-3-0}. For the notions \(\mathbf{T}_{m,n}\)-SDO and \( \mathbf{T}_{m,n} \)-SD, we will show that they are actually equivalent to SDO and SD respectively. 
Then we will focus on TWSD-B, and bring out several necessary and/or sufficient conditions of TWSD-B under various assumptions. 
% In contrast, our study will demonstrate that there is significant difference between the notions defined ``weakly'' and SD under many kinds of assumptions.

\subsection{Transformation based projectively simultaneously diagonalizable matrices}
In this subsection, we start from the following two definitions, and then prove that they cover no more matrices than SDO and SD.
\begin{definition}\label{def:TSDO-SD}
Let the set $\mathcal{C}$ be as in \eqref{set_C}, and $n\geq m$.\\
(i) The set $\mathcal{C}$ is \emph{transformation based projectively simultaneously diagonalizable on $\St(m,n)$} ($\textbf{T}_{m,n}$-SDO), if there exist $\matr{P}\in\St(m,n)$ and $ \mathcal{D} = \{ \matr{D}^{(i)}\}_{1 \leq i \leq L} \subseteq \mathbf{D}_{n}$, such that
$\varphi_{\rm T}(\matr{P}, \mathcal{D}) = 0$, {\it i.e.}, \( \matr{P} \matr{A}_{i} \matr{P}^{\T} \) is diagonal for all \( 1 \leq i \leq L \). 
We denote the class of \( \mathbf{T}_{m,n} \)-SDO sets by $\set{T}_{m,n}\textrm{-}\set{SDO}$. \\
(ii) The set $\mathcal{C}$ is \emph{transformation based projectively simultaneously diagonalizable on $\textbf{RSL}(m,n)$} ($\textbf{T}_{m,n}$-SD), if there exist $\matr{P}\in\textbf{RSL}(m,n)$ and $ \mathcal{D} = \{ \matr{D}^{(i)}\}_{1 \leq i \leq L} \subseteq \mathbf{D}_{n}$, such that
$\varphi_{\rm T}(\matr{P}, \mathcal{D}) = 0$, {\it i.e.}, \( \matr{P} \matr{A}_{i} \matr{P}^{\T} \) is diagonal for all \( 1 \leq i \leq L \). We denote the class of \( \mathbf{T}_{m,n} \)-SD sets by $\set{T}_{m,n}\textrm{-}\set{SD}$.
\end{definition}

It is obvious that $\set{SDO}\subseteq\set{T}_{m,n}\textrm{-}\set{SDO}$ and $\set{SD}\subseteq\set{T}_{m,n}\textrm{-}\set{SD}$ for \( n \geq m \), {\it i.e.}, \( \mathbf{T}_{m,n} \)-SDO and \(\mathbf{T}_{m,n} \)-SD are weaker than SDO and SD, respectively.
On the other hand, since the rank of \( \matr{P}\matr{A} \matr{P}^{\T} \) is always less or equal to \( m \), which is the dimension of \( \matr{A}\), it is natural to guess that the set of \( \matr{P} \matr{A} \matr{P}^{\T}\) doesn't expand no matter how large \( n \) is.  
Here, in \cref{thm:TSDO-SDO}, we will demonstrate that these two notions are essentially equivalent to SDO and SD, respectively.

\begin{theorem}\label{thm:TSDO-SDO}
Let the set $\mathcal{C}$ be as in \eqref{set_C}.\\
(i) For any \( n \geq m \), the set \( \mathcal{C} \) is \( \mathbf{T}_{m,n} \)-SDO if and only if it is SDO, \emph{i.e.}, $\set{T}_{m,n}\textrm{-}\set{SDO}=\set{SDO}$.\\
(ii) For any \( n \geq m \), the set \( \mathcal{C} \) is \( \mathbf{T}_{m,n} \)-SD if and only if it is SD, \emph{i.e.}, $\set{T}_{m,n}\textrm{-}\set{SD}=\set{SD}$.
\end{theorem}

\begin{proof}
(i) We only need to prove that, if \( \mathcal{C} \) is \( \mathbf{T}_{m,n} \)-SDO, then it is SDO. By \Cref{def:TSDO-SD}(i), there exists \( \matr{P} \in \St(m,n) \) such that \( \matr{P}\matr{A}_{i}\matr{P}^{\T} = \matr{D}^{(i)} \), where \( \matr{D}^{(i)} \in \mathbf{D}_n \) for all $1\leq i\leq L$.
Then, for all \(1\leq i \neq j \leq L\), we have
\begin{equation*} \matr{P}\matr{A}_i\matr{P}^{\T}\matr{P}\matr{A}_j\matr{P}^{\T} = \matr{D}^{(i)}\matr{D}^{(j)} = \matr{D}^{(j)}\matr{D}^{(i)} =  \matr{P}\matr{A}_j\matr{P}^{\T}\matr{P}\matr{A}_i\matr{P}^\T. 
\end{equation*}
Note that \( \matr{P}^{\T}\matr{P} = \matr{I}_m \). It follows that 
\[
\matr{P}(\matr{A}_i\matr{A}_j - \matr{A}_j\matr{A}_i)\matr{P}^{\T} = 0.
\]
Since \( \matr{P} \) has full column rank, we have that \( \matr{A}_i\matr{A}_j = \matr{A}_j\matr{A}_i\). 
Therefore, the set \( \mathcal{C}\) is SDO by \cref{lem:SDO-community}.\\
(ii) We only need to prove that, if \( \mathcal{C} \) is \( \mathbf{T}_{m,n} \)-SD, then it is SD. By \Cref{def:TSDO-SD}(ii), there exists \( \matr{P} \in \mathbf{RSL}(m,n) \) such that \( \matr{P}\matr{A}_{i}\matr{P}^{\T} = \matr{D}^{(i)} \), where \( \matr{D}^{(i)} \in \mathbf{D}_n \) for all $1\leq i\leq L$.
% Because SD implies \( \mathbf{T}_{m,n} \)-SD, we only need to prove if \( \mathcal{C} \) is \( \mathbf{T}_{m,n}-SD \), then it is SD. By definition, there exists a matrix \( \matr{P} \in \mathbf{RSL}(m,n) \), such that \( \matr{P}\matr{A}_{i}\matr{P}^{\T} = \matr{D}^{(i)} \) with \( \matr{D}^{(i)} \in \mathbf{D}_n \) for all \( 1 \leq i \leq L \).
Note that there exist \( \matr{U} \in \St(m,n) \) and \( \matr{V} \in \SL_m(\mathbb{R}) \) such that \( \matr{P} = \matr{U}\matr{V} \) by equation \eqref{eq:rela-equivalent}. 
We have \( \matr{U}\matr{V}\matr{A}_{i}\matr{V}^{\T}\matr{U}^{\T} = \matr{P}\matr{A}_{i}\matr{P}^{\T} = \matr{D}^{(i)} \) for all $1\leq i\leq L$, which means that the set  $\{\matr{V}\matr{A}_i\matr{V}^\T\}_{1\leq i\leq L}$ is \TSDO.
By part (i), this set is SDO. Thus, there exists a matrix $\matr{Q} \in \SON_m$ such that $\matr{Q}\matr{V}\matr{A}_{i}\matr{V}^{\T}\matr{Q}^{\T}$ is diagonal for all \( 1 \leq i \leq L \). Note that \( \matr{Q}\matr{V} \in \SL_m(\mathbb{R})\). It follows that the set \( \mathcal{C}\) is SD. The proof is complete.
\end{proof}

\subsection{Transformation based weakly simultaneously diagonalizable matrices}
In this subsection, we define two notions TWSD and TWSD-B, which are both weaker than SD. The characterizations of these two new notions will be given in the later subsections.
\begin{definition}\label{def:twsd}
Let the set $\mathcal{C}$ be as in \eqref{set_C}.\\
(i) The set $\mathcal{C}$ is \emph{transformation based weakly simultaneously diagonalizable} (TWSD), if there exist $\matr{P}_k \in \SL_{m}(\RR)$ and $ \mathcal{D}_{k} = \{ \matr{D}_k^{(i)}\}_{1 \leq i \leq L} \subseteq \mathbf{D}_{m}$ for $k\geq 1$, such that
\begin{align*}
\lim_{k\rightarrow\infty} \varphi_{\rm T}(\matr{P}_k, \mathcal{D}_{k}) = 0.
\end{align*}
We denote the class of TWSD sets by $\set{TWSD}$.\\
(ii) The set $\mathcal{C}$ is \emph{bounded TWSD} (TWSD-B), if there exist $\matr{P}_k \in \SL_{m}(\RR)$ and $\mathcal{D}_{k} = \{ \matr{D}_k^{(i)}\}_{1 \leq i \leq L} \subseteq \mathbf{D}_{m}$ for $k\geq 1$, such that
\begin{align*}
\lim_{k\rightarrow\infty} \varphi_{\rm T}(\matr{P}_k, \mathcal{D}_k) = 0,
\end{align*}
and $\|\matr{D}_k^{(i)}\|$ is uniformly bounded.
We denote the class of TWSD-B sets by $\set{TWSD}\textrm{-}\set{B}$.\\
\end{definition}

\begin{remark}\label{defini:TWSD-B}
(i) Although the definition of \( \varphi_T \) in  \eqref{eq:func_T} is the sum of \( \|\matr{P} \matr{A}_{i} \matr{P}^{\top} - \matr{D}^{(i)} \|^2 \), for convenience, we will also sometimes use \( \| \matr{P}^{\top} \matr{A}_{i} \matr{P} - \matr{D}^{(i)} \|^{2} \) to replace it in this paper when we talk about TWSD and TWSD-B notions.
In these two cases, since \( \matr{P}\in\SL_m(\RR) \) is square, they are equivalent to each other. \\
(ii) It is clear that the set  $\mathcal{C}$ is TWSD-B, if and only if there exists $\matr{P}_k \subseteq \SL_{m}(\RR)$, such that
\begin{align*}\label{eq:offdiag_zero-limits}
\lim_{k\rightarrow\infty} \|\textbf{offdiag}(  \matr{P}_k^{\T}\matr{A}_i \matr{P}_k)\|= 0
\end{align*}
for all $1\leq i\leq L$, and there exists $M>0$ such that the diagonal elements satisfy
\begin{align*}
\|\textbf{diag}(\matr{P}_k^\T \matr{A}_{i} \matr{P}_k)\| \leq M
\end{align*}
for all $1\leq i\leq L$ and $k\geq 1$, if and only if there exist $\{\matr{P}_k\}_{k \geq 1} \subseteq \SL_{m}(\RR)$ and \( \{\matr{D}^{(i)}\}_{1\leq i\leq L} \subseteq \mathbf{D}_m \)  such that
\[
\lim_{k \to \infty}\matr{P}_k^{\T}\matr{A}_{i}\matr{P}_k = \matr{D}^{(i)}
\]
for all \( 1 \leq i \leq L \).\\
(iii) In \Cref{def:twsd}, if the constraint \( \matr{P}_k \in \SL_m(\mathbb{R}) \) is changed to that \( \det(\matr{P}_k)=c \), where $c$ is a nonzero constant, then TWSD and TWSD-B notions remain the same. 
It is the same case with the above remarks. 
% , since we can find a constant diagonal matrix \( \matr{N} \) such that \( \matr{P} \matr{N} \in \SL_m(\mathbb{R}) \) and replace \( \matr{P} \) by it. After that, others will remain the same.
\end{remark}

% \begin{remark}
% Let the set $\mathcal{C}$ be as in \cref{defini:TWSD-B}. Then $\mathcal{C}$ is TWSD-B if and only if there exists a sequence $\matr{P}_k\in\mathbf{SL}(n)$ with $k\geq 1$, and $L$ diagonal matrices $\matr{D}_{i}\in\mathbf{D}(n)$ such that
% \begin{align}
% \lim_{k\rightarrow\infty} \|\matr{P}_k^{\T}\matr{A}_i \matr{P}_k - \matr{D}_{i}\|= 0,
% \end{align}
% for all $1\leq i\leq L$.
% In fact, ...
% \end{remark}

It is obvious by \cref{def:twsd} that \(\set{SD}\subseteq   \set{TWSD}\textrm{-}\set{B}\subseteq \set{TWSD}\). 
It will be shown in \cref{exa:TWSD-not-TWSD-B} that \(\set{TWSD}\textrm{-}\set{B}\subsetneqq \set{TWSD}\).
We now present two examples to show that \(\set{SD}\subsetneqq   \set{TWSD}\textrm{-}\set{B}\), and $\set{TWSD}$ doesn't include all the sets of symmetric matrices.

\begin{example}[A set which is TWSD-B, but not SD]
Let $\matr{A}=\left[ \begin{array}{cc} 0 & 1 \\ 1 & 0 \end{array} \right] $ and $\matr{B}=\left[ \begin{array}{cc} 1 & 0 \\ 0 & 0 \end{array} \right]$.
Then, by setting $\matr{P}_k=\left[ \begin{array}{cc} \frac{1}{k} & \frac{1}{2k} \\ -k & \frac{k}{2} \end{array} \right]$ for $k\geq 1$, we see that
\[
\matr{P}_k^{\T}\matr{A}\matr{P}_k = \left[ \begin{array}{cc}-2 & 0\\0 & \frac{1}{2}\end{array} \right],\
\matr{P}_k^{\T}\matr{B}\matr{P}_k = \left[ \begin{array}{cc} \frac{1}{k^{2}} & \frac{1}{2 k^{2}}\\\frac{1}{2 k^{2}} & \frac{1}{4 k^{2}} \end{array} \right].
\]
Note that \( \matr{P}_k^{\T}\matr{A} \matr{P}_k\) and \(  \matr{P}_k^{\T}\matr{B} \matr{P}_k\) both converge to diagonal matrices when $k\rightarrow\infty$.
The set \( \{ \matr{A}, \matr{B} \} \) is TWSD-B by \Cref{defini:TWSD-B}(ii). However, since \( \matr{A}^{-1}\matr{B} = \left[ \begin{array}{cc} 0 & 0 \\ 1 & 0 \end{array} \right]  \), its Jordan normal form is itself, and thus this set is not SD by \cref{lem:nonsingular-pair-SD}.

\end{example}

% \begin{remark}
% % In \cref{defini:TWSD-B}, a natural question may be whether the condition \eqref{eq:diag_finite} is redundant. For the set $\mathcal{C}$, if only the condition \eqref{eq:offdiag_zero} is satisfied, we call it to be {\it weakly simultaneously diagonalizable}\/ (WSD). 
% It will be shown in \cref{exa:TWSD-not-TWSD-B} that a set \( \mathcal{C} \) may be in $\set{TWSD}$, but not in $\set{TWSD}\textrm{-}\set{B}$.
% \end{remark}

\begin{example}[A set which is not TWSD]\label{exa:not-TWSD}
Let $\matr{A}=\left[ \begin{array}{cc} 0 & 1 \\ 1 & 0 \end{array} \right] $ and $\matr{B}=\left[ \begin{array}{cc} 1 & 0 \\ 0 & -1 \end{array} \right] $.
Then the set \( \{ \matr{A}, \matr{B} \} \) is not TWSD. 
We now prove it by contradiction. 
Assume that this set is TWSD. Then there exist a sequence $\matr{P}_k=\left[ \begin{array}{cc} p^{(k)}_{11} & p^{(k)}_{12} \\ p^{(k)}_{21} & p^{(k)}_{22} \end{array} \right]$ for $k\geq 1$, such that
\begin{align}
&p_{11}^{(k)}p_{22}^{(k)} - p_{12}^{(k)} p_{21}^{(k)} = 1,\, \forall k, \label{eq:notTWSD-1} \\
&\lim_{k\rightarrow \infty}  \left( p_{11}^{(k)}p_{22}^{(k)} + p_{12}^{(k)} p_{21}^{(k)} \right) = 0 , \label{eq:notTWSD-2}\\
&\lim_{k\rightarrow \infty}  \left( p_{11}^{(k)}p_{12}^{(k)} - p_{21}^{(k)} p_{22}^{(k)} \right) = 0.\label{eq:notTWSD-3}
\end{align}
It follows by equations \eqref{eq:notTWSD-1} and \eqref{eq:notTWSD-2} that
\[
\lim_{k\rightarrow \infty}   p_{11}^{(k)}p_{22}^{(k)}  = 0.5, \quad \lim_{k\rightarrow \infty}   p_{12}^{(k)}p_{21}^{(k)}  = -0.5.
\]
We now consider the following cases:
\begin{description}
\item[Case 1]
$p_{11}^{(k)}>0$, $p_{22}^{(k)}>0$, $p_{12}^{(k)}>0$ and $p_{21}^{(k)}<0$ for infinitely many $k$'s (the other case of $p_{12}^{(k)}<0$ and $p_{21}^{(k)}>0$ is similar).
In this case, there exists $K_1>0$ such that 
\[p_{11}^{(k)} > 0.25 / p_{22}^{(k)}, \quad p_{21}^{(k)} < -0.25 / p_{12}^{(k)}\]
for infinitely many indices $k>K_1$.
Thus,
\[p_{11}^{(k)}p_{12}^{(k)} - p_{21}^{(k)} p_{22}^{(k)} > 0.25 \times \frac{p_{12}^{(k)}}{p_{22}^{(k)}} + 0.25 \times \frac{p_{22}^{(k)}}{p_{12}^{(k)}} \geq 0.5,\]
which is in contradiction with equation \eqref{eq:notTWSD-3}.

\item[Case 2]
$p_{11}^{(k)}<0$, $p_{22}^{(k)}<0$, $p_{12}^{(k)}>0$ and $p_{21}^{(k)}<0$ for infinitely many $k$'s (the other case of $p_{12}^{(k)}<0$ and $p_{21}^{(k)}>0$ is similar).
In this case, there exists $K_2$ such that 
\[p_{11}^{(k)} < 0.25 / p_{22}^{(k)}, \quad -p_{21}^{(k)} > 0.25 / p_{12}^{(k)}\]
for infinitely many indices $k>K_2$.
Thus,
\[  p_{11}^{(k)}p_{12}^{(k)} - p_{21}^{(k)} p_{22}^{(k)} < 0.25 \times \frac{p_{12}^{(k)}}{p_{22}^{(k)}} + 0.25 \times \frac{p_{22}^{(k)}}{p_{12}^{(k)}}  \leq -0.5,\]
which is also in contradiction with equation \eqref{eq:notTWSD-3}.
\end{description}
\end{example}

\subsection{Characterizations of TWSD-B for a nonsingular pair}

In this subsection, we mainly prove \cref{thm:better-decomposition} and \cref{thm:TWSD-B-nonsingular-pair}, which can be seen as extensions of \cref{lem:Uhlig-canonical-nonsingular} and \cref{lem:nonsingular-pair-SD}, respectively.
The proofs are both postponed to \Cref{proofs-1}. 

\begin{lemma}\label{thm:better-decomposition}
Let $\matr{A},\matr{B}\in\textbf{symm}(\RR^{m\times m})$, and \( \matr{A} \) be nonsingular. 
Let \( \matr{A}^{-1} \matr{B} \) have the real Jordan normal form \eqref{eq:Jord_nonsin-AB}.  
Then there exists a sequence \( \{  \matr{P}_{k}\}_{k\geq 1}\subseteq\GL_{m}(\RR) \) with constant determinant, such that
\begin{align*}
{\small\matr{P}_k^{\T}\matr{A}\matr{P}_k} &{\small = \Diag{\sigma_1\matr{G}(m_1), \ldots,  \sigma_r\matr{G}(m_r), \matr{E}(m_{r+1}), \ldots, \matr{E}(m_{p})}, \ \forall k,} \\
{\small\lim_{k \to \infty} \matr{P}_k^{\T}\matr{B}\matr{P}_k} &{\small= \Diag{\lambda_1\sigma_1\matr{G}(m_1), \ldots,  \lambda_r\sigma_r\matr{G}(m_r), \matr{E}(m_{r+1})\matr{J}(\lambda_{r+1}, m_{r+1}), \ldots, \matr{E}(m_{p})\matr{J}(\lambda_p, m_p)},}
\end{align*}
where \( \sigma_i = \pm 1 \) for \( 1 \leq s \leq r \) are as in \cref{lem:Uhlig-canonical-nonsingular}.
\end{lemma}

\begin{theorem}\label{thm:TWSD-B-nonsingular-pair}
Let 
$\matr{A}, \matr{B}\in\textbf{symm}(\RR^{m\times m})$, and \( \matr{A} \) be nonsingular. 
Then \( \matr{A} \) and \( \matr{B} \) are TWSD-B if and only if \( \matr{A}^{-1}\matr{B} \)  has only real eigenvalues.
\end{theorem}

By \cref{thm:TWSD-B-nonsingular-pair} and the proof of \cref{thm:better-decomposition}, we are now able to propose \cref{alg:two-matrices-nonsingular} to test whether two symmetric matrices are TWSD-B or not, if at least one of them is nonsingular.
The congruent matrices $\matr{P}_k$ can be obtained as well.

\begin{algorithm}[h]
\caption{Check whether matrices $\matr{A}, \matr{B}$ are TWSD-B, where \( \matr{A} \) is nonsingular.}
\label{alg:two-matrices-nonsingular}
\begin{algorithmic}
\REQUIRE A pair of matrices $\matr{A}, \matr{B}\in\textbf{symm}(\RR^{m\times m})$, where \( \matr{A} \) is nonsingular. 
\IF{$\matr{A}^{-1}\matr{B}$ has an eigenvalue which is not real}
\STATE \textbf{Return} the set \(\{ \matr{A}, \matr{B} \}\) is not TWSD-B.
\ELSIF{\( \matr{A}^{-1}\matr{B} \) has only real eigenvalues}
\STATE Find \( \matr{\tilde{P}} \) such that \( \matr{\tilde{P}}^{\T}\matr{A}\matr{\tilde{P}} \) and \( \matr{\tilde{P}}^{\T}\matr{B}\matr{\tilde{P}} \) are in the canonical form as in \cref{lem:Uhlig-canonical-nonsingular}.
\STATE Define \( \matr{P}_{k}^{(i)}\) as in the proof of \cref{thm:better-decomposition} for each Jordan block in the Jordan normal form of \( \matr{A}^{-1} \matr{B} \).
\STATE \textbf{Return} the congruent matrices \( \matr{P}_k = \matr{\tilde{P}} \Diag{\matr{P}_k^{(1)} , \ldots , \matr{P}_k^{(r)}} \).
\ENDIF
\end{algorithmic}
\end{algorithm}

\subsection{Characterizations of TWSD-B for a singular pair}
In this subsection, we show that all singular pairs are TWSD-B in \Cref{thm:TWSD-B-singular-pair}.
The proof is postponed to \Cref{proofs-1}. 
\begin{theorem}\label{thm:TWSD-B-singular-pair}
Let $\{\matr{A}, \matr{B}\} \subseteq  \textbf{symm}(\RR^{m\times m})$ be a singular pair. Then it is TWSD-B.
\end{theorem}

Based on \cref{thm:TWSD-B-nonsingular-pair}, \cref{thm:TWSD-B-singular-pair} and their proofs, we now propose \cref{alg:two_matrix} to check whether a general pair of matrices \(\{ \matr{A},  \matr{B} \}\) is TWSD-B or not.
If it is TWSD-B, the congruent matrices $\matr{P}_{k}$ can be calculated as well.

\begin{algorithm}[h]
\caption{Check whether general pair of matrices \(\{ \matr{A},  \matr{B} \}\) is TWSD-B.}
\label{alg:two_matrix}
\begin{algorithmic}
\REQUIRE A general pair of matrices $\{\matr{A}, \matr{B}\} \subseteq  \textbf{symm}(\RR^{m\times m})$. 
\STATE Find a congruent matrix \( \matr{\tilde{P}} \) such that \( \matr{\tilde{P}}^{\T}\matr{A}\matr{\tilde{P}} \) and \( \matr{\tilde{P}}^{\T}\matr{B}\matr{\tilde{P}} \) are in the form of \cref{lem:lancaster-canonical-general-pair}.
\IF {\(p_4 + p_5 > 0\)}
\STATE \textbf{Return} \( \{ \matr{A}, \matr{B} \} \) is TWSD-B and the congruent matrices \( \matr{P_k} \) are as in the proof of \cref{thm:TWSD-B-singular-pair}.
\ELSIF{\(p_3 = 0\)}
\STATE For \( 1 \leq i \leq r \), define matrices \( \matr{P}_k^{(i)} = \matr{R}_k(m_i) \matr{Q}_i \) for the corresponding block as in the proof of \cref{lem:E-F-simul}.
\STATE \textbf{Return} \( \{ \matr{A}, \matr{B} \} \) is TWSD-B and the congruent matrices \( \matr{P_k} = \matr{\tilde{P}} \Diag{\matr{P}_k^{(1)} , \ldots , \matr{P}_k^{(p)}} \).
\ELSE
\STATE \textbf{Return} \( \{ \matr{A}, \matr{B} \} \) is not TWSD-B.
\ENDIF
\end{algorithmic}
\end{algorithm}

\subsection{Characterizations of TWSD-B for a set of finitely-many matrices}
In this subsection, we begin to consider the case when the set \( \mathcal{C} \) in \eqref{set_C} contains not only two matrices. We first show that, if the set \( \mathcal{C} \) has a positive definite pencil, then TWSD-B is equivalent to SD. 
%Then we give a necessary condition and a sufficient condition for general case.
\begin{theorem}\label{thm:PD-equivalent}
Let the set $\mathcal{C}$ be as in \eqref{set_C}.
If $\mathcal{C}$ is positive definite, then it is TWSD-B if and only if it is SD. 
\end{theorem}

\begin{proof}
We only need to prove that, if $\mathcal{C}$ is TWSD-B, then it is SD.
Without loss of generality, we let \( \matr{A}_1 \) be positive definite.
Then there exists a matrix \( \matr{P}\in\GL_{m}(\RR)\), such that \( \matr{P}^{\T}\matr{A}_1\matr{P} = \matr{I}_{m} \).
Let \( \matr{\bar{A}}_{i} = \matr{P}^{\T}\matr{A}_{i}\matr{P}\) for \(1 \leq i \leq L \). It is clear that the set  \(\{\matr{\bar{A}}_i\}_{1\leq i\leq L}\) is also TWSD-B.
Therefore, there exists a sequence \( \{\matr{V}_k\}_{k \geq 1}\subseteq\SL_{m}(\RR)\) such that \( \matr{V}_k^{\T}\matr{\bar{A}}_{i}\matr{V}_k\) converges to a diagonal matrix for all \( 1 \leq i \leq L \).
In particular, since \(\matr{\bar{A}}_1 = \matr{I}_{m}\), we see that \( \matr{V}_k^{\T}\matr{V}_k = \matr{V}_{k}^{\T} \matr{\bar{A}}_{1} \matr{V}_k \) converges to a diagonal matrix, and thus is bounded.
It follows that $\|\matr{V}_k\|^2=\textrm{tr}(\matr{V}_k^{\T}\matr{V}_k)$ is also bounded, and thus the sequence $\{\matr{V}_k\}_{k\geq 1}$ has a convergent subsequence.
Denote the convergent subsequence by $\{\matr{V}_{k_p}\}_{p\geq 1}$ and its limit by \( \matr{V} \). Then $\det (\matr{V}) =1$, since the determinant is a continuous function of matrix.
Note that \( \matr{V}^{\T} \matr{\bar{A}}_{i} \matr{V} = \lim_{p \to \infty} \matr{V}_{k_p}^{\T} \matr{\bar{A}}_{i} \matr{V}_{k_p} \) is diagonal for $1\leq i\leq L$.
So the set \(\{\matr{\bar{A}}_i\}_{1\leq i\leq L}\) is SD.
Therefore, the set \( \mathcal{C} \) is also SD, and the proof is complete. 
\end{proof}

Now we give a necessary condition for TWSD-B. 
The following theorem can be regarded as an extension of the necessary condition in \cref{lem:SD-multiple-nonsingular}.

\begin{theorem}\label{thm:TWSD-B-multiple-necessary}
Let the set $\mathcal{C}$ be as in \eqref{set_C},
and \(\matr{S} \in \operatorname{span}(\mathcal{C}) \) be nonsingular.  
If \(\mathcal{C}\) is TWSD-B, then \( \matr{S}^{-1}\matr{A}_{i} \) has only real eigenvalues for all \( 1 \leq i \leq L \), and $[\matr{A}_i, \matr{A}_j]_{\matr{S}}$ is nilpotent for all \( 1\leq i\neq j\leq L \).
\end{theorem}

\begin{proof}
Since \(
\mathcal{C}\) is TWSD-B, there exists a sequence \( \{\matr{P}_k\}_{k \geq 1}\subseteq\SL_{m}(\RR)\) such that \( \matr{P}_k^{\T}\matr{S}\matr{P}_k\) and \(\matr{P}_k^{\T}\matr{A}_{i}\matr{P}_k  \) all converge to diagonal matrices for all \( 1 \leq i \leq L \). Then \((\matr{P}_k^{\T}\matr{S}\matr{P}_{k})^{-1}\) converges to a diagonal matrix by \cref{lem:bounded-inverse}(ii), and thus \( \matr{P}_k^{-1}\matr{S}^{-1}\matr{A}_{i}\matr{P}_k = (\matr{P}_k^{\T}\matr{S}\matr{P}_{k})^{-1}(\matr{P}_k^{\T}\matr{A}_i\matr{P}_k)\) also converges to a diagonal matrix. It follows by \cref{lem:weak-Jordan-form} that \( \matr{S}^{-1}\matr{A}_{i} \) has only real eigenvalues for all \( 1 \leq i \leq L \).

Note that \( \matr{P}_k^{-1}\matr{S}^{-1}\matr{A}_i\matr{P}_k\) and \(\matr{P}_k^{-1}\matr{S}^{-1}\matr{A}_j\matr{P}_k \) both converge to diagonal matrices for \( 1\leq i\neq j\leq L \).
We have that 
\[
\lim_{k \to \infty}(\matr{P}_k^{-1}\matr{S}^{-1}\matr{A}_i\matr{P}_k) (\matr{P}_k^{-1}\matr{S}^{-1}\matr{A}_j\matr{P}_k) =  \lim_{k \to \infty}(\matr{P}_k^{-1}\matr{S}^{-1}\matr{A}_j\matr{P}_k) (\matr{P}_k^{-1}\matr{S}^{-1}\matr{A}_i\matr{P}_k),
\]
which implies that
\[
\lim_{k \to \infty}\matr{P}_k^{-1}(\matr{S}^{-1}\matr{A}_i\matr{S}^{-1}\matr{A}_j - \matr{S}^{-1}\matr{A}_j\matr{S}^{-1}\matr{A}_i)\matr{P}_k = \lim_{k \to \infty} \matr{P}_k^{-1} [\matr{A}_i, \matr{A}_j]_{\matr{S}} \matr{P} = 0.
\]
Therefore, the eigenvalues of $[\matr{A}_i, \matr{A}_j]_{\matr{S}}$ are all \( 0 \) by the proof of \cref{lem:weak-Jordan-form}. 
The proof is complete.
\end{proof}

Now we give an example to show that the condition in \cref{thm:TWSD-B-multiple-necessary} is not sufficient.
\begin{example}\label{exa:TWSD-B-multiple-necessary-not-sufficient}
Consider \( \mathcal{C} = \{ \matr{A}_1, \matr{A}_2, \matr{A}_3 \} \) with
\[
\matr{A}_1 = \matr{I}_2, \quad \matr{A}_2 = \begin{bmatrix}
  1 & 1\\
  0 & 1
\end{bmatrix}, \quad
\matr{A}_3 = \begin{bmatrix}
  0 & 0\\
  0 & 1
\end{bmatrix}.
\]
Let \(\matr{S}=\matr{A}_1\) be nonsingular. 
Then
\[ [ \matr{A}_2, \matr{A}_3 ]_{\matr{S}}= \matr{A}_2 \matr{A}_3 - \matr{A}_3 \matr{A}_2 =
\begin{bmatrix}
  0 & 1\\
  0 & 1
\end{bmatrix}-
\begin{bmatrix}
  0 & 0\\
  0 & 1
\end{bmatrix}=\begin{bmatrix}
  0 & 1\\
  0 & 0
\end{bmatrix}
\]
is nilpotent. 
%So \( \mathcal{C} \) satisfies the condition in \cref{thm:TWSD-B-multiple-necessary}. 
However, since \( \matr{A}_2 \) and \( \matr{A}_3 \) don't commute with each other, \( \mathcal{C} \) is not SD by \cref{lem:SDO-community}. 
Thus, \( \mathcal{C} \) is not TWSD-B by \cref{thm:PD-equivalent}.
\end{example}

Finally, similar as the sufficient condition in \cref{lem:SD-multiple-nonsingular}, we close this subsection by giving two sufficient conditions for TWSD-B.
The proofs are both postponed to \Cref{proofs-1}. 

\begin{theorem}\label{thm:TWSD-B-multiple-sufficient-condition}
Let the set $\mathcal{C}$ be as in \eqref{set_C},
and \(\matr{S} \in \operatorname{span}(\mathcal{C}) \) be nonsingular. 
If \( \matr{S}^{-1} \matr{A}_{i}\) has only real eigenvalues for all $1 \leq i \leq L$,  \( [\matr{A}_i, \matr{A}_j]_{\matr{S}} = \matr{0} \) for all \( 1\leq i\neq j\leq L \) and there exists \( i \) such that there does not exist two Jordan blocks in the Jordan normal form of \( \matr{S}^{-1} \matr{A}_i \) with the same eigenvalue and size, then \( \mathcal{C} \) is TWSD-B.
\end{theorem}

\begin{theorem}\label{thm:TWSD-B-three-multiple-sufficient-condition}
Let $\mathcal{C} = \{ \matr{S}, \matr{A}_1, \matr{A}_2 \} $ and \(\matr{S} \) be nonsingular.
If \( \matr{S}^{-1} \matr{A}_{i}\) has only real eigenvalues for all $1 \leq i \leq 2$, and \( [\matr{A}_1, \matr{A}_2]_{\matr{S}} = \matr{0} \), then \( \mathcal{C} \) is TWSD-B.
\end{theorem}

\section{Characterizations of TWSD}\label{sec:TWSD}

In this section, for the notion TWSD proposed in \Cref{sec:trans_based_SD}, we will prove several sufficient conditions for TWSD under different assumptions, as well as its relationship with TWSD-B.

\subsection{Sufficient conditions of TWSD}

It is clear that $\set{TWSD}\textrm{-}\set{B}\subseteq\set{TWSD}$ by  \cref{def:twsd}.
We now first present an example to show that they are not equivalent in general.
\begin{example}[A set which is TWSD, but not TWSD-B]\label{exa:TWSD-not-TWSD-B}
Let
\[
\matr{A} = \begin{bmatrix}
1 & 0 & 0 \\
0 & 0 & 1 \\
0 & 1 & 0
\end{bmatrix},\
\matr{B} = \begin{bmatrix}
1 & 0 & 0 \\
0 & 1 & 0 \\
0 & 0 & -1
\end{bmatrix}.
\]
Then \( \matr{A}^{-1}\matr{B} = \begin{bmatrix}
1 & 0 & 0 \\
0 & 0 & -1 \\
0 & 1 & 0
\end{bmatrix}
\), whose eigenvalues are \( 1, \pm i\). 
Therefore, the set \(\{\matr{A}, \matr{B}\}\) is not TWSD-B by \cref{thm:TWSD-B-nonsingular-pair}.
However, if \( \matr{P}_k = \Diag{k^2, 1/k, 1/k}  \) for $k\geq 1$, we have
\[
\matr{P}_k^{\T}\matr{A}\matr{P}_k = \begin{bmatrix}
k^4& 0 & 0 \\
0 & 0 &1/k^{2}\\
0 &1/k^{2}& 0
\end{bmatrix},\
\matr{P}_k^{\T}\matr{B} \matr{P}_{k} = \begin{bmatrix}
k^4& 0 & 0 \\
0 &1/k^{2}& 0 \\
0 & 0 &-1/k^{2}
\end{bmatrix}.
\]
It follows that the set \(\{\matr{A}, \matr{B}\}\) is TWSD.
\end{example}

It is well-known that a sufficient condition for two symmetric matrices to be SD is that they have a positive definite pencil.
This is also a corollary of \cref{lem:nonsingular-pair-SD}. For TWSD, we now have a similar result as shown below.
The proof is postponed to 
\Cref{proofs-2}. 

\begin{theorem}\label{thm:PSD-sufficient}
Let $\matr{A}, \matr{B}\in\textbf{symm}(\RR^{m\times m})$. If there exist $\alpha,\beta\in\RR$ (not both zero) such that
\begin{equation}\label{eq:positive_semid_TWSD}
\alpha \matr{A} + \beta \matr{B} \succeq 0, 
\end{equation}
then the set $\{\matr{A}, \matr{B}\}$ is TWSD.
\end{theorem}

\begin{example}
Let $a,b,c\in\RR$ satisfy $a\not=0$ and $b\not=c$.
Let
\[
\matr{A}=\left[ \begin{array}{cc} a & b \\ b & 0 \end{array} \right] \mbox{ and } \matr{B}=\left[ \begin{array}{cc} a & c \\ c & 0 \end{array} \right].
\]
Note that
\begin{align*}
\frac{-\mbox{sign}(a) c}{b-c} \cdot \left[ \begin{array}{cc} a & b \\ b & 0 \end{array} \right]  + \frac{\mbox{sign}(a) b}{b-c} \cdot \left[ \begin{array}{cc} a & c \\ c & 0 \end{array} \right] 
=\left[ \begin{array}{cc} |a| & 0 \\ 0 & 0 \end{array} \right] \succeq \matr{0}. 
\end{align*}
It follows by \cref{thm:PSD-sufficient} that $\matr{A}$ and $\matr{B}$ are TWSD.
As a special example, we see that 
\[
\matr{A}=\left[ \begin{array}{cc} 1 & -1 \\ -1 & 0 \end{array} \right] \mbox{ and } \matr{B}=\left[ \begin{array}{cc} 1 & 1 \\ 1 & 0 \end{array} \right]
\]
are TWSD. 
\end{example}
We now present an example to show that the condition \eqref{eq:positive_semid_TWSD} in \cref{thm:PSD-sufficient} is not necessary.
\begin{example}
Let
\(\matr{A} = \begin{bmatrix}
-1 & 0 & 0\\
0 & 1 & 1 \\
0 & 1 & -1
\end{bmatrix}
\mbox{ and  }
\matr{B} = \begin{bmatrix}
-1 & 0 & 0 \\
0 & 1 & 0 \\
0 & 0 & 0
\end{bmatrix}.
\)
%\begin{itemize}
%\item %\(\matr{A}\) and \(\matr{B}\) are TWSD. 
Note that
\[
\matr{A}^{-1}\matr{B} = \begin{bmatrix}
-1 & 0 & 0 \\
0 & 1/2 & 1/2 \\
0 & 1/2 & -1/2
\end{bmatrix}
\begin{bmatrix}
-1 & 0 & 0 \\
0 & 1 & 0 \\
0 & 0 & 0
\end{bmatrix}
=
\begin{bmatrix}
1 & 0 & 0 \\
0 & 1/2 & 0 \\
0 & 1/2 & 0
\end{bmatrix},\]
whose eigenvalues are \( 1, 1/2\) and \( 0 \). By \cref{thm:TWSD-B-nonsingular-pair}, we see that \(\matr{A}\) and \(\matr{B}\) are TWSD-B, and thus TWSD.
%\item %The pencil of \(\matr{A}\) and \(\matr{B}\) contains no positive semidefinite matrix. We prove it by contradiction. 
Assume that there exist $\alpha,\beta\in\RR$ (not both zero) such that 
\begin{equation}\label{eq:al_bet_psd}
\alpha\matr{A} + \beta\matr{B} =
\begin{bmatrix}
-(\alpha + \beta) & 0 & 0 \\
0 & \alpha + \beta & \alpha \\
0 & \alpha & -\alpha
\end{bmatrix}\succeq 0.
\end{equation}
Then all of its diagonal elements are non-negative, \emph{i.e.,} $-(\alpha + \beta)\geq 0, \alpha + \beta\geq 0, -\alpha\geq 0$. 
It follows that $\alpha + \beta= 0$ and $\alpha\leq 0$. 
Since $\alpha$ and $\beta$ are not both zero, we see that $\alpha< 0$ and $\beta>0$.
Then the determinant of the submatrix \( \begin{bmatrix}
\alpha + \beta & \alpha \\
\alpha & -\alpha
\end{bmatrix} \)
of $\alpha\matr{A} + \beta\matr{B}$ in \eqref{eq:al_bet_psd} is \( - \alpha^2 < 0 \), which contradicts the assumption \eqref{eq:al_bet_psd}. 
%\end{itemize}
\end{example}

Now we give a sufficient condition for multiple matrices to be TWSD, if they are all block diagonal matrices.

%Now we give a sufficient condition for TWSD which is related to TWSD-B for block diagonal matrices.
\begin{theorem}\label{thm:TWSD-block-TWSD-B-suffic}
Let the set $\mathcal{C}$ be as in \eqref{set_C}, where 
%Given a set of block diagonal matrices \( \mathcal{C} = \{ \matr{A}_1 , \ldots , \matr{A_L} \} \) with  
\(\matr{A}_{i} = \Diag{\bar{\matr{A}}_{i}, \tilde{\matr{A}}_{i}} \) for all \( 1 \leq i \leq L \), with \( \bar{\matr{A}}_{i} \in \mathbb{R}^{r \times r}\) and \(\tilde{\matr{A}}_{i} \in \mathbb{R}^{(m-r) \times (m-r)}  \). If the set \( \{ \bar{\matr{A}}_i\}_{1\leq i\leq L}\subseteq \textbf{symm}(\RR^{r\times r}) \) is TWSD-B, then $\mathcal{C}$ is TWSD.
\end{theorem}

\begin{proof}
Since the set \( \{ \bar{\matr{A}}_i\}_{1\leq i\leq L} \) is TWSD-B, there exists a sequence \(\{\matr{P}_k\}_{k\geq 1}\subseteq\SL_{r}(\RR)\) such that all the off-diagonal elements of \( \matr{P}_k^{\T}\bar{\matr{A}}_{i}\matr{P}_k \) converge to \( 0 \).
Denote \(a_{p,q}^{({i},k)}=(\matr{P}_k^{\T}\bar{\matr{A}}_{i}\matr{P}_k)_{p,q}\).
Since there are finitely many off-diagonal elements, we can find a sequence $\{\epsilon_k\}_{k\geq 1}\subseteq\RR$ such that \( \lim_{k \to \infty} \epsilon_k = 0 \) and \( \lim_{k \to \infty}a_{p,q}^{({i},k)}/\epsilon_k^2 = 0 \). 
Let \( \matr{V}_k = \Diag{\matr{P}_k /\epsilon_k,\epsilon_k^{rk/(m-r)}\matr{I}_{m-r}} \).
It can be seen that the off-diagonal elements of $\matr{V}_k^{\T}\matr{A}_{i}\matr{V}_k$ are \(  a_{p,q}^{({i},k)}/\epsilon_k^2 \) and \( \epsilon_k^{2rk/(m-r)} a_{p,q}^{({i},k)} \), which all converge to \( 0 \). 
Therefore, the set \( \mathcal{C} \) is TWSD, and the proof is complete. 
\end{proof}

\begin{corollary}\label{cor:TWSD-nonsingular}
Let $\matr{A}, \matr{B}\in\textbf{symm}(\RR^{m\times m})$, and \( \matr{A} \) be nonsingular.
If \( \matr{A}^{-1}\matr{B} \) has a real eigenvalue, then \( \matr{A} \) and \( \matr{B} \) are TWSD.
\end{corollary}

\begin{proof}
If \( \matr{A}^{-1}\matr{B} \) has a real eigenvalue, by \cref{lem:Uhlig-canonical-nonsingular}, there exists \(\matr{P}\in\GL_{m}(\RR)\) such that
\begin{align*} 
{\small \matr{P}^{\T}\matr{A}\matr{P}} &{\small= \Diag{a_1 \matr{E}(m_{1}), \ldots, a_r\matr{E}(m_{r}), \matr{E}(m_{r+1}), \ldots \matr{E}(m_{p})},}\\ 
{\small\matr{P}^{\T}\matr{B}\matr{P}} &{\small= \textbf{Diag}\{a_1 \matr{E}(m_{1})\matr{J}(\lambda_{1}, m_{1}), \ldots, a_r\matr{E}(m_{r})\matr{J}(\lambda_{r}, m_{r})},\\
&\ \ \ \ \ \ \ \ \ \   {\small\matr{E}(m_{r+1})\matr{J}(\lambda_{r+1}, m_{r+1}), \ldots \matr{E}(m_{p})\matr{J}(\lambda_{p}, m_{p})\}},
\end{align*}
where \( a_s = \pm 1 \) for $1 \leq s \leq r$, as in \cref{lem:Uhlig-canonical-nonsingular}.
Without loss of generality, we suppose that \( \matr{J}(\lambda_{1}, m_{1}) \) is the Jordan block with real eigenvalue.
Then \( a_1\matr{E}(m_{1}) \) and \( a_1\matr{E}(m_{1})\matr{J}(\lambda_{1}, m_{1}) \) are TWSD-B by \cref{thm:TWSD-B-nonsingular-pair}.
It follows that \(\matr{A} \) and \( \matr{B} \) are TWSD by \cref{thm:TWSD-block-TWSD-B-suffic}. The proof is complete. 
\end{proof}

\subsection{Equivalence of TWSD and TWSD-B with a totally diagonal nonsingular matrix}
By \cref{def:twsd}, there is a big difference between TWSD and TWSD-B. 
If the set $\mathcal{C}$ in \eqref{set_C} is TWSD-B, there exists $\{\matr{P}_k\}_{k \geq 1} \subseteq \SL_{m}(\RR)$ such that $\matr{P}_k^{\top} \matr{A}_i\matr{P}_k$ converge to diagonal matrices for all $1\leq i \leq L$. However, if the set $\mathcal{C}$ is TWSD, we only require that the off-diagonal elements of $\matr{P}_k^{\top} \matr{A}_i\matr{P}_k$ converge to 0, while the diagonal elements of $\matr{P}_k^{\top} \matr{A}_i\matr{P}_k$ may be unbounded.
Then a natural question is whether TWSD is equivalent to TWSD-B, if we further require that the diagonal elements of one matrix, \emph{e.g.},  $\matr{P}_k^{\top} \matr{A}_1\matr{P}_k$,  are bounded, not for all $1\leq i\leq L$.
In the following \Cref{thm:equi_A1_bounded}, we will show that it is true, if the matrix $\matr{A}_1$ is nonsingular and \( \matr{P}_k^{\top} \matr{A}_1 \matr{P}_k \) is bounded and diagonal for all \( k \geq 1 \).
The proof is postponed to 
\Cref{proofs-2}. 

\begin{theorem}\label{thm:equi_A1_bounded}
Let the set $\mathcal{C}$ be as in \eqref{set_C}, and \( \matr{A}_1 \) be nonsingular.
Suppose that \( \mathcal{C} \) is TWSD, \emph{i.e.,} there exists a sequence $\{\matr{P}_k\}_{k \geq 1} \subseteq \mathbf{SL}_{m}(\mathbb{R})$ such that $\lim_{k\rightarrow\infty}\textbf{offdiag}(\matr{P}_k^{\top} \matr{A}_i\matr{P}_k)=\matr{0}$ for all $1\leq i\leq L$.
If $\matr{P}_k^{\top} \matr{A}_1\matr{P}_k$ is diagonal and  uniformly bounded for $k\geq 1$, then %$\{\textbf{diag}(\matr{P}_k^{\top} \matr{A}_1\matr{P}_k)\}_{k\geq 1, 1\leq i\leq L}$ are uniformly bounded, and thus
\( \mathcal{C} \) is TWSD-B.
\end{theorem}

\begin{remark}
In \cref{thm:equi_A1_bounded}, the condition that $\matr{A}_1$ is nonsingular is necessary. 
For example, if $\matr{A}, \matr{B}, \matr{P}_k\in \textbf{symm}(\RR^{m\times m})$
are the matrices defined as in \cref{exa:TWSD-not-TWSD-B}, then 
it is clear that \( \mathcal{C}=\{\matr{A}, \matr{B}, \matr{0}_{3\times 3}\} \) is TWSD, and \( \matr{P}_k^{\top} \matr{0}_{3\times 3} \matr{P}_k \) is bounded. However, the set \( \mathcal{C} \) is not TWSD-B.
\end{remark}

\subsection{Equivalence of TWSD and TWSD-B for two matrices in \( \mathbf{symm}(\mathbb{R}^{2 \times 2}) \)}
It has been shown in \cref{exa:TWSD-not-TWSD-B} that $\set{TWSD}\textrm{-}\set{B}\subsetneqq\set{TWSD}$ in general. 
Now we prove an interesting fact that they are actually equivalent to each other for a pair of matrices \( \{ \matr{A}, \matr{B} \} \subseteq \mathbf{symm}(\mathbb{R}^{2 \times 2}) \). 
%Here we show the equivalence by the canonical form of a matrix pair.
\begin{lemma}\label{lem:2times2_equiv}
Let $\matr{A},\matr{B}\in\textbf{symm}(\RR^{2\times 2})$.
Then the set \(\{ \matr{A}, \matr{B} \}\) is TWSD-B if and only if it is TWSD.
\end{lemma}

\begin{proof}%[Proof of \Cref{lem:2times2_equiv}]
If \( \{ \matr{A}, \matr{B} \}\subseteq \textbf{symm}(\RR^{2\times 2}) \) is a nonsingular pair, we assume that \( \matr{A} \) is nonsingular without loss of generality. 
Denote by \( \matr{J} \) the Jordan normal form of \( \matr{A}^{-1}\matr{B} \). 
By \cref{lem:Uhlig-canonical-nonsingular}, we only need to consider the following three cases.

\textbf{Case 1:} \( \matr{A}^{-1}\matr{B} \) has two different real eigenvalues, and \( \matr{J}=\begin{bmatrix}
\lambda_1 & 0 \\
0 & \lambda_2
\end{bmatrix} \).  
In this case, they are TWSD-B by \cref{thm:TWSD-B-nonsingular-pair}, which also implies they are TWSD.

\textbf{Case 2:} \( \matr{A}^{-1}\matr{B} \) has one real eigenvalue, and \( \matr{J}=\begin{bmatrix}
\lambda & 1 \\
0 & \lambda
\end{bmatrix} \).
In this case, they are TWSD-B by \cref{thm:TWSD-B-nonsingular-pair}, which also implies they are TWSD.

\textbf{Case 3:} \( \matr{A}^{-1}\matr{B} \) has a pair of complex eigenvalues \( a \pm b i \), and
\( \matr{J} =\left[ \begin{array}{cc} a & -b\\ b & a \end{array} \right]\).  
In this case, there exists \( \matr{P}\in\GL_2(\RR) \) such that
\[
\matr{P}^{\T}\matr{A} \matr{P} = \begin{bmatrix}
0 & 1\\
1 & 0
\end{bmatrix}, \
\matr{P}^{\T}\matr{B} \matr{P} = \begin{bmatrix}
  -b & a \\
  a & b
\end{bmatrix} =  a\matr{P}^{\T}\matr{A} \matr{P} - b \begin{bmatrix}
  1 & 0 \\
  0 & -1
\end{bmatrix}.
\]
Thus, they are TWSD if and only if \( \begin{bmatrix}
0 & 1\\
1 & 0
\end{bmatrix} \) and \( \begin{bmatrix}
1 & 0 \\
0 & -1
\end{bmatrix} \) are TWSD.
It has be shown in \cref{exa:not-TWSD} that this set is not TWSD. Thus, it is also not TWSD-B.

If \( \{ \matr{A}, \matr{B}\}  \) is a singular pair, by \cref{lem:lancaster-canonical-general-pair}, we only need to consider the following case.

\textbf{Case 4:}  There exists \( \matr{P}\in\GL_2(\RR) \) such that
\[
\matr{P}^{\T}\matr{A} \matr{P} = \begin{bmatrix}
a & 0\\
0 & 0
\end{bmatrix}, \
\matr{P}^{\T}\matr{B} \matr{P} = \begin{bmatrix}
  b & 0 \\
  0 & 0
\end{bmatrix}.    \]
Since they are already diagonal, they are both TWSD and TWSD-B.
\end{proof}

\section{Decomposition based simultaneously diagonalizable matrices}\label{sec:decom_based_SD}

In this section, we will extend the SDO and SD notions using the function $\varphi_{\rm D}$ in \eqref{eq:func_D}, and propose the notions \(\mathbf{D}_{m,n}\)-SDO, \(\mathbf{D}_{m,n}\)-SD and DWSD in \Cref{table-example-3-0}.  
For the notions \(\mathbf{D}_{m,n}\)-SD and DWSD, we will show that they are exactly the notions  $d$-\emph{RSDC} and \emph{ASDC} proposed in~\cite{wang2021new}. 
For the new notion \(\mathbf{D}_{m,n}\)-SDO, we will prove an interesting result in \Cref{thm:trival-Gmn-SDO}, which will be applied to ICA in \Cref{sec:application-approximate-diagonalization}. 

% Then we will focus on TWSD-B, and prove several necessary and/or sufficient conditions of TWSD-B for a pair of matrices under different assumptions. 
% In this section, we will extend the SDO and SD notions using the function $\varphi_{\rm D}$ in \eqref{eq:func_D}. 
% Two of these new notions are essentially equivalent to the ASDC and d-RSDC proposed in~\cite{wang2021new}, which have been well studied there. 
%So we mainly present the results obtained there.

\subsection{Decomposition based projectively simultaneously diagonalizable matrices}
%Similar to \( \mathbf{T}_{m,n} \)-SD and \( \mathbf{T}_{m,n} \)-SDO, we define new notions as follows.

\begin{definition}\label{def:D-SDO-SD}
Let the set $\mathcal{C}$ be as in \eqref{set_C}, and $n\geq m$.\\
(i) The set $\mathcal{C}$ is \emph{decomposition based projectively simultaneously diagonalizable on $\St(m,n)$} ($\textbf{D}_{m,n}$-SDO), if there exist $\matr{P}\in\St(m,n)$ and $\mathcal{D} = \{ \matr{D}^{(i)}\}_{1 \leq i \leq L} \subseteq \mathbf{D}_{n}$, such that
$\varphi_{\rm D}(\matr{P},\mathcal{D}) = 0$.
%, {\it i.e.}, there exist \( \matr{P} \in \St(m,n) \) and diagonal matrices \( \matr{D}^{(i)} \) such that \(\matr{A}_{i} = \matr{P}^{\T}D^{(i)} \matr{P} \) for all \( 1 \leq i \leq L \). 
We denote the class of $\textbf{D}_{m,n}\textrm{-}$SDO sets by $\set{D}_{m,n}\textrm{-}\set{SDO}$. \\
(ii) The set $\mathcal{C}$ is \emph{decomposition based projectively simultaneously diagonalizable} on $\textbf{RSL}(m,n)$($\textbf{D}_{m,n}$-SD), if there exist $\matr{P}\in\textbf{RSL}(m,n)$ and $\mathcal{D} = \{ \matr{D}^{(i)}\}_{1 \leq i \leq L} \subseteq \mathbf{D}_{n}$, such that
$\varphi_{\rm D}(\matr{P},\mathcal{D}) = 0$. 
We denote the class of \( \mathbf{D}_{m,n} \)-SD sets by $\set{D}_{m,n}\textrm{-}\set{SD}$.
\end{definition}

For $\matr{X}\in\RR^{n\times m}$, we define a mapping
\begin{equation*}\label{eq:func_rho}
\rho_{\matr{X}}:\RR^{n\times n}\rightarrow\RR^{m\times m},\  \matr{B}\mapsto\matr{X}^{\T}\matr{B}\matr{X}.
\end{equation*}
By this mapping, we have the following equivalent characterizations of \(\mathbf{D}_{m,n}\)-SDO.

\begin{lemma}\label{lem:equiv_D_SDO}
Let the set $\mathcal{C}$ be as in \eqref{set_C}. 
Then the following statements are equivalent:\\
(i) $\mathcal{C}$ is $\textbf{D}_{m,n}$-SDO.\\
(ii) there exists $\matr{X}\in\St(m,n)$ and a set $\mathcal{S}\subseteq\textbf{symm}(\RR^{n\times n})$ such that $\mathcal{S}$ is SDO, and $\rho_{\matr{X}}(\mathcal{S})=\mathcal{C}$.\\
(iii) there exists a set $\mathcal{S}\subseteq\textbf{symm}(\RR^{n\times n})$ such that $\mathcal{S}$ is SDO, and $\rho_{\matr{I}_{m,n}}(\mathcal{S})=\mathcal{C}$.
\end{lemma}

\begin{proof}
(i) \(\Rightarrow\)(ii): By \Cref{def:D-SDO-SD}(i), there exist \( \matr{P} \in \St(m,n) \) and diagonal matrices \( \matr{D}^{(i)} \) such that \( \matr{A}_{i}= \matr{P}^{\top} \matr{D}^{(i)} \matr{P} \) for all \( 1 \leq i \leq L \). Let \( \matr{X} = \matr{P} \) and \( \mathcal{S} = \{ \matr{D}^{(i)}\}_{1\leq i\leq L}\). Then the set \( \mathcal{S} \) is SDO and \( \rho_{\matr{X}}(\mathcal{S}) = \mathcal{C} \).\\
(ii) \(\Rightarrow\) (iii): Since \( \matr{X} \in \St(m,n) \), there exists an orthogonal matrix \( \matr{Q} \in \mathbb{R}^{n \times n} \) such that \( \matr{X} = \matr{Q} \matr{I}_{m,n} \). Let \( \mathcal{\bar{S}} = \{ \matr{Q}^{\top}\matr{S}^{(i)} \matr{Q} \mid \matr{S}^{(i)} \in \mathcal{S}\} \).
Then $\mathcal{\bar{S}}$ is SDO, since \( \mathcal{S} \) is SDO and \( \matr{Q} \) is orthogonal. Note that \( \matr{D}^{(i)} = \matr{X}^{\top} \matr{S}^{(i)} \matr{X} = \matr{I}_{m,n}^{\top} \matr{Q}^{\top} \matr{S}^{(i)} \matr{Q} \matr{I}_{m,n} \). It follows that \( \rho_{\matr{I}_{m,n}}(\mathcal{\bar{S}}) = \mathcal{C} \).\\
(iii) \( \Rightarrow \) (i): Since \( \mathcal{S} \) is SDO, there exists an orthonormal matrix \( \matr{Q} \) and diagonal matrices \( \matr{D}^{(i)} \) for \( 1 \leq i \leq L \) such that \( \mathcal{S} = \{ \matr{Q}^{\top} \matr{D}^{(i)} \matr{Q} \}_{1 \leq i \leq L} \). Since \( \rho_{\matr{I}_{m,n}}(\mathcal{S}) = \mathcal{C} \), for all \( 1 \leq i \leq L \), we have \( \matr{I}_{m,n}^{\top} \matr{Q}^{\top} \matr{D}^{(i)} \matr{Q} \matr{I}_{m,n} = \matr{A}_{i} \). 
Let \( \matr{P} = \matr{Q} \matr{I}_{m,n} \).  Then \( \matr{P} \in \St(m,n) \) and \( \matr{A}_{i} = \matr{P}^{\top} \matr{D}^{(i)} \matr{P} \). It follows that \( \mathcal{C} \) is \( \mathbf{D}_{m,n} \)-SDO.
\end{proof}

%Similarly, we have the following equivalent characterizations of \(\mathbf{D}_{m,n}\)-SD.

%Similar results can be obtained for \( \mathbf{D}_{m,n} \)-SD.

\begin{lemma}\label{lem:equiv_D_SD}
%Let $\rho_{\matr{X}}$ be as in \eqref{eq:func_rho}. 
Let the set $\mathcal{C}$ be as in \eqref{set_C}. 
Then the following statements are equivalent:\\
(i) $\mathcal{C}$ is $\textbf{D}_{m,n}$-SD.\\
(ii) there exists $\matr{X}\in\RSL(m,n)$ and a set $\mathcal{S}\subseteq\textbf{symm}(\RR^{n\times n})$ such that $\mathcal{S}$ is SD, and $\rho_{\matr{X}}(\mathcal{S})=\mathcal{C}$.\\
(iii) there exists a set $\mathcal{S}\subseteq\textbf{symm}(\RR^{n\times n})$ such that $\mathcal{S}$ is SD, and $\rho_{\matr{I}_{m,n}}(\mathcal{S})=\mathcal{C}$.
\end{lemma}

\begin{proof}
(i) \(\Rightarrow\)(ii): By \Cref{def:D-SDO-SD}(ii), there exist \( \matr{P} \in \RSL(m,n) \) and diagonal matrices \( \matr{D}^{(i)} \) such that \( \matr{A}_{i}= \matr{P}^{\top} \matr{D}^{i} \matr{P} \) for all \( 1 \leq i \leq L \). Let \( \matr{X} = \matr{P} \) and \( \mathcal{S} = \{ \matr{D}^{(i)}\}_{1\leq i\leq L}\). Then, the set \( \mathcal{S} \) is SD and \( \rho_{\matr{X}}(\mathcal{S}) = \mathcal{C} \).\\
(ii) \(\Rightarrow\) (iii): %Assume \( \mathcal{S} \) is SD, \( \matr{X} \in \RSL(m,n) \) and \( \rho_{\matr{X}}(\mathcal{S}) = \mathcal{C} \). 
Since \( \matr{X} \in \RSL(m,n) \), it has full column rank. There exists a nonsingular matrix \( \matr{U} \in \mathbb{R}^{n \times n} \) such that \( \matr{X} = \matr{U} \matr{I}_{m,n} \). Let \( \mathcal{\bar{S}} = \{ \matr{U}^{\top}\matr{S}^{(i)} \matr{U} | \matr{S}^{(i)} \in \mathcal{S}\} \). Then \( \mathcal{\bar{S}}\) is SD,  since \( \mathcal{S} \) is SD and \( \matr{U} \) is nonsingular. 
Note that \( \matr{D}^{(i)} = \matr{X}^{\top} \matr{S}^{(i)} \matr{X} = \matr{I}_{m,n}^{\top} \matr{U}^{\top}\) \(\matr{S}^{(i)} \matr{U} \matr{I}_{m,n} \). 
It follows that \( \rho_{\matr{I}_{m,n}}(\mathcal{\bar{S}}) = \mathcal{C} \).\\
(iii) \( \Rightarrow \) (i): Since \( \mathcal{S} \) is SD, there exists a nonsingular matrix \( \matr{U} \) and diagonal matrices \( \matr{D}^{(i)} \) such that \( \mathcal{S} = \{ \matr{U}^{\top} \matr{D}^{(i)} \matr{U} \}_{1 \leq i \leq L} \). Since \( \rho_{\matr{I}_{m,n}}(\mathcal{S}) = \mathcal{C} \), for all \( 1 \leq i \leq L \), we have \( \matr{I}_{m,n}^{\top} \matr{U}^{\top} \matr{D}^{(i)} \matr{U} \matr{I}_{m,n} = \matr{A}_{i} \). 
Let \( \matr{P} = \matr{U} \matr{I}_{m,n} \). Then \( \matr{P} \in \RSL(m,n) \) and \( \matr{A}_{i} = \matr{P}^{\top} \matr{D}^{(i)} \matr{P} \). 
It follows that \( \mathcal{C} \) is \( \mathbf{D}_{m,n} \)-SD.
\end{proof}

\begin{remark}
For a matrix $\matr{B}\in\RR^{n\times n}$, it is easy to see that $\rho_{\matr{I}_{m,n}}(\matr{B})$ is the top-left \( m \times m \) submatrix of \( \matr{B} \).
Therefore, the equivalent characterization in  \cref{lem:equiv_D_SD}(iii) is actually the ($n-m$)-RSDC proposed in \cite[Definition 13]{wang2021new}.
\end{remark}

% \begin{lemma}
% \begin{align*}
% \textbf{D}_{m,n}-\textbf{SDO} = \textbf{R}_{m,n}-\textbf{SDO}, \ \ \ \textbf{D}_{m,n}-\textbf{SD} = \textbf{R}_{m,n}-\textbf{SD}
% \end{align*}
% \end{lemma}

We now show that the classes $\set{D}_{m,n}\textrm{-}\set{SDO}$ and $\set{D}_{m,n}\textrm{-}\set{SD}$ both become larger when \( n \) grows.

\begin{lemma}\label{lem:D-SD-larger}
For any $n\geq m$, we have that
\begin{equation*}
\set{D}_{m,n}\textrm{-}\set{SDO}\subseteq \set{D}_{m,n+1}\textrm{-}\set{SDO}, \ \ \  \set{D}_{m,n}\textrm{-}\set{SD}\subseteq \set{D}_{m,n+1}\textrm{-}\set{SD}.
\end{equation*}
\end{lemma}

\begin{proof}
If the set \( \mathcal{C} \) is \( \mathbf{D}_{m,n} \)-SDO, there exist \( \matr{P} \in \St(m,n) \) and diagonal matrices \( \matr{D}^{(i)} \) such that \( \matr{A}_{i} = \matr{P}^{\top} \matr{D}^{(i)} \matr{P} \) for all \( 1 \leq i \leq L \). Now we define a matrix \( \matr{\tilde{P}} \in \mathbb{R}^{(n+1) \times m} \) by
\[ \tilde{P}_{ij} = \left\{
\begin{array}{cl}
 P_{ij} & i \leq n, \\
  0 & i = n+1,
\end{array}
\right.\]
for $\ 1 \leq i \leq n+1, 1 \leq j \leq m$. It is the case that \(\matr{\tilde{P}}\in \St(m, n+1) \). Let \( \matr{\tilde{D}}^{(i)} = \Diag{\matr{D}^{(i)}, 0} \) for all \( 1 \leq i \leq L \). Then \( \matr{A}_{i} = \matr{P}^{\top} \matr{D}^{(i)} \matr{P} = \matr{\tilde{P}}^{\top} \matr{\tilde{D}}^{(i)} \matr{\tilde{P}}\), and thus \( \mathcal{C}\) is \(\mathbf{D}_{m,n+1} \)-SDO. %It follows that $\set{D}_{m,n}\textrm{-}\set{SDO} \subseteq  \set{D}_{m,n+1}\textrm{-}\set{SDO}$.
The other case for $\mathbf{D}_{m,n}$-SD can be proved similarly, and the proof is complete.
\end{proof}

% \begin{lemma}\label{lem:equiv-H-SD-G-SD}
% For any $n\geq m$, we have $\textbf{D}_{m,n}\textrm{-SD}\subseteq \textbf{T}_{m,n}\textrm{-SD}$ and  $\textbf{D}_{m,n}\textrm{-SDO} \subseteq \textbf{T}_{m,n}\textrm{-SDO}$. 
% \end{lemma}

% \begin{proof}
% When \( A_1, A_2, \ldots , A_L \) is \( \textbf{D}_{m,n}-SD \), there exists \( P \in \mathbf{RSL}(m,n) \) such that \( PA_iP^{\T} = D_i \), where $D_i \in \mathbb{R}^{n \times n}$ is diagonal. Since \( P \) has left inverse \( P^{\dag} \) such that \( P^{\dag}P = I_m \) and \( P^{\T} (P^{\dag})^{\T} = I_m \), we have \( A_i= P^{\dag}D_i(P^{\dag})^{\T} \). Note that \( P^{\dag} \) has full row rank, we have \( \psi((P^{\dag})^{\T}, D_i) = 0 \), that is, $A_i$ is \( \textbf{T}_{m,n}-SD \).

% When \( A_1, A_2, \ldots , A_L \) is \( \textbf{D}_{m,n}-SDO \), there exists \( P \in \St(m,n) \) such that \( PA_iP^{\T} = D_i \), where $D_i \in \mathbb{R}^{n \times n}$ is diagonal. Since \( P \in \St(m, n) \), \( P^{\T}P = I_m \) and \( P^{\T} P = I_m \), we have \( A_i= P^{\T}D_iP \). Note that \( P^{\T} \) has full row rank, we have \( \psi(P, D_i) = 0 \), that is, $A_i$ is \( \textbf{T}_{m,n}-SD \).
% \end{proof}

An interesting fact is that, if $n$ is large enough in \( \mathbf{D}_{m,n} \)-SDO, \emph{e.g.}, $n = Lm$, then $\set{D}_{m,n}\textrm{-}\set{SDO}$ will include all sets of symmetric matrices in $\textbf{symm}(\RR^{m\times m})$.

\begin{theorem}\label{thm:trival-Gmn-SDO}
Let the set $\mathcal{C}$ be as in \eqref{set_C}. 
Then $\mathcal{C}$ is $\textbf{D}_{m,Lm}$-SDO.
\end{theorem}

\begin{proof}
Suppose the orthogonal decomposition of $\matr{A}_{i}$ is $\matr{A}_{i} = \matr{Q}_{i}^{\T}\matr{X}_{i}\matr{Q}_{i}$, where $\matr{Q}_{i}\in\SON_{m}$ and  $\matr{X}_{i} \in \mathbb{R}^{m \times m}$ is a diagonal matrix for \( 1 \leq i \leq L \). 
Define matrices \( \matr{P} \in \mathbb{R}^{Lm \times m}, \matr{D}^{(i)} \in \mathbb{R}^{Lm \times Lm} \) by
\[
\matr{P} = \frac{1}{\sqrt{L}}\begin{bmatrix}
\matr{Q}_1 \\
\matr{Q}_2 \\
\vdots \\
\matr{Q}_L
\end{bmatrix}, \ \ 
\matr{D}^{(i)} = \Diag{\matr{0}_{m \times m}, \matr{0}_{m \times m} , \ldots , L\matr{X}_{i} , \ldots , \matr{0}_{m \times m}}, 
\]
for $1\leq i\leq L$, 
where the $i$-th block of $\matr{D}^{(i)}$ is $L\matr{X}_{i}$ and others are all \( \matr{0}_{m \times m} \).
Then 
\[
\matr{P}^{\T}\matr{D}^{(i)}\matr{P} = \matr{Q}_{i}^{\T}\matr{X}_{i}\matr{Q}_{i} = \matr{A}_{i}, 
\]
for $1\leq i\leq L$. 
Note that \( \matr{P} \in \St(m,Lm) \). 
The set $\{\matr{A}_i\}_{1\leq i\leq L}$ is $\textbf{D}_{m,Lm}$-SDO, and thus the proof is complete. 
\end{proof}

\begin{corollary}
Let the set $\mathcal{C}$ be as in \eqref{set_C}.  
Then $\mathcal{C}$ is $\textbf{D}_{m,m^2(m+1)/2}$-SDO.
\end{corollary}

\begin{proof}
Consider a basis of \( \mathbf{symm}(\mathbb{R}^{m \times m}) \), for example, $\mathcal{T} = \{ \matr{T}^{(i,j)}, 1 \leq i \leq j \leq m \}$, where $\matr{T}^{(i,j)}$ is a symmetric matrix whose \((i,j), (j,i)\)-th entries are \( 1 \), and others are \( 0 \). 
By \cref{thm:trival-Gmn-SDO}, there exists a matrix \( \matr{P} \in \St(m,m^2(m+1)/2) \) such that $\mathcal{T}$ is  $\mathbf{D}_{m,m^2(m+1)/2}$-SDO by \( \matr{P} \). 
Since \( \mathcal{T} \) is a basis of \( \mathbf{symm}(\mathbb{R}^{m \times m}) \), any matrix of \( \mathcal{C} \) can be expressed as a linear combination of \( \mathcal{T} \). Thus, the matrix \( \matr{P} \) also diagonalizes \( \mathcal{C} \). The proof is complete. 
\end{proof}

Correspondingly, it was proved in~\cite{wang2021new} that almost all the matrix pairs in \( \mathbf{symm}(\mathbb{R}^{m \times m}) \) are \( \mathbf{D}_{m,m+1} \)-SD (equivalently, 1-RSDC). %Besides, they also give algorithms to get \( \matr{D}^{(i)} \) and nonsingular matrix \( \matr{P} \).
We present them here for the convenience of readers. 

\begin{lemma}[{\cite[Theorem 12]{wang2021new}}]
Let $\matr{A}, \matr{B} \in \mathbf{symm}(\mathbb{R}^{m \times m})$. Then for any $\epsilon>0$, there exist $\tilde{\matr{A}}, \tilde{\matr{B}} \in \mathbf{symm}(\mathbb{R}^{m \times m})$ satisfying $\|\matr{A}-\tilde{\matr{A}}\|<\epsilon$ and $\|\matr{B}-\tilde{\matr{B}}\|<\epsilon$, such that $\{\tilde{\matr{A}}, \tilde{\matr{B}}\}$ is \( \mathbf{D}_{m, m+1} \)-SD. Furthermore, if $\matr{A}$ is nonsingular and $\matr{A}^{-1} \matr{B}$ has simple eigenvalues, then $\{\matr{A}, \matr{B}\}$ is itself \( \mathbf{D}_{m, m+1} \)-SD.
\end{lemma}

\begin{lemma}[{\cite[Corollary 5]{wang2021new}}]\label{lem:d-sd-almost-sure}
Let $\{\matr{A}, \matr{B}\}$ be a pair of matrices jointly sampled according to an absolutely continuous probability measure on $\mathbf{symm}(\mathbb{R}^{m \times m}) \times \mathbf{symm}(\mathbb{R}^{m \times m})$. Then, the set $\{\matr{A}, \matr{B}\}$ is $\mathbf{D}_{m,m+1}$-SD almost surely.
\end{lemma}

\subsection{Decomposition based weakly simultaneously diagonalizable matrices}
%In this subsection, we first give two notions, and show the first one is equivalent to ASDC in~\cite{wang2021new}.
%Similar to \( \mathbf{D}_{m,n} \)-SD and \( \mathbf{D}_{m,n} \)-SDO, we define new notions as follows.

\begin{definition}\label{def:DWSD}
Let the set $\mathcal{C}$ be as in \eqref{set_C}.
The set $\mathcal{C}$ is \emph{decomposition based weakly simultaneously diagonalizable} (DWSD), if there exist $\{\matr{P}_k\}_{k \geq 1} \subseteq \SL_{m}(\RR)$ and $\mathcal{D}_{k} = \{ \matr{D}_k^{(i)}\}_{1 \leq i \leq L} \subseteq \mathbf{D}_{m}$ for $k\geq 1$, such that
\begin{align*}
\lim_{k\rightarrow\infty} \varphi_{\rm D}(\matr{P}_k, \mathcal{D}_{k}) = 0.
\end{align*}
We denote the class of DWSD sets by $\set{DWSD}.$  
% \\
% \begin{align}
% \lim_{k\rightarrow\infty} \varphi_{\rm D}(\matr{P}_k, \mathcal{D}_{k}) = 0,
% \end{align}
% and $\|\matr{D}_k^{(i)}\|$ is uniformly bounded. We denote the class of DWSD-B sets by $\set{DWSD}\textrm{-}\set{B}.$  \\
\end{definition}

It is easy to see that the notion DWSD in \cref{def:DWSD} is exactly the notion ASDC proposed in \cite{wang2021new}. 
We now recall several important results about DWSD in~\cite{wang2021new}, which will be used in \cref{sec:relationship}.

\begin{lemma}[{\cite[Theorem 7]{wang2021new}}]\label{lem:DWSD-nonsingular-pair}
Let $\matr{A}, \matr{B} \in \mathbf{symm}(\mathbb{R}^{m \times m})$ and $\matr{A}$ be nonsingular. 
Then the set $\{\matr{A}, \matr{B}\}$ is DWSD if and only if $\matr{A}^{-1} \matr{B}$ has only real eigenvalues.
\end{lemma}

\begin{lemma}[{\cite[Theorem 8]{wang2021new}}]\label{lem:DWSD-singular-pair}
Let $\matr{A}, \matr{B} \in \mathbf{symm}(\mathbb{R}^{m \times m})$.
If the set $\{\matr{A}, \matr{B}\}$ is singular, then it is DWSD.
\end{lemma}

\begin{lemma}[{\cite[Theorem 9]{wang2021new}}]\label{lem:DWSD-triple-necessary-sufficient-condition}
Let $\matr{A}, \matr{B}, \matr{C} \in \mathbf{symm}(\mathbb{R}^{m \times m})$ and $\matr{A}$ be nonsingular. 
Then, the set $\{\matr{A}, \matr{B}, \matr{C}\}$ is DWSD if and only if $\left\{\matr{A}^{-1} \matr{B}, \matr{A}^{-1} \matr{C}\right\}$ is a pair of commuting matrices with real eigenvalues.
\end{lemma}

\begin{lemma}[{\cite[Corollary 1]{wang2021new}}]\label{lem:PD-DWSD-equiv-SD}
Let the set $\mathcal{C}$ be as in \eqref{set_C}, and $\matr{S} \in \operatorname{span}(\mathcal{C})$ be positive definite. 
Then \( \mathcal{C}\) is SD if and only if it is DWSD.
\end{lemma}

Although the following necessary condition for DWSD is not formally stated in~\cite{wang2021new}, it can be derived easily  from the proof of \cite[Theorem 9]{wang2021new}.  %\cref{lem:DWSD-triple-necessary-sufficient-condition}

\begin{lemma}\label{lem:DWSD-general-necessary-condition}
Let the set $\mathcal{C}$ be as in \eqref{set_C}, and $\matr{S} \in \operatorname{span}(\mathcal{C})$ be nonsingular. 
If \( \mathcal{C} \) is DWSD, then \( \matr{S}^{-1}\matr{A}_i \) has only real eigenvalues for $1\leq i\leq   L$, and $[\matr{A}_i, \matr{A}_j]_{\matr{S}}=\matr{0}$ for all \( 1 \leq i\neq  j \leq L\).
\end{lemma}

\section{The relationships between projectively and weakly simultaneously diagonalizable matrices}\label{sec:relationship}

In \cref{sec:preliminary,sec:trans_based_SD,sec:TWSD,sec:decom_based_SD}, using the functions $\varphi_{\rm T}$ in \eqref{eq:func_T} and $\varphi_{\rm D}$ in \eqref{eq:func_D}, we have defined several classes of sets of symmetric matrices in \Cref{table-example-3-0}, and presented many sufficient and/or necessary conditions of them.
In this section, we will summarize and study the relationships between these new notions and conditions from different perspectives.

\subsection{General relationships}

The relationships between the classes of sets of symmetric matrices in \Cref{table-example-3-0} can be summarized in \cref{fig:relationship-WSD-PSD}.  In particular, the relationship between \( \mathcal{T}_{m,n}\textrm{-}\mathcal{SDO} \) and \( \set{SDO} \) was proved in \cref{thm:TSDO-SDO}(i), and the relationship between \( \mathcal{T}_{m,n}\textrm{-}\mathcal{SDO} \) and $\set{SD}$ was proved in \cref{thm:TSDO-SDO}(ii). The relationship between \( \mathcal{D}_{m,n} \textrm{-}\mathcal{SDO} \)(\( \mathcal{D}_{m,n} \textrm{-}\mathcal{SD} \)) and \( \mathcal{D}_{m,n+1}\textrm{-}\mathcal{SDO}\)(\( \mathcal{D}_{m,n+1}\textrm{-}\mathcal{SD}\)) was proved in \cref{lem:D-SD-larger}.

\begin{figure}
\begin{tikzcd}
% \vdots\arrow[d, phantom, "{\rotatebox[origin=c]{90}{$\subseteq$}}" description] & \vdots \arrow[d, phantom, "{\rotatebox[origin=c]{90}{$\subseteq$}}" description]  & & \\
\set{T}_{m,n}\set{\textrm{-}SDO}\arrow[d, phantom, "{\rotatebox[origin=c]{90}{$=$}}" description] \arrow[r, phantom, "\subseteq" description]& \set{T}_{m,n}\set{\textrm{-}SD} \arrow[d, phantom, "{\rotatebox[origin=c]{90}{$=$}}" description]  & & \\
\vdots\arrow[d, phantom, "{\rotatebox[origin=c]{90}{$=$}}" description] & \vdots \arrow[d, phantom, "{\rotatebox[origin=c]{90}{$=$}}" description]  & & \\
\set{T}_{m,m+1}\set{\textrm{-}SDO}\arrow[d, phantom, "{\rotatebox[origin=c]{90}{$=$}}" description] \arrow[r, phantom, "\subseteq" description]& \set{T}_{m,m+1}\set{\textrm{-}SD} \arrow[d, phantom, "{\rotatebox[origin=c]{90}{$=$}}" description]  & &  \\
\set{T}_{m,m}\set{\textrm{-}SDO}\arrow[d, phantom, "{\rotatebox[origin=c]{90}{$=$}}" description] \arrow[r, phantom, "\subseteq" description]& \set{T}_{m,m}\set{\textrm{-}SD} \arrow[d, phantom, "{\rotatebox[origin=c]{90}{$=$}}" description]  \arrow[r, phantom, "\subseteq" description]& \set{TWSD\textrm{-}B} \arrow[r, phantom, "\subseteq" description]& \set{TWSD}\\
\set{SDO}\arrow[d, phantom, "{\rotatebox[origin=c]{90}{$=$}}" description] \arrow[r, phantom, "\subseteq" description]& \set{SD} \arrow[d, phantom, "{\rotatebox[origin=c]{90}{$=$}}" description]  & \\
\set{D}_{m,m}\set{\textrm{-}SDO}\arrow[d, phantom, "{\rotatebox[origin=c]{270}{$\subseteq$}}" description] \arrow[r, phantom, "\subseteq" description]& \set{D}_{m,m}\set{\textrm{-}SD} \arrow[d, phantom, "{\rotatebox[origin=c]{270}{$\subseteq$}}" description] \arrow[r, phantom, "\subseteq" description]& \set{DWSD}\\
\set{D}_{m,m+1}\set{\textrm{-}SDO}\arrow[d, phantom, "{\rotatebox[origin=c]{270}{$\subseteq$}}" description] \arrow[r, phantom, "\subseteq" description]& \set{D}_{m,m+1}\set{\textrm{-}SD} \arrow[d, phantom, "{\rotatebox[origin=c]{270}{$\subseteq$}}" description]  & & \\
\vdots\arrow[d, phantom, "{\rotatebox[origin=c]{270}{$\subseteq$}}" description] & \vdots \arrow[d, phantom, "{\rotatebox[origin=c]{270}{$\subseteq$}}" description]  & & \\
\set{D}_{m,n}\set{\textrm{-}SDO} \arrow[r, phantom, "\subseteq" description]& \set{D}_{m,n}\set{\textrm{-}SD} & & \\
%\vdots  & \vdots   & &
\end{tikzcd}
\caption{General relationships}\label{fig:relationship-WSD-PSD}
\end{figure}

% \begin{figure}{\small
% \begin{tikzcd}
% \set{T}_{m,m}\set{\textrm{-}SDO}\arrow[d, phantom, "{\rotatebox[origin=c]{90}{$=$}}" description] \arrow[r, phantom, "\subseteq" description]& \set{T}_{m,m}\set{\textrm{-}SD} \arrow[d, phantom, "{\rotatebox[origin=c]{90}{$=$}}" description]  \arrow[r, phantom, "\subseteq" description]& \set{TWSD\textrm{-}B} \arrow[r, phantom, "\subseteq" description]& \set{TWSD} &&&&&\\
% \set{SDO}\arrow[d, phantom, "{\rotatebox[origin=c]{90}{$=$}}" description] \arrow[r, phantom, "\subseteq" description]& \set{SD} \arrow[d, phantom, "{\rotatebox[origin=c]{90}{$=$}}" description]  &  &&&&&\\
% \set{D}_{m,m}\set{\textrm{-}SDO}\arrow[d, phantom, "{\rotatebox[origin=c]{270}{$\subseteq$}}" description] \arrow[r, phantom, "\subseteq" description]& \set{D}_{m,m}\set{\textrm{-}SD} \arrow[d, phantom, "{\rotatebox[origin=c]{270}{$\subseteq$}}" description] \arrow[r, phantom, "\subseteq" description]& \set{DWSD} &&&&&\\
% \set{T}_{m,n}\set{\textrm{-}SDO}\arrow[r, phantom, "=" description] &\cdots\arrow[r, phantom, "=" description] &\set{T}_{m,m+1}\set{\textrm{-}SDO}\arrow[r, phantom, "=" description] & \set{T}_{m,m}\set{\textrm{-}SDO}\arrow[r, phantom, "=" description] &\set{SDO}&\set{D}_{m,m}\set{\textrm{-}SDO}&\set{D}_{m,m+1}\set{\textrm{-}SDO}&&\set{T}_{m,n}\set{\textrm{-}SDO}
% \end{tikzcd}}
% \caption{General relationships}\label{fig:relationship-WSD-PSD-2}
% \end{figure}

\subsection{The relationships for a nonsingular pair}
Let matrices \(\matr{A}, \matr{B}\in\mathbf{symm}(\mathbb{R}^{m \times m}) \), and \( \matr{A} \) be nonsingular.
In \cref{sec:preliminary,sec:trans_based_SD,sec:TWSD,sec:decom_based_SD}, we have presented the following results to verify whether the set $\{\matr{A}, \matr{B}\}$ is SDO, SD, TWSD, TWSD-B or DWSD. 
Some of these results are proved in this paper, while others are from \cite{uhlig1973SimultaneousBlockDiagonalizationa,wang2021new}.

\begin{itemize}
\item \( \{ \matr{A}, \matr{B} \} \) is SDO if and only if $[\matr{A}, \matr{B}]_{\matr{I}_{m}}=\matr{0}_{m \times m}$  (\cref{lem:SDO-community}).
\item \( \{ \matr{A}, \matr{B} \} \) is SD if and only if the real Jordan normal form of \( \matr{A}^{-1} \matr{B} \) is diagonal (\cref{lem:nonsingular-pair-SD}).
\item \( \{ \matr{A}, \matr{B} \} \) is TWSD-B if and only if \( \matr{A}^{-1} \matr{B} \) has only real eigenvalues (\cref{thm:TWSD-B-nonsingular-pair}).
\item \( \{ \matr{A}, \matr{B} \} \) is DWSD if and only if \( \matr{A}^{-1} \matr{B} \) has only real eigenvalues (\cref{lem:DWSD-nonsingular-pair}).
\item \( \{ \matr{A}, \matr{B} \} \) is TWSD if \( \matr{A}^{-1} \matr{B} \) has a real eigenvalue (\cref{cor:TWSD-nonsingular}).
\end{itemize}
In particular, by \cref{thm:TWSD-B-nonsingular-pair} and \cref{lem:DWSD-nonsingular-pair}, we see that TWSD-B and DWSD are equivalent for a nonsingular pair of symmetric matrices. 
By \cref{exa:not-TWSD}, a nonsingular pair may not be TWSD.
The relationships between them can be summarized in \cref{figure-relathionships-WSD}(i).
% We use ``All'' to denote the set of nonsingular pairs.
% \begin{figure}[ht]
% \centering
% \includegraphics[width=0.7\linewidth]{figs/nonsingular-pair.pdf}
% \caption{The relationships for a nonsingular pair}\label{fig:nonsingular-pair-venn}
% \end{figure}

\subsection{The relationships for a singular pair}
Let \(\{\matr{A}, \matr{B}\}\subseteq\mathbf{symm}(\mathbb{R}^{m \times m}) \) be a singular pair. 
By \cref{thm:TWSD-B-singular-pair} and \cref{lem:DWSD-singular-pair}, the set $\{\matr{A}, \matr{B}\}$ is always TWSD-B and DWSD.
The relationships between them can be summarized in \cref{figure-relathionships-WSD}(ii).  %\cref{fig:singular-pair-venn}.
% \begin{figure}[ht]
% \centering
% \includegraphics[width=0.7\linewidth]{figs/singular-pair.pdf}
% \caption{Relationship Between Sets of Weakly Simultaneously Diagonalizable Matrices for Singular Pairs}\label{fig:singular-pair-venn}
% \end{figure}

\subsection{The relationships for a general nonsingular set}

Let the set $\mathcal{C}$ be as in \eqref{set_C}, and \(\matr{S} \in \operatorname{span}(\mathcal{C}) \) be nonsingular. 
In \cref{sec:preliminary,sec:trans_based_SD,sec:TWSD,sec:decom_based_SD}, we have presented the following results to verify whether the set $\mathcal{C}$ is SD, TWSD, TWSD-B or DWSD.

\begin{itemize}
\item \( \mathcal{C} \) is SD if and only if the Jordan normal form of \( \matr{S}^{-1}\matr{A}_i \) is diagonal for all \( 1 \leq i \leq L\), and 
$[\matr{A_{i}}, \matr{A}_{j}]_{\matr{S}}=\matr{0}$ for all \( 1 \leq i, j \leq L\) (\cref{lem:SD-multiple-nonsingular}).
\item If \( \mathcal{C} \) is DWSD, then \( \matr{S}^{-1}\matr{A}_i \) has only real eigenvalues for all $1\leq i\leq   L$, and $[\matr{A}_i, \matr{A}_j]_{\matr{S}}=\matr{0}$ for all \( 1 \leq i\neq  j \leq L\)  (\cref{lem:DWSD-general-necessary-condition}).
\item If \( \mathcal{C} \) is TWSD-B, then \( \matr{S}^{-1} \matr{A}_i \) has only real eigenvalues and \( \matr{S}^{-1} \matr{A}_i - \matr{S}^{-1} \matr{A}_j \) is nilpotent for all \( 1 \leq i\neq j \leq L\) (\cref{thm:TWSD-B-multiple-necessary}).
\item \( \mathcal{C} \) is TWSD-B, if \( \matr{S}^{-1} \matr{A}_{i}\) has only real eigenvalues for all $1 \leq i \leq L$, \( [\matr{A}_i, \matr{A}_j]_{\matr{S}} = \matr{0} \) for all \( 1 \leq i\neq j \leq L\) and there exists \( l \) such that the there are no Jordan blocks with the same eigenvalues and sizes in the real Jordan normal form of \( \matr{S}^{-1} \matr{A}_l \)  (\cref{thm:TWSD-B-multiple-sufficient-condition}).
\end{itemize}
In particular, by \cref{lem:DWSD-general-necessary-condition} and \cref{thm:TWSD-B-multiple-sufficient-condition}, we see that, if $\mathcal{C}$ is nonsingular and DWSD, then it is TWSD-B. 
% Above all, if \( \mathcal{C} \) is nonsingular, we have the following relationships
% \[
% \set{SDO} \subseteq  \set{SD} \subseteq \set{DWSD} \subseteq \set{TWSD}\textrm{-}\set{B} \subseteq \set{TWSD},
% \]
% which can be summarized in
The relationships between them can be summarized in \cref{figure-relathionships-WSD}(iii).

\subsection{The relationships for a general positive definite set}

Let the set $\mathcal{C}$ be as in \eqref{set_C}, and \(\matr{S} \in \operatorname{span}(\mathcal{C}) \) be positive definite. 
Then, for the set $\mathcal{C}$, TWSD-B and DWSD both reduce to SD by \cref{thm:PD-equivalent} and \cref{lem:PD-DWSD-equiv-SD}. 
Now, in \cref{exa:PD-TWSD-not-SD}, we will show that $\set{SD}\subsetneqq\set{TWSD}$.

\begin{example}\label{exa:PD-TWSD-not-SD}
Let \( \mathcal{C} = \{ \matr{A}_1, \matr{A}_2, \matr{A}_3 \}\subseteq\mathbf{symm}(\mathbb{R}^{3 \times 3}) \) with
\[
\matr{A}_1= \matr{I}_3, \ \
\matr{A}_2 = \begin{bmatrix}
  1 & 0 & 0\\
  0 & 1 & 0\\
  0 & 0 & -1
\end{bmatrix}, \ \
\matr{A}_3 = \begin{bmatrix}
  1 & 0 & 0\\
  0 & 0 & 1\\
  0 & 1 & 0
\end{bmatrix}.
\]
Then \( \mathcal{C} \) has a positive definite pencil since \( \matr{A}_1 \succ 0\). Note that \( \matr{A}_2 \) does not commute with \( \matr{A}_3 \). \( \mathcal{C} \) is not SD by \cref{lem:SDO-community}. However, it is TWSD by \cref{thm:TWSD-block-TWSD-B-suffic} since they are block diagonal matrices \( \matr{A}_i = \Diag{1, \matr{\tilde{A}}_i}\) for all $1\leq i\leq 3$, and the set \( \{ [1], [1], [1] \}\subseteq\RR^{1\times 1} \) is SD.
\end{example}

% Therefore, for the set \( \mathcal{C} \) with a positive definite pencil, we have the following relationships:
% \[
% \set{SD} = \set{TWSD}\textrm{-}\set{B} = \set{DWSD} \subsetneqq  \set{TWSD},
% \]
% which can be summarized in
The relationships between them can be summarized in \cref{figure-relathionships-WSD}(iv).

% add example for TWSD
\begin{figure}[htb!]
\begin{minipage}[b]{0.45\linewidth}
  \centering
\begin{tikzpicture}
  % All pairs
  \node[draw,
  rectangle,
  fill=color6,
  minimum width=6cm,
  minimum height=5cm,
  label={[shift={(0,-0.65)}]Set of Nonsingular Pairs}] at (0,0.25){};
\node [draw,
    ellipse,
    fill=color5,
    minimum width =5.7cm,
    minimum height =4.3cm,
label={[shift={(0.0,-0.63)}]\(\set{TWSD}\)}
    ] at (-0.05,0){};
\node [draw,
    ellipse,
    fill=color4,
    minimum width =5cm,
    minimum height =3.5cm,
label={[shift={(0.0,-1)}]\(\set{TWSD}\textrm{-}\set{B} = \set{DWSD}\)}
    ] at (-0.3,-0.15){};
\node [draw,
    ellipse,
    fill=color2,
    minimum width =3cm,
    minimum height =2.5cm,
label={[shift={(0.0,-0.7)}]\(\set{SD}\)}
    ] at (-0.3,-0.6){};
\node [draw,
    ellipse,
    fill=color1,
    minimum width =2cm,
    minimum height =1cm,
label={[shift={(0.0,-0.8)}]\(\set{SDO}\)}
    ] at (-0.5,-0.75){};
\end{tikzpicture}
% \vspace{1.5cm}
 \centerline{(i) For a nonsingular pair.}\medskip
\end{minipage}
%\hfill
\begin{minipage}[b]{0.45\linewidth}
  \centering
\begin{tikzpicture}
  % All pairs
  \node[draw,
  rectangle,
  fill=color6,
  minimum width=6cm,
  minimum height=5cm,
  label={[shift={(0,-1)}]Set of Nonsingular Pairs}] at (0,0.25){};
\node[] at (0,1.5) {\( = \set{TWSD} = \set{TWSD}\)-\(\set{B} = \set{DWSD} \)};
\node [draw,
    ellipse,
    fill=color2,
    minimum width =3cm,
    minimum height =2.5cm,
label={[shift={(0.0,-0.7)}]\(\set{SD}\)}
    ] at (-0.3,-0.6){};
\node [draw,
    ellipse,
    fill=color1,
    minimum width =2cm,
    minimum height =1cm,
label={[shift={(0.0,-0.8)}]\(\set{SDO}\)}
    ] at (-0.5,-0.75){};
\end{tikzpicture}
\centerline{(ii) For a singular pair.}\medskip
\end{minipage}
\begin{minipage}[b]{0.45\linewidth}
 \centering
\begin{tikzpicture}
  % All pairs
  \node[draw,
  rectangle,
  fill=color6,
  minimum width=6cm,
  minimum height=5cm,
  label={[shift={(0,-0.65)}]Set of Nonsingular Sets}] at (0,0.25){};
\node [draw,
    ellipse,
    fill=color5,
    minimum width =5.7cm,
    minimum height =4.3cm,
label={[shift={(0.0,-0.63)}]\(\set{TWSD}\)}
    ] at (-0.05,0){};
\node [draw,
    ellipse,
    fill=color4,
    minimum width =4.5cm,
    minimum height =3.5cm,
label={[shift={(0.0,-0.7)}]\(\set{TWSD}\textrm{-}\set{B}\)}
    ] at (-0.3,-0.15){};
\node [draw,
    ellipse,
    fill=color3,
    minimum width =4.3cm,
    minimum height =2.8cm,
label={[shift={(0.0,-0.7)}]\(\set{DWSD}\)}
    ] at (0.2,-0.5){};
\node [draw,
    ellipse,
    minimum width =4.5cm,
    minimum height =3.5cm,
    ] at (-0.3,-0.15){};
\node [draw,
    ellipse,
    fill=color2,
    minimum width =3cm,
    minimum height =1.8cm,
label={[shift={(0.0,-0.6)}]\(\set{SD}\)}
    ] at (-0.3,-0.6){};
\node[] at (-2.2,0) {?};
\node[] at (2.1,-0.5) {?};
\node [draw,
    ellipse,
    fill=color1,
    minimum width =2cm,
    minimum height =1cm,
label={[shift={(0.0,-0.8)}]\(\set{SDO}\)}
    ] at (-0.5,-0.75){};
\end{tikzpicture}
 \centerline{(iii) For a general nonsingular set.}\medskip
\end{minipage}
%\hfill
\begin{minipage}[b]{0.45\linewidth}
 \centering
\begin{tikzpicture}
  \node[draw,
  rectangle,
  fill=color6,
  minimum width=6cm,
  minimum height=5cm,
  label={[shift={(0,-0.65)}]{\small Set of General Positive Definite Sets}}] at (0,0.25){};
\node [draw,
    ellipse,
    fill=color5,
    minimum width =5.7cm,
    minimum height =4.3cm,
label={[shift={(0.0,-0.7)}]{\small \(\set{TWSD}\)}}
    ] at (-0.05,0){};
\node [draw,
    ellipse,
    fill=color4,
    minimum width =5cm,
    minimum height =3.1cm,
label={[shift={(0.0,-1)}]\(\set{TWSD}\textrm{-}\set{B} = \set{DWSD}\)}
    ] at (-0.3,-0.15){};
\node [draw,
    ellipse,
    fill=color2,
    minimum width =3cm,
    minimum height =1.8cm,
label={[shift={(0.0,-0.6)}]\(\set{SD}\)}
    ] at (-0.3,-0.6){};
\node [draw,
    ellipse,
    fill=color1,
    minimum width =2cm,
    minimum height =1cm,
label={[shift={(0.0,-0.8)}]\(\set{SDO}\)}
    ] at (-0.5,-0.75){};
\end{tikzpicture}
% \vspace{1.5cm}
 \centerline{(iv) For a general positive definite set.}\medskip
\end{minipage}
\caption{Relationships under different assumptions.}
\label{figure-relathionships-WSD}
\end{figure}
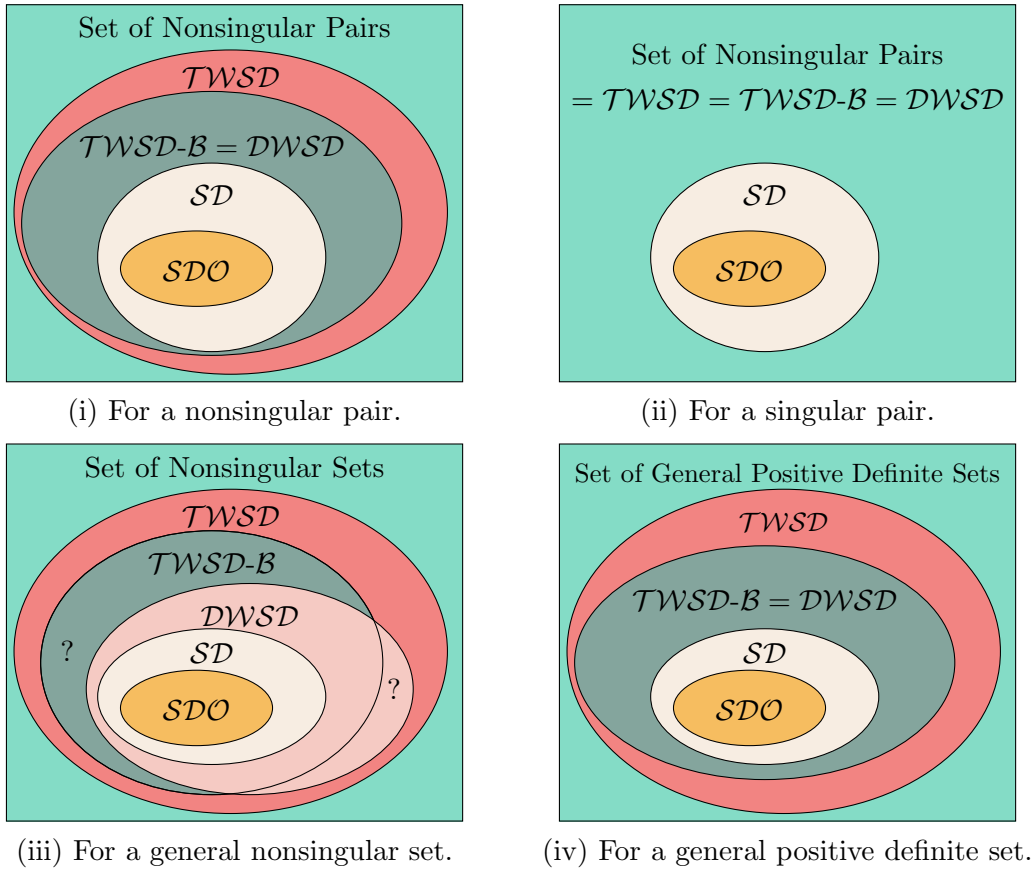
\section{Applications to quadratically constrained quadratic programming}\label{sec:application-QCQP}

In this section, we mainly consider the following \emph{quadratically constrained quadratic programming} (QCQP) model
\begin{align}\label{eq:problem-QCQP-gene}
\begin{array}{ll}
\min & \frac{1}{2}\vect{x}^\T \matr{A}_0 \vect{x} + \vect{a}_{0}^{\T}\vect{x} \\
\mbox{s.t.} & \frac{1}{2}\vect{x}^\T \matr{A}_{\ell} \vect{x} + \vect{a}_{\ell}^{\T}\vect{x} + c_{\ell} \le 0,\, 1\leq\ell\leq L,
\end{array}
\end{align}
where $\vect{x}\in\RR^{m}$, $\matr{A}_{\ell} \in \textbf{symm}(\RR^{m\times m})$ and $\vect{a}_{\ell} \in \RR^{m}$ for $0\leq\ell\leq L$, and $c_{\ell}\in\RR$ for $1\leq\ell\leq L$.
The matrices $\matr{A}_{\ell}$ here are not necessarily definite or semi-definite, and thus this is not necessarily a convex optimization problem. 
We will discuss the applications of TWSD and TWSD-B to problem \eqref{eq:problem-QCQP-gene}.

% \subsection{Solve model \eqref{eq:problem-QCQP-gene} using SD matrices}

% \subsubsection{$L=1$ using two SD matrices}
% When \( A_0, A_1 \) are SD, the quadratic functions \( x^{\T}A_0x, x^{\T}A_1x \)  In this case, \cite{ben2014hidden} shows the hidden convexity of the problem, that is, it is equivalent to a convex problem. Moreover, the dual of it is a conic quadratic (CQ) one. Since CQ is generally more tractable then SDP, the problem can be solved much more efficiently with the reformulation.

% \subsubsection{$L>1$ using a finite number of SD matrices}

% In \cite{jiang2016simultaneous}, \eqref{eq:problem-QCQP-gene} is considered when \( A_i, i=0,1, \ldots , m \) are SD and \( a_i = 0, i=0,1, \ldots,m \) (note that we can always transform the inhomogenous one to a homogenous one by introducing an additional vriable \( x_{n+1} = \pm 1 \), which is equivalent to \( x_{n+1}^2 = 1 \)), the problem is then equivalent to a linear programming:
% \begin{align}\label{eq:problem-QCQP-gene-SD}
% \begin{array}{ll}
% \min & \Diag{A_0}^{\T}y \\
% \mbox{s.t.} & \Diag{A_{\ell}}^{\T}y + c_{\ell} \le 0,\, 1\leq\ell\leq L,\\
% & y \geq 0
% \end{array}
% \end{align}

\subsection{Solve problem \eqref{eq:problem-QCQP-gene} using TWSD matrices}\label{subsec:qcqp_twsd}

\subsubsection{Case $\vect{a}_{\ell}=\vect{0}$}
In model \eqref{eq:problem-QCQP-gene}, if $\vect{a}_{\ell}=\vect{0}$ for $0\leq \ell\leq L$, then it can be represented as
\begin{align}\label{eq:problem-QCQP-gene-u-0}
\begin{array}{ll}
\min & \frac{1}{2}\vect{x}^\T \matr{A}_0 \vect{x} \\
\mbox{s.t.} & \frac{1}{2}\vect{x}^\T \matr{A}_{\ell} \vect{x} + c_{\ell} \le 0,\, 1\leq\ell\leq L.
\end{array}
\end{align}
% where $\vect{x}\in\RR^{m}$, $\matr{A}_{\ell} \in \textbf{symm}(\RR^{m\times m})$ and $c_{\ell}\in\RR$ for $0\leq\ell\leq L$.
In this case, if the set $\{\matr{A}_{\ell}\}_{0\leq \ell \leq L}$ is TWSD, then for any $\epsilon>0$, we may find a nonsingular linear transformation $\matr{P}_{k} \in\mathbf{SL}_{m}(\RR)$ such that $\|\mathbf{offdiag}(\matr{P}_{k}^\T\matr{A}_{\ell}\matr{P}_{k})\|<\epsilon$.
Denote $\vect{y}=\matr{P}_{k}^{-1}\vect{x}$ and $\matr{W}^{(\ell)}_{k} = \matr{P}_{k}^\T\matr{A}_\ell\matr{P}_{k}$. It follows that
\begin{align}\label{eq:diag-offdiag}
\vect{x}^\T \matr{A}_\ell \vect{x}=\vect{y}^\T \matr{W}^{(\ell)}_{k} \vect{y}=\vect{y}^\T \mathbf{diag}(\matr{W}^{(\ell)}_{k}) \vect{y}+\vect{y}^\T \mathbf{offdiag}(\matr{W}^{(\ell)}_{k}) \vect{y},
\end{align}
for $0\leq \ell\leq L$. If $\epsilon>0$ is small enough, we drop the off-diagonal elements in \eqref{eq:diag-offdiag}, and then formulate a new optimization problem
\begin{align}\label{eq:problem-QCQP-gene-u-0-variant}
\begin{array}{ll}
\min & \frac{1}{2}\vect{y}^\T \mathbf{diag}(\matr{W}^{(0)}_{k}) \vect{y} \\
\mbox{s.t.} & \frac{1}{2}\vect{y}^\T \mathbf{diag}(\matr{W}^{(\ell)}_{k}) \vect{y} + c_{\ell} \le 0,\, 1\leq\ell\leq L,
\end{array}
\end{align}
where $\vect{y}\in\RR^{m}$, $\matr{W}^{(\ell)}_{k} \in \textbf{symm}(\RR^{m\times m})$ and $c_{\ell}\in\RR$ for $0\leq\ell\leq L$.
Let $\lambda^{(\ell)}_{i}=(\matr{W}^{(\ell)}_{k})_{ii}$ for $1\leq i\leq m$ and $0\leq \ell\leq L$. Denote $u_i=y_i^2$ for $1\leq i\leq m$. Then problem \eqref{eq:problem-QCQP-gene-u-0-variant} can be further represented as a linear programming problem
\begin{align}\label{eq:problem-QCQP-gene-u-0-variant-linear}
\begin{array}{ll}
\min & \frac{1}{2}\sum_{i=1}^{m}\lambda^{(0)}_i u_i \\
\mbox{s.t.} & \frac{1}{2}\sum_{i=1}^{m}\lambda^{(\ell)}_i u_i + c_{\ell} \le 0,\, 1\leq\ell\leq L,\\
& u_i\geq 0,\ 1\leq i\leq n.
\end{array}
\end{align}

\subsubsection{Case $\vect{a}_{\ell}\neq\vect{0}$}
In model \eqref{eq:problem-QCQP-gene}, if $\vect{a}_{\ell}$ is not necessarily equal to $\vect{0}$ for $0\leq \ell\leq L$, we can still reformulate it as a homogeneous QCQP by introducing new variables \(x_{m+1} = \pm 1\). Then \eqref{eq:problem-QCQP-gene} is equivalent to the following problem:
\begin{align}\label{eq:problem-QCQP-gene-after-homo}
\begin{array}{ll}
\min & \frac{1}{2}\vect{x}^\T \matr{A}_0 \vect{x} + \vect{a}_{0}^{\T}\vect{x}x_{m+1} \\
\mbox{s.t.} & \frac{1}{2}\vect{x}^\T \matr{A}_{\ell} \vect{x} + \vect{a}_{\ell}^{\T}\vect{x}x_{m+1} + c_{\ell}x_{m+1}^2 \le 0,\, 1\leq\ell\leq L.
\end{array}
\end{align}
Let \( \bar{\matr{A}}_i = \begin{pmatrix}\matr{A}_i & \vect{a}_i\\ \vect{a}_i^{\top} & c_i\end{pmatrix}\) for \( 0 \leq i \leq L \), \( \bar{\matr{A}}_{L+1} = \begin{pmatrix}\matr{0}_{m \times m} & 0\\ 0 & 1\end{pmatrix} \) and \( \bar{\vect{x}}^{\T} = (\vect{x}^{\T}, 1)^{\T} \). Then \eqref{eq:problem-QCQP-gene-after-homo} can be rewritten as:
\begin{align}\label{eq:problem-QCQP-gene-after-homo-2}
\begin{array}{ll}
\min & \frac{1}{2}\vect{\bar{x}}^\T \matr{\bar{A}}_0 \vect{\bar{x}}\\
  \mbox{s.t.} & \frac{1}{2}\vect{\bar{x}}^\T \matr{\bar{A}}_{\ell} \vect{\bar{x}} \le 0,\, 1\leq\ell\leq L,\\
                &  \vect{ \bar{x}}^{\top} \matr{ \bar{A}}_{L+1} \vect{ \bar{x}} = 1,
\end{array}
\end{align}
which is a homogeneous QCQP. 
Similarly, if the set \(\{\matr{\bar{A}}_i\}_{1\leq \ell\leq L+1} \) is TWSD, we can also approximately reformulate \eqref{eq:problem-QCQP-gene-after-homo-2} as a linear programming like \eqref{eq:problem-QCQP-gene-u-0-variant-linear}.
% then for any given precision $\epsilon>0$ we similarly find $\matr{P}_{k} \in\mathbf{SL}(n)$, and denote $\vect{y}=\matr{P}_{k}^{-1}\vect{x}$ and $\matr{W}^{(\ell)}_{k} = \matr{P}_{k}^\T\matr{A}_\ell\matr{P}_{k}$.
% Let $\lambda^{(\ell)}_{i}=(\matr{W}^{(\ell)}_{k})_{ii}$, $u_i=y_i^2$ and $\vect{b}_{\ell}=\matr{P}_{k}^\T\vect{a}_\ell$ for $1\leq i\leq n$ and $0\leq \ell\leq L$.
% Then, after dropping the off-diagonal elements, we formulate the following optimization model:
% \begin{align}\label{eq:problem-QCQP-gene-u-0-variant-linear-uneq-0}
% \begin{array}{ll}
% \min & \frac{1}{2}\sum_{i=1}^{n}\lambda^{(0)}_i u_i + \vect{b}_{0}^{\T}\vect{y} \\
% \mbox{s.t.} & \frac{1}{2}\sum_{i=1}^{n}\lambda^{(\ell)}_i u_i + \vect{b}_{\ell}^{\T}\vect{y}  + c_{\ell} \le 0,\, 1\leq\ell\leq L,\\
% & u_i=y_i^2,\ 1\leq i\leq n.
% \end{array}
% \end{align}
% Now we try to solve model \eqref{eq:problem-QCQP-gene-u-0-variant-linear-uneq-0} ...
\subsection{Solve problem \eqref{eq:problem-QCQP-gene} using TWSD-B matrices}
In \Cref{subsec:qcqp_twsd}, we directly drop the off-diagonal elements in \eqref{eq:diag-offdiag}, and approximately reformulate the QCQP model \eqref{eq:problem-QCQP-gene} as a linear programming, in which the error may be difficult to control even when \(\epsilon>0\) is very small.
In this subsection, we mainly consider the following homogeneous QCQP model with a single constraint: 
\begin{equation}\label{eq:problem-QCQP-gene-one-constraint}\tag{\( P \)}
\begin{array}{ll}
\min & \vect{x}^\T \matr{B} \vect{x}\\
\mbox{s.t.} & \vect{x}^\T \matr{A} \vect{x} \le b,
\end{array}
\end{equation}
where $\vect{x}\in\RR^{m}$ and \(\matr{A}, \matr{B}\in\textbf{symm}(\RR^{m\times m})\). 
We will consider the case where \( \matr{A} \) is nonsingular.
Without loss of generality, as in \cref{lem:Uhlig-canonical-nonsingular}, we assume \( \matr{A} \) and \( \matr{B} \) are in the following form:
\begin{equation}
  \label{eq:QCQP-A-B}
 \begin{aligned}
\matr{A} &= \Diag{\sigma_1 \matr{E}(m_{1}), \ldots, \sigma_r\matr{E}(m_{r})},\\
\matr{B} &= \Diag{\sigma_1 \matr{E}(m_{1})\matr{J}(\lambda_{1}, m_{1}), \ldots, \sigma_r\matr{E}(m_{r})\matr{J}(\lambda_{r}, m_{r})}.
\end{aligned}
\end{equation}
We will show that the approximation method is stable if \eqref{eq:problem-QCQP-gene-one-constraint} is bounded from below and \( \matr{A} \) is nonsingular. Under Slater condition, it has been shown in \cite{jiang2018SOCPReformulationGeneralized} that \( \matr{A} \) and \( \matr{B} \) must satisfy some additional requirements if \eqref{eq:problem-QCQP-gene-one-constraint} is bounded from below. 
\begin{lemma}[{\cite[Theorem 6]{jiang2018SOCPReformulationGeneralized}}]\label{lem:jiang-jordan-block-size}
  If problem \eqref{eq:problem-QCQP-gene-one-constraint} has an optimal value bounded from below and Slater condition holds, then:\\
  (i) \(\lambda_i \in \mathbb{R}\), for \(1\leq i\leq r\);\\
  (ii) \(m_i \leq 2\), for \(1\leq i\leq r\);\\
  (iii) If $m_i = 2$ for some index \(i\), then \(\sigma_i = 1\) and \(\lambda_i \leq 0\).
\end{lemma}

Therefore, in this case, the set \( \{ \matr{A}, \matr{B} \} \) in \eqref{eq:QCQP-A-B} is TWSD-B by \cref{thm:TWSD-B-nonsingular-pair}.
Let 
$$\matr{R}_k = \Diag{\matr{R}_k(m_1) , \ldots , \matr{R}_k(m_r)},$$ 
and \( \matr{\hat{B}}_k = \matr{R}_k^{\T} \matr{B} \matr{R}_k \) for \( k \geq 1 \), where \( \matr{R}_k(m_i) \) is defined as in \eqref{def:matrix-G-R}. Then
\begin{equation}\label{eq:limit-R-A-R}
\begin{aligned}
&  \matr{R}_k^{\top} \matr{A} \matr{R}_k = \Diag{\sigma_1 \matr{E}(m_1) , \ldots , \sigma_r \matr{E}(m_r)} = \matr{A},\ \forall k \geq 1,\\
& \matr{\hat{B}} \eqdef \lim_{k \to \infty} \matr{R}_k^{\top} \matr{B} \matr{R}_k = \lim_{k\to \infty} \matr{\hat{B}}_k = \Diag{\sigma_1 \lambda_1 \matr{E}(m_1) , \ldots , \sigma_r \lambda_r \matr{E}(m_r)}.
\end{aligned}
\end{equation}
Let \( \vect{y}_k = \matr{R}_k\vect{x} \in \RR^m\) for any fixed \( k \). Then by equation \eqref{eq:limit-R-A-R}, problem \eqref{eq:problem-QCQP-gene-one-constraint} is equivalent to:
\begin{equation}\label{eq:problem-QCQP-gene-one-constraint-Pk}\tag{\( P^{(k)} \)}
  \begin{array}{ll}
    \min & \vect{y}_k^{\top}\matr{\hat{B}_k}\vect{y}_k \\
    \mbox{s.t.} & \vect{y}_k^{\top}\matr{A}\vect{y}_k \leq b.
  \end{array}
\end{equation}
In other words, they always have the same optimal value, and \(\vect{x}^{*}\) is an optimal solution of \eqref{eq:problem-QCQP-gene-one-constraint} if and only if \( \matr{R}_k \vect{x}^{*} \) is an optimal solution of \eqref{eq:problem-QCQP-gene-one-constraint-Pk}. 
When \( k\to\infty\), we obtain the following ``limit'' problem:
\begin{equation}\label{eq:problem-QCQP-gene-one-constraint-Pinf}\tag{\(P^{(\infty)}\)}
\begin{array}{ll}
    \min & \vect{y}^{\top}\matr{\hat{B}}\vect{y}\\
    \mbox{s.t.} & \vect{y}^{\top} \matr{A}\vect{y} \leq b.
  \end{array}
\end{equation}
Note that problems \eqref{eq:problem-QCQP-gene-one-constraint}, \eqref{eq:problem-QCQP-gene-one-constraint-Pk} and \eqref{eq:problem-QCQP-gene-one-constraint-Pinf} have the same feasible region but different objective functions, and \(\lim\limits_{k \to \infty} \vect{y}^{\top} \matr{\hat{B}}_k\vect{y} = \vect{y}^{\top} \matr{\hat{B}y}\) for any fixed \(\vect{y} \in \RR^m\), that is, the objective functions of \eqref{eq:problem-QCQP-gene-one-constraint-Pk} converge pointwise to the objective function of \eqref{eq:problem-QCQP-gene-one-constraint-Pinf}. For general optimization problem over a noncompact feasible set, this doesn't imply that the optimal value of \eqref{eq:problem-QCQP-gene-one-constraint} is equal to that of  \eqref{eq:problem-QCQP-gene-one-constraint-Pinf}. 
However, under some conditions, it is not difficult to prove the following result, and we omit the detailed proof here. 

\begin{lemma}\label{thm:condi_QCQP_appli}
Let \(f_k(\vect{y}):D\to\RR\) be a sequence of functions for \(k\geq 1\), where \(D\subseteq\RR^{m}\) is a noncompact set. Let \(f(\vect{y}):D\to\RR\) be a function.
If \\
(i) \(\min f_k(\vect{y})=c\), and the minimum $c$ is attainable for all $k\geq 1$, \\
(ii) \(f_k(\vect{y})\) has at most \(M\) minimizers,\\
(iii) for all \(\vect{y}\in D\), the sequence \(f_k(\vect{y})\) is decreasing or increasing, and converge to \(f(\vect{y})\),\\
then \(\min f(\vect{y}) = c\) and the minimum is also attainable.
\end{lemma}

Let \(f(\vect{y}) = \vect{y}^{\top}\matr{\hat{B}}\vect{y}\) and \(f_k(\vect{y}) = \vect{y}^{\top} \matr{\hat{B}}_k\vect{y}\) for \( k \geq 1 \). It is easy to check \(f_k(\vect{y})\) is decreasing for all \( \vect{y} \) in the feasible region. Thus, we have the following corollary by \cref{thm:condi_QCQP_appli}.
\begin{corollary}
If \eqref{eq:problem-QCQP-gene-one-constraint} is bounded from below and attainable with finite minimizers, then \eqref{eq:problem-QCQP-gene-one-constraint-Pinf} has the same optimal value as \eqref{eq:problem-QCQP-gene-one-constraint}, and the minimum is also attainable.
\end{corollary}

%In \cite{jiang2018SOCPReformulationGeneralized}, the possible sizes of the blocks in \eqref{eq:QCQP-A-B} are investigated when \eqref{eq:problem-QCQP-gene-one-constraint} is bounded from below and Slater condition holds. 
In fact, based on \cref{lem:jiang-jordan-block-size} (\cite[Theorem 6]{jiang2018SOCPReformulationGeneralized}), 
we can easily prove a more general result.
%that the the optimal value of \eqref{eq:problem-QCQP-gene-one-constraint} is equal to that of \eqref{eq:problem-QCQP-gene-one-constraint-Pinf} in this case. 
%The following lemmas \todo \fin are adapted from  \cite{jiang2018SOCPReformulationGeneralized}.

\begin{theorem}
Let \( \matr{A} \) and \( \matr{B} \) be as in \eqref{eq:QCQP-A-B}. Then the optimal value of \eqref{eq:problem-QCQP-gene-one-constraint} is always equal to the optimal value of \eqref{eq:problem-QCQP-gene-one-constraint-Pinf} if \eqref{eq:problem-QCQP-gene-one-constraint} is bounded from below.
\end{theorem}
\begin{proof}
Denote the optimal value of \eqref{eq:problem-QCQP-gene-one-constraint} and \eqref{eq:problem-QCQP-gene-one-constraint-Pinf} by \( v \) and \( v^{(\infty)} \), respectively.
If Slater condition does not hold, then \( \matr{A} \) is positive definite and \( b = 0 \). In this case, \( v = v^{(\infty)} = 0 \).
If Slater condition holds, combining \cref{lem:jiang-jordan-block-size} with \eqref{eq:QCQP-A-B}, we have
\begin{align*}
  \matr{P}_k^{\top}\matr{A}\matr{P}_k = \matr{A} &= \Diag{\sigma_1 , \ldots , \sigma_l, \matr{E}(2) , \ldots , \matr{E}(2) }, k \geq 1,\\
  \matr{B} &= \Diag{\sigma_1 \lambda_1, \ldots, \sigma_l \lambda_l, \matr{E}(2) \matr{J}(\lambda_{l+1}, 2) , \ldots ,  \matr{E}(2) \matr{J}(\lambda_r, 2)},\\
  \matr{\hat{B}} &= \Diag{\sigma_1 \lambda_1, \ldots, \sigma_l \lambda_l,  \lambda_{l+1} \matr{E}(2) , \ldots , \lambda_r\matr{E}(2)},
\end{align*}
where \( l \) is the number of one by one blocks in \eqref{eq:QCQP-A-B}.
  Let \( \vect{x} \) be any vector satisfying \( \vect{x}^{\top}\matr{A}\vect{x} \leq b \). Then it is also a feasible point of the problem \eqref{eq:problem-QCQP-gene-one-constraint-Pk} for any \( k \). Since the optimal value of \eqref{eq:problem-QCQP-gene-one-constraint-Pk} is also \( v \), we have that \( \vect{x}^{\top} \matr{\hat{B}}_k\vect{x} \geq v \), which implies
  \[
    \vect{x}^{\top} \matr{\hat{B}} \vect{x} = \lim_{k \to \infty} \vect{x}^{\top}\matr{P}_k^{\top}\matr{B}\matr{P}_k\vect{x} \geq v.
  \]
  Thus, we have \( v^{(\infty)} \geq v \).
On the other hand, since 
  \[ \vect{z}^{\top} \matr{E}(2) \matr{J}(\lambda, 2) \vect{z} = \lambda \vect{z}^{\top} \matr{E}(2) \vect{z} + z_2^2 \geq \lambda \vect{z}^{\top} \matr{E}(2) \vect{z}, \]
 for any \( \vect{z} = (z_1, z_2)^{\top} \in \RR^2 \) and \(\lambda \in \mathbb{R}\), we always have that 
  \begin{align*}
    \vect{x}^{\top}\matr{B}\vect{x} &=
    \vect{x}^{\top} \Diag{\sigma_1 \lambda_1, \ldots, \sigma_l \lambda_l, \matr{E}(2) \matr{J}(\lambda_{l+1}, 2) , \ldots ,  \matr{E}(2) \matr{J}(\lambda_r, 2)} \vect{x} \\
    & \geq  \vect{x}^{\top} \Diag{\sigma_1 \lambda_1, \ldots, \sigma_l \lambda_l, \lambda_{l+1} \matr{E}(2) , \ldots ,  \lambda_r\matr{E}(2)} \vect{x}  = \vect{x}^{\top} \matr{\hat{B}} \vect{x}.
\end{align*}
It follows that \( v \geq v^{(\infty)} \). The proof is complete.
\end{proof}

As shown in the proof of \cref{thm:TWSD-B-nonsingular-pair}, the set \( \{ \matr{A}, \matr{\hat{B}} \}  \) is SD. Thus, problem \eqref{eq:problem-QCQP-gene-one-constraint-Pinf} can be reformulated as a linear programming, and it has a closed-form solution as well.

\section{Applications to independent component analysis}\label{sec:application-approximate-diagonalization}

As generalizations of the eigenvalue decomposition of a single symmetric matrix, the SDO and SD properties of multiple symmetric matrices can be seen as finding a set of basis on which they all have simple representations. In \emph{independent component analysis}(ICA) \cite{Cardoso93:JADE,Como94:sp,Como10:book}, since multiple symmetric matrices are often not SDO or SD, the \emph{approximate simultaneous diagonalization}(ASD) \cite{Cardoso93:JADE,Como94:sp,LUC2017globally,LUC2018,li2020gradient,ULC2019} of them has become an important approach to solve ICA.
Let the set $\mathcal{C}$ be as in \eqref{set_C}. 
This approach is to find a nonsingular or orthogonal matrix $\matr{P}\in\RR^{m\times m}$ to minimize the off-diagonal elements, \emph{i.e.}, \begin{align}\label{eq:offdiag_zero}
\min_{\matr{P}}g(\matr{P})\eqdef \sum_{i=1}^{L}\|\textbf{offdiag}(  \matr{P}^{\T}\matr{A}_i \matr{P})\|^2. 
\end{align} 
When the feasible set of \( \matr{P} \) is compact, \emph{e.g.}, the special orthogonal group $\SON_{m}$, the algorithms to solve problem \eqref{eq:offdiag_zero} have been extensively studied; see for example \cite{LUC2017globally,LUC2018,ULC2019}. 
When the feasible set of \( \matr{P} \) is not compact, \emph{e.g.}, the special linear group $\SL_{m}(\RR)$, several algorithms have been developed as well \cite{afsari2006simple,li2020gradient}.
% it is 
% %possible that the minimum of the optimization problem is unattinable.
% %However, it is 
% also an important approach to solve ICA. 

The new notions in \Cref{table-example-3-0} we studied in this paper can be regarded as the ``weakly joint eigenvalue decomposition" or ``projectively joint eigenvalue decomposition" of multiple symmetric matrices, and have much broader scopes than SD and SDO. 
Therefore, a natural question is whether these new notions can be used to replace the approach \eqref{eq:offdiag_zero} to serve ICA. To answer this question, in this section, we take the \(\mathbf{D}_{m, n} \)-SD property as an example to illustrate its application to ICA. 

The basic linear \emph{blind source separation} (BSS) model \cite{herault1986space,comon1994tensor} was formulated as  
\begin{equation}\label{eq:BSS-linear}
\vect{x}(t)=\matr{P}^{\T}\cdot\vect{s}(t)+\vect{n}(t),
\end{equation}
where $\vect{x}(t)\in\RR^{m}$ is the \emph{observation signal vector},  $\vect{s}(t)\in\RR^{n}$ is the \emph{source signal vector}, $\vect{n}(t)\in\RR^{m}$ is a \emph{noise} and $\matr{P}\in\RR^{n\times m}$ is the linear \emph{mixing matrix}. The goal of BSS problem in \eqref{eq:BSS-linear} is to estimate the source signal vector $\vect{s}(t)\in\RR^{n}$ under two assumptions:
\begin{itemize}
\item the components of $\vect{s}(t)$ are statistically independent. 
\item at most one component of $\vect{s}(t)$ is Gaussian.
\end{itemize}

For simplicity, in this section, we neglect the noise item $\vect{n}(t)$ in model  \eqref{eq:BSS-linear}. 
% It is clear that the model in \eqref{eq:BSS-linear} has an inherent ambiguity about the scale and order of the source signal vectors. 
In 1994, Comon \cite{Como94:sp,comon1994tensor} proved the following important result, which has been the theoretical core for almost all ICA algorithms.

\begin{theorem}
Suppose that we find a matrix $\matr{B}\in\RR^{n\times m}$ such that the components of $\vect{y}(t)=\matr{B}\cdot\vect{x}(t)$ are statistically independent. 
Then $\matr{B}\matr{P}^{\T}$  is \emph{essentially diagonal}, \emph{i.e.}, there exist an invertible diagonal matrix $\matr{\Lambda}$ and a permutation
matrix $\matr{\Pi}$ such that $\matr{B}\matr{P}^{\T}=\matr{\Lambda}\matr{\Pi}$. 
Therefore, $\vect{y}(t)$ and $\vect{s}(t)$ are essentially the same.
\end{theorem}

For a random vector $\vect{x}(t)\in\RR^{m}$, as a higher order generalization of the \emph{covariance matrix}, the $d$-th order \emph{cumulant tensor} $\tens{C}^{\vect{x}}_{d}\in\RR^{m\times m\times\cdots\times m}$ can be considered as a measure of independence, since, if the components of $\vect{x}(t)$ are independent, then the cumulant tensor is diagonal. 
If $\vect{x}(t)=\matr{P}^{\T}\cdot\vect{s}(t)$, the cumulant tensors have the following \emph{multilinear} property
\begin{align*}
\tens{C}^{\vect{x}}_{d} = \tens{C}^{\vect{s}}_{d}\contr{1}\matr{P}^{\T}\contr{2}\cdots\contr{d}\matr{P}^{\T}.
\end{align*}
It follows that the matrix slices have a common decomposition
\begin{align}\label{eq:matr_comm_decom}
(\tens{C}^{\vect{x}}_{d})_{:,:,i_3,\cdots,i_d} =\matr{P}^{\T} (\tens{C}^{\vect{s}}_{d}\contr{3}\vect{p}_{i_3}\contr{4}\cdots\contr{d}\vect{p}_{i_d})\matr{P}\eqdef\matr{P}^{\T} \matr{D}_{i_3,\cdots,i_d}\matr{P},
\end{align}
for all $1\leq i_3, i_4, \cdots, i_d\leq m$.
Now, instead of the canonical approach using the optimization problem \eqref{eq:offdiag_zero}, we now present the following theoretical result based on \cref{thm:trival-Gmn-SDO} for the \( \mathbf{D}_{m, n} \)-SDO property. 
% \begin{theorem}
% In model \eqref{eq:BSS-linear}, if $n\geq m^{d-1}$, then we can always find $\matr{P}\in\textbf{St}(m,n)$ and $\vect{s}(t)\in\RR^{n}$ such that the model \eqref{eq:BSS-linear} holds, and the components of $\vect{s}(t)$ are statistically independent.
% \end{theorem}

\begin{theorem}
In model \eqref{eq:matr_comm_decom}, if $n\geq m^{d-1}$, then we can always find $\matr{P}\in\textbf{St}(m,n)$ such that the matrices $\matr{D}_{i_3,\cdots,i_d}$ are all diagonal. 
\end{theorem}

%$\textbf{St}_{m,n}(\RR)$

\section{Conclusions}\label{sec:conclusi}

Athough the SD and SDO notions are widely used in both theoretical and practical problems, the sets of matrices satisfying these two properties are limited. 
In this paper, using the functions $\varphi_{\rm T}$ in \eqref{eq:func_T} and $\varphi_{\rm D}$ in \eqref{eq:func_D}, 
we extend the SDO and SD from two different angles, and introduce several new notions, which are summarized in \Cref{table-example-3-0}. 
The $\varphi_{\rm D}$ based notions have been mostly studied in \cite{wang2021new}, some results of which are presented in \Cref{sec:decom_based_SD}. 
In \cref{sec:trans_based_SD,sec:TWSD}, We pay most attention to the $\varphi_{\rm D}$ based notions, and characterize them under different assumptions. 

The functions $\varphi_{\rm T}$ and $\varphi_{\rm D}$ look similar to each other.
However, the notions based on them are surprisingly different under various conditions as evidenced in \cref{sec:relationship}, and it is difficult to directly use the relationship between $\varphi_{\rm T}$ and $\varphi_{\rm D}$ to study these new notions.
Therefore, in this paper, we characterize the $\varphi_{\rm T}$ based notions using methods different from that in \cite{wang2021new}. 
We now take the sufficient conditions for TWSD-B and DWSD of nonsingular pairs as an example. Recall that \(\{ \matr{A}, \matr{B} \}\) is SD if and only if the real Jordan normal form of \( \matr{A}^{-1} \matr{B} \) is diagonal. In the proof of \cref{lem:DWSD-nonsingular-pair} ({\cite[Theorem 7]{wang2021new}}), the main idea of \cite{wang2021new} is to perturb the matirx \( \matr{B} \) such that \( \matr{A}^{-1} \matr{B} \) has distinct real eigenvalues, which implies diagonal Jordan normal form. While in the proof of \cref{thm:TWSD-B-nonsingular-pair}, we construct a sequence \(\{ \matr{P}_k \}_{k \geq 1}\) such that the off-diagonal elements of the Jordan normal form of \( \matr{P}_k^{-1} \matr{A}^{-1} \matr{B} \matr{P}_k \) tend to 0, so that the Jordan normal form of its limit is diagonal.
%In this process, our method preserves the eigenvalues of \( \matr{A}^{-1} \matr{B} \), while the method of \cite{wang2021new} may change them.

As two examples, in \cref{sec:application-QCQP,sec:application-approximate-diagonalization}, we consider the applications of these new notions to the well-known QCQP and BSS problems, respectively. It can be expected that, as the ``weakly" and ``projectively" extensions of SD and SDO, they will have more potential applications in other areas.

% There are also some issues unsolved in this paper, which may deserve further investigation. 

% In this paper, based on the error measures \( \varphi_T \) and \( \varphi_D \), we proposed new notions of {\it simultaneous diagonalization}\/ which are signified by the adjectives ``projectively'' and ``weakly'': \( \mathbf{T}_{m,n} \)-SD(SDO), TWSD, TWSD-B, \( \mathbf{D}_{m,n} \)-SD (SDO) and DWSD. For \( \mathbf{T}_{m,n} \)-SD (SDO), we showed that they are equivalent to SD (SDO). Then, we studied the characterizations of TWSD-B. Given any nonsingular matrix \( \matr{A} \) and any matrix \( \matr{B} \) , we showed they are TWSD-B if and only if \( \matr{A}^{-1} \matr{B} \) has only real eigenvalues by the canonical form of the matrix pair, while for any singular pair, we showed that they are actually always TWSD-B. For any multiple nonsingular matrices, we also give a necessary condition and two sufficient conditions for TWSD-B. We also derived the connection between TWSD and TWSD-B, by which we obtain a sufficient condition for TWSD. For the corresponding notions defined by \( \varphi_D \), we established the equivalence between them and the definitions introduced in \cite{wang2021new}. Moreover, we showed that given any set of matrices, we can always find sufficiently large \( n \) such that they are \( \mathbf{D}_{m,n} \)-SDO. After that, we demonstrated the relationship between these new notions for matrix pair, nonsingular set and positive definite set. Finally, we considered the applications of these new notions on QCQP and ICA respectively.

There are still missing characterizations, counter-examples, and open questions to complete the picture, which may deserve further investigations. 
\begin{itemize}
\item Are the conditions in  \cref{thm:TWSD-B-multiple-sufficient-condition}, \cref{thm:TWSD-block-TWSD-B-suffic} and \cref{cor:TWSD-nonsingular} also necessary?  
\item We have shown that any set \( \mathcal{C} \) is \( \mathbf{D}_{m,n} \)-SD (\( \mathbf{D}_{m,n} \)-SDO) when \( n \) is large enough. It may be interesting to further estimate the smallest \( n \) such that any set \( \mathcal{C} \) is always \( \mathbf{D}_{m,n} \)-SD (\( \mathbf{D}_{m,n} \)-SDO).
\end{itemize}

% Despite the similarity between \( \varphi_T \) and \( \varphi_D \), what followed from them are surprisingly different under various conditions, as evidenced in \cref{sec:relationship}. Though the relations have been established under a host of necessary and/or sufficient conditions for each one of them on an individual basis, it will be interesting to know if they are related in a more direct manner. 
%unknown whether there is more direct method to study their relationships.

\appendix%\label{appendix-sec}

\section{Long proofs in \Cref{sec:trans_based_SD}}\label{proofs-1}

Before the proofs of \cref{thm:better-decomposition,thm:TWSD-B-nonsingular-pair}, we first need to present several lemmas. 

\begin{lemma}\label{lem:weak-Jordan-form}
Let 
$\matr{A}\in\RR^{m\times m}$. 
Then there exist a sequence \( \{ \matr{P}_k \}_{k \geq 1} \subseteq \SL_m(\mathbb{R}) \) such that \( \matr{P}_k^{-1}\matr{A}\matr{P}_k \) converges to a diagonal matrix if and only if \( \matr{A} \) has only real eigenvalues.
\end{lemma}
\begin{proof}
If \( \matr{A} \) has only real eigenvalues, its Jordan normal form is derived by
\[ \matr{A} = \matr{P}^{-1}  \Diag{\matr{J}(\lambda_1, m_1) , \ldots , \matr{J}(\lambda_r, m_r)} \matr{P},\]
with $\lambda_s\in\RR$ for $1\leq s\leq r$. 
Let \( \matr{P}_k = \matr{P} \Diag{\matr{R}_{k}(m_1) , \ldots , \matr{R}_{k}(m_r)} \) for $k\geq 1$. Then, by equation \eqref{eq:R-sequence-J-limit}, we have that 
\[ \lim_{k \to \infty}\matr{P}_k^{-1}\matr{A}\matr{P}_k = \Diag{\lambda_1\matr{I}_{m_1} , \ldots , \lambda_r \matr{I}_{m_r}},\] which is a diagonal matrix. 
Conversely, if \( \lim_{k \to \infty}\matr{P}_k^{-1}\matr{A}\matr{P}_k=\matr{D}\), where $\matr{D}\in\textbf{D}_{m}$ is a diagonal matrix, then their characteristic polynomials also converge to the characteristic polynomial of $\matr{D}$, which has only real roots. 
Note that \( \matr{P}_k^{-1}\matr{A}\matr{P}_k \) and \( \matr{A} \) have the same characteristic polynomial. The characteristic polynomial of \( \matr{A} \) has only real roots as well, and thus \( \matr{A} \) has only real eigenvalues.
The proof is complete. 
\end{proof}

\begin{lemma}
Let $\matr{A}\in\textbf{symm}(\RR^{m\times m})$. 
Then there exist a sequence \( \{ \matr{P}_k \}_{k \geq 1} \subseteq \SL_m(\mathbb{R}) \) such that $\lim\limits_{k \to \infty}\matr{P}_k^{\T}\matr{A}\matr{P}_k = \matr{0}_{m\times m}$ if and only if \( \det(\matr{A}) = 0 \).
\end{lemma}
\begin{proof}
Since \( \det (\matr{P}_k) = 1 \), we have \( \det (\matr{P}_k^{\T}\matr{A}\matr{P}_k) = \det (\matr{A})\) for \( k \geq  1 \). Note that the determinant is a continuous function of matrices. If \(\lim_{k \to \infty} \matr{P}_k^{\T}\matr{A}\matr{P}_k = \matr{0}_{m\times m} \), then \[\det (\matr{A}) = \lim_{k \to \infty} \det (\matr{P}_k^{\T}\matr{A}\matr{P}_k) = \det(\lim_{k \to \infty} \matr{P}_k^{\T}\matr{A}\matr{P}_k) = 0.\]
Conversely, if \( \det(\matr{A}) = 0 \), without loss of generality, we choose $\matr{P}\in\GL_{m}(\RR)$ such that 
\[ \matr{P}^{\T}\matr{A}\matr{P} = \Diag{\matr{I}_p, -\matr{I}_q, \matr{0}_{r\times r}},\] 
with $p,q\geq 0$ and \( r \geq 1 \). Let \( \matr{P}_k = \matr{P} \Diag{1/k, 1/k, \ldots , 1/k, k^{m-1}/\det(\matr{P}) }\) for $k\geq 1$. Then 
\[ \matr{P}_k^{\T}\matr{A}\matr{P}_k = \Diag{\matr{I}_p / k^2, -\matr{I}_q / k^2, \matr{0}_{r\times r}} \to \matr{0}_{m \times m} \mbox{ when } k \to \infty,\] 
and $\matr{P}_k\in  \SL_{m}(\RR)$ for $k\geq 1$. The proof is complete.
\end{proof}

\begin{lemma}\label{lem:bounded-inverse}
Let \(\{\matr{X}_k\}_{k \geq 1} \subseteq\GL_m(\RR)\) be a sequence of matrices having the same determinant, \emph{i.e.}, \( \det(\matr{X}_k) = c \neq 0 \) for all $k\geq 1$. \\
(i) If there exists \( M >0 \) such that \( \|\matr{X}_k\| \leq M \) for all \(k\geq 1\), then there exists \(M'>0\) such that \( \|\matr{X}_k^{-1}\| \leq M' \) for all \(k\geq 1\). \\
(ii) If \( \lim_{k\to \infty} \matr{X}_k = \matr{D} \) with \( \matr{D} \in \mathbf{D}_m \), then \(\lim_{k\to \infty}\matr{X}_k^{-1}=\matr{D}^{-1}\).
\end{lemma}

\begin{remark}
In \cref{lem:bounded-inverse}(ii), if we only know that \( \lim_{k\to \infty} \textbf{offdiag}(\matr{X}_k) = \matr{0}_{m\times m}\),
it is not necessary that \( \lim_{k\to \infty} \textbf{offdiag}(\matr{X}_k^{-1}) = \matr{0}_{m\times m}\).
For example, let
\[
\matr{X}_k = \begin{bmatrix}
2/k & 1/k & 0 \\
1/k & 1/k & 0 \\
0 & 0 & k^2
\end{bmatrix}
, \quad 
\matr{X}_k^{-1} = \begin{bmatrix}
k & -k & 0 \\
-k & 2k &  0 \\
0 & 0 & 1/k^2
\end{bmatrix},
\]
for $k\geq 1$.
It can be seen that \( \lim_{k\to \infty} \textbf{offdiag}(\matr{X}_k) = \matr{0}_{m\times m}\), while $\textbf{offdiag}(\matr{X}_k^{-1})$ doesn't converge.
\end{remark}

\begin{remark}
Let the set $\mathcal{C}$ be as in \eqref{set_C} satisfying  $\mathcal{C}\subseteq\GL_m(\RR)$. 
It can be seen from \cref{lem:bounded-inverse} that $\mathcal{C}$ is TWSD-B if and only if the set $\{\matr{A}_{i}^{-1}\}_{1\leq i\leq L}$ is TWSD-B.
\end{remark}

\begin{lemma}\label{lem:E-F-simul}
Let $\matr{E}(m)$ and $\matr{F}(m)$ be as in \eqref{def:E-F-H}. Then the set \( \{ \matr{E}(m), \matr{F}(m) \} \) is TWSD-B.
\end{lemma}

\begin{proof}
Let \( \matr{R}_{k}(m)\) be as in \eqref{def:matrix-G-R-k} for $k\geq 1$. 
Then we have that 
\[
\matr{R}_{k}(m)^{\T}\matr{E}(m)\matr{R}_{k}(m) = \matr{E}(m) ,\quad \matr{R}_{k}(m)^{\T}\matr{F}(m)\matr{R}_{k}(m) = \frac{1}{k} \matr{F}(m).
\]
Let $\matr{Q}\in\SON_{m}$ satisfy that \( \matr{Q}^{\T}\matr{E}(m)\matr{Q} = \matr{G}(m)\),
where $\matr{G}(m)$ is the diagonal matrix as in \eqref{def:matrix-G-R}.
Let \( \matr{P}_k = \matr{R}_{k}(m)\matr{Q} \). 
Then
\begin{equation}\label{eq:proof_lemm_EF}
\matr{P}_k^{\T}\matr{E}(m)\matr{P}_k = \matr{Q}^{\T}\matr{E}(m)\matr{Q} = \matr{G}(m), \quad \matr{P}_k^{\T}\matr{F}(m)\matr{P}_k = \frac{1}{k} \matr{Q}^{\T}\matr{F}(m)\matr{Q}.
\end{equation}
It can be seen that \( \lim_{k \to \infty}\matr{P}_k^{\T}\matr{E}(m) \matr{P}_k = \matr{G}(m)\) and \(\lim_{k \to \infty} \matr{P}_k^{\T}\matr{F}(m)\matr{P}_k = \matr{0}_{m\times m} \). 
The proof is complete.
\end{proof}

\begin{proof}[Proof of \cref{thm:better-decomposition}]
By \cref{lem:Uhlig-canonical-nonsingular}, there exists \(\matr{\tilde{P}}\in\GL_{m}(\RR)\) such that
\begin{align*} 
{\small \matr{\tilde{P}}^{\T}\matr{A}\matr{\tilde{P}}} &{\small= \Diag{\sigma_1 \matr{E}(m_{1}), \ldots, \sigma_r\matr{E}(m_{r}), \matr{E}(m_{r+1}), \ldots \matr{E}(m_{p})},}\\
{\small\matr{\tilde{P}}^{\T}\matr{B}\matr{\tilde{P}}} &{\small= \textbf{Diag}\{\sigma_1 \matr{E}(m_{1})\matr{J}(\lambda_{1}, m_{1}), \ldots, \sigma_r\matr{E}(m_{r})\matr{J}(\lambda_{r}, m_{r})},\\
&\ \ \ \ \ \ \ \ \ \   {\small\matr{E}(m_{r+1})\matr{J}(\lambda_{r+1}, m_{r+1}), \ldots \matr{E}(m_{p})\matr{J}(\lambda_{p}, m_{p})\}},
\end{align*}
where \(\lambda_s \in \mathbb{R}\) for \(1 \leq s \leq r\) and \( \lambda_s \in \mathbb{C}\backslash\mathbb{R} \) for \( r < s \leq p \). For \( 1 \leq s \leq r \), note that
\begin{equation}\label{eq:EJ-E-F}
{\small\matr{E}(m_s)\matr{J}(\lambda_s, m_s) = \begin{bmatrix}
0       &   &   &   & \lambda_s \\
  &   &   & \iddots & 1       \\
  &   & \iddots & \iddots &         \\
  & \iddots & \iddots &   &         \\
\lambda_s & 1 &   &   & 0       \\
\end{bmatrix} = \lambda_s \matr{E}(m_s) + \matr{F}(m_s).}
\end{equation}
By \cref{lem:E-F-simul} and \eqref{eq:proof_lemm_EF}, we see that \( \matr{E}(m_s) \) and  $\matr{E}(m_s)\matr{J}(\lambda_s, m_s)$ are TWSD-B, and there exists a sequence \( \{\matr{P}_{k}^{(s)}\}_{k\geq 1}\subseteq\SL_{m_s}(\RR)\) such that
\[
(\matr{P}_k^{(s)})^{\T}\matr{E}(m_s)\matr{P}_k^{(s)} = \matr{G}(m_s), \ \lim_{k \to \infty}(\matr{P}_k^{(s)})^{\T}\matr{F}(m_s)\matr{P}_k^{(s)} = \matr{0}_{m_s \times m_s},
\]
where $\matr{G}(m_s)$ is defined as in \eqref{def:matrix-G-R}. 
Then it follows by \eqref{eq:EJ-E-F} that 
\[
\lim_{k \to \infty} (\matr{P}^{(s)}_k)^{\T}\matr{E}(m_s) \matr{J}(\lambda_s, m_s)\matr{P}^{(s)}_k =   \lim_{k \to \infty} (\matr{P}^{(s)}_k)^{\T}(\lambda_s\matr{E}(m_s) + \matr{F}(m_s))\matr{P}^{(s)}_k = \lambda_s \matr{G}(m_s).
\]
Now we let
\( \matr{P}_k = \matr{\tilde{P}}\Diag{\matr{P}^{(1)}_k , \ldots , \matr{P}_k^{(r)}, \matr{I}_{m_{r+1}} , \ldots , \matr{I}_{m_p}}.
\)
Then 
\begin{align*}
{\small\matr{P}_k^{\T}\matr{A}\matr{P}_k} &{\small = \Diag{\sigma_1\matr{G}(m_1), \cdots,  \sigma_r\matr{G}(m_r), \matr{E}(m_{r+1}), \cdots, \matr{E}(m_{p})}, \ \forall k,} \\
{\small\lim_{k \to \infty} \matr{P}_k^{\T}\matr{B}\matr{P}_k} &{\small= \Diag{\lambda_1\sigma_1\matr{G}(m_1), \cdots,  \lambda_r\sigma_r\matr{G}(m_r), \matr{E}(m_{r+1})\matr{J}(\lambda_{r+1}, m_{r+1}), \cdots, \matr{E}(m_{p})\matr{J}(\lambda_p, m_p)}.}
\end{align*}
The proof is complete. 
\end{proof}

\begin{proof}[Proof of \cref{thm:TWSD-B-nonsingular-pair}]
If the set \(\{ \matr{A}, \matr{B} \} \) is TWSD-B, there exists a sequence \(\{\matr{P}_k\}_{k\geq 1} \subseteq \SL_m(\mathbb{R}) \) such that the sequences \( \matr{P}_k^{\T}\matr{A}\matr{P}_k \) and \( \matr{P}_k^{\T}\matr{B}\matr{P}_k \) both converge to diagonal matrices. By \cref{lem:bounded-inverse}, the sequence \( \matr{P}_k^{-1} \matr{A}^{-1}\matr{P}_k^{-\T}\) converges to a diagonal matrix, and thus \( (\matr{P}_k^{-1} \matr{A}^{-1}\matr{P}_k^{-\T})(\matr{P}_k^{\T}\matr{B}\matr{P}_k) = \matr{P}_k^{-1}\matr{A}^{-1}\matr{B}\matr{P}_k\) also converges to a diagonal matrix.
By \cref{lem:weak-Jordan-form}, the product \( \matr{A}^{-1}\matr{B} \) has only real eigenvalues. 
Conversely, if \( \matr{A}^{-1}\matr{B} \) has only real eigenvalues, by \cref{thm:better-decomposition}, the set \(\{ \matr{A}, \matr{B} \} \) is TWSD-B, since there are only Jordan blocks associated with real eigenvalue in the canonical form of \( \matr{A} \) and \( \matr{B} \).
The proof is complete. 
\end{proof}

\begin{proof}[Proof of \Cref{thm:TWSD-B-singular-pair}]
  By \cref{lem:lancaster-canonical-general-pair}, there exists a nonsingular matrix \( \matr{\bar{P}} \) such that
  \[
    \matr{\bar{P}}^{\top} \matr{A} \matr{\bar{P}} = \Diag{\matr{X}_1 , \ldots , \matr{X}_p} \mbox{ and } \matr{\bar{P}}^{\top} \matr{B} \matr{\bar{P}} = \Diag{\matr{Y}_1 , \ldots , \matr{Y}_p}
  \]
are both block diagonal matrices with the same block structure. Since \( \{ \matr{A}, \matr{B} \} \) is a singular pair, $p_4+p_5 > 0$ in  \cref{lem:lancaster-canonical-general-pair}.
Then it is sufficient to prove that \( \matr{A} \) and \( \matr{B} \) are TWSD-B in the following two cases:

\textbf{Case 1:} $p_5 = 1$. Then \( \matr{X}_p = \matr{Y}_p = \matr{0}_{m_p\times m_p} \). Let
\[\matr{U_k} = \Diag{\underbrace{1/k, 1/k , \ldots , 1/k}_{\small (m-1) \mbox{ elements}}, k^{m-1}}\]
for \( k \geq 1 \). Then we have that \( \det (\matr{U}_k) = 1 \) and
\begin{align*}
  \matr{U}_k^{\top} \matr{\bar{P}}^{\top} \matr{A} \matr{\bar{P}} \matr{U}_k = \Diag{\matr{X}_1/k^2 , \ldots , \matr{X}_{p-1}/k^2, \matr{0}_{m_p\times m_p}} \to \matr{0}_{m\times m},\\
  \matr{U}_k^{\top} \matr{\bar{P}}^{\top} \matr{B} \matr{\bar{P}} \matr{U}_k = \Diag{\matr{Y}_1/k^2 , \ldots , \matr{Y}_{p-1}/k^2, \matr{0}_{m_p\times m_p}} \to \matr{0}_{m\times m},
\end{align*}
when \( k \to \infty \). 
Let \( \matr{P}_k = \matr{\bar{P}} \matr{U}_k \) for \( k \geq 1 \). Then \( \det(\matr{P}_k) = \det(\matr{\bar{P}}) \), and \( \matr{P}_k^{\top} \matr{A} \matr{P}_k, \matr{P}_k^{\top} \matr{B} \matr{P}_k \) both converge to \( \matr{0}_{m\times m} \). It follows that \(\{ \matr{A}, \matr{B} \}\) is TWSD-B.

\textbf{Case 2:} $p_4 \geq 1, p_{4}=0$.
Since \(p_5 = 0\), the last diagonal blocks of their canonical form are
\[
\matr{X}_p = \begin{pmatrix}
& & \matr{E}(m_{p}) \\
& 0 & \\
\matr{E}(m_{p}) & & \\
\end{pmatrix},\ \matr{Y}_p = \matr{F}(2m_{p}+1).
\]
Let \( \matr{W}_k = (2m_p+1)\matr{R}_k \) for \( k \geq 1 \). Then we always have
\begin{equation}\label{eq:singular-pair-1}
\matr{W}_k^{\T}\matr{X}_p\matr{W}_k =\matr{X}_p, \ \matr{W}_k^{\T}\matr{Y}_p \matr{W}_k =  \matr{Y}_p/k.
\end{equation}
Now we define a matrix 
\[  \matr{V}_{k} = \Diag{ k^{- \frac{1}{2(m-1)}} \matr{I}_{m_p},  k^{\frac{1}{2}}, k^{- \frac{1}{2(m-1)}} \matr{I}_{m_p}}\in\RR^{( 2m_p+1)\times ( 2m_p+1)}.\]
By equation \eqref{eq:singular-pair-1}, we have that
\begin{align*}
  & \matr{V}_k^{\top} \matr{W}_k \matr{X}_p \matr{W}_k \matr{V}_k = k^{- \frac{1}{(m-1)}} \matr{X}_p \to \matr{0}_{(2m_p+1)\times(2m_p+1)}, \\
  & \matr{V}_k^{\top} \matr{W}_k \matr{Y}_p \matr{W}_k \matr{V}_k =
  {\tiny
    \left[
\begin{array}{ccc|c|ccc}
 &                                            &  &                   &                   &                                            & \\
 &                                            &  &                   &                   &  k^{- \frac{m}{m-1}} \matr{F}(m_p) & \\
 &                                            &  &                   &                   &                                            & \\
 \hline
 &                                            &  & 0                 & k^{- \frac{1}{2}} &                                            & \\
 \hline
 &                                            &  & k^{- \frac{1}{2}} &                   &                                            & \\
 &  k^{- \frac{m}{m-1}} \matr{F}(m_p) &  &                   &                   &                                            & \\
 &                                            &  &                   &                   &                                            &
\end{array}\right]} \to \matr{0}_{(2m_p+1)\times(2m_p+1)},
\end{align*}
when \( k \to \infty \). 
Let \( \matr{P}_k = \matr{\bar{P}} \Diag{\matr{I}_{m-2m_p-1}, \matr{W}_k} \Diag{k^{- \frac{1}{2(m-1)}}\matr{I}_{m-2m_{p}-1}, \matr{V}_k} \).
Then we have that  \( \det(\matr{\bar{P}}_k) = \det(\matr{\bar{P}}) \) and 
\begin{align*}
  \matr{P}_k^{\top} \matr{A} \matr{P}_{k} = \textbf{Diag}\left\{k^{- \frac{1}{m-1}} \matr{X}_1 , \ldots , k^{- \frac{1}{m-1}} \matr{X}_{p-1}, \matr{V}_k^{\top} \matr{W}_k^{\top} \matr{X}_p \matr{W}_k \matr{V}_k\right\} \to \matr{0}_{m\times m}, \\
  \matr{P}_k^{\top} \matr{B} \matr{P}_k = \textbf{Diag}\left\{k^{- \frac{1}{m-1}} \matr{Y}_1 , \ldots , k^{- \frac{1}{m-1}} \matr{Y}_{p-1}, \matr{V}_k^{\top} \matr{W}_k^{\top} \matr{Y}_p \matr{W}_k \matr{V}_k\right\} \to \matr{0}_{m\times m},
\end{align*}
when \( k \to \infty \). 
The proof is complete.
\end{proof}

Before the proofs of  \Cref{thm:TWSD-B-multiple-sufficient-condition,thm:TWSD-B-three-multiple-sufficient-condition}, we first need to present several lemmas. 

\begin{lemma}\label{lem:R-lower-stripped-R}
Let \( \matr{A} \in \mathbb{R}^{m_1 \times m_2} \) satisfy \( A_{ij} = 0 \) for \( i+j \leq \max \{ m_1, m_2 \}  \), {\it i.e.}, 
\begin{equation*}\label{def:lower-stripped-matrix}
{\footnotesize\matr{A}=\left[\begin{array}{lllll}
      0                & \cdot   & \cdot   & \cdot   & 0                          \\
      \cdot            &         &         &         & \cdot                      \\
      \cdot            &         &         &         & \cdot                      \\
      \cdot            &         &         &         & 0                          \\
      \cdot            &         &         & \iddots & A_{m_1-m_2+1,m_{2}} \\
      \cdot            &         & \iddots & \iddots & \cdot                      \\
      \cdot            & \iddots & \iddots & \cdot   & \cdot                      \\
      0                & \iddots &         &         & \cdot                      \\
      A_{m_1,1} & \cdot   & \cdot   & \cdot   & A_{m_1,m_{2}}
    \end{array}\right]
  \mbox{ or }
  \ \left[\begin{array}{ccccccccc}
            0     & \cdot & \cdot & \cdot & \cdot & \cdot                     & \cdot & 0     & A_{1,m_{2}} \\
            \cdot &       &       &       &       &                           & \cdot & \iddots & \cdot              \\
            \cdot &       &       &       &       & \cdot                     & \iddots &       & \cdot              \\
            0     & \cdot & \cdot & \cdot & 0     & A_{m_1, m_2-m_1+1} & \cdot & \cdot & A_{m_1,m_{2}}
          \end{array}\right].}
\end{equation*}
% all of the elements above anti-diagonal line are \( 0 \),
Then
\begin{equation}\label{eq:lemm_Abar}
\lim_{k \to \infty} \matr{R}_{k}(m_1)^{\T}\matr{A}\matr{R}_{k}(m_{2}) = \left\{
\begin{array}{cl}
  \matr{0}_{m_1\times m_2}, & \mbox{ if }m_1 \neq m_2; \\
  \bar{\matr{A}}, & \mbox{ if } m_1 = m_2,
\end{array}
\right.
\end{equation}
where $\matr{R}_{k}(m_1)$ and $\matr{R}_{k}(m_2)$ are as in \eqref{def:matrix-G-R-k}, and \( \bar{\matr{A}} \in \mathbb{R}^{m \times m} \) is defined as:  
\[ \bar{A}_{ij} = \left\{
\begin{array}{cl}
A_{ij}, & \mbox{if }i+j = m+1; \\
0, & \mbox{otherwise}.
\end{array}
\right. \]
\end{lemma}

\begin{proof} Let \(\tilde{\matr{A}}^{(k)} =   \matr{R}_{k}(m_{1})^{\T}\matr{A}\matr{R}_{k}(m_{2}) \) for $k\geq 1$. 
Then
\begin{align*}
(\tilde{A}^{(k)})_{i,j}= (R_{k}(m_{1}))_{i,i}A_{ij}(R_{k}(m_{2}))_{j,j} = k^{(m_1+1)/2-i}A_{ij}k^{(m_2+1)/2-j} = k^{(m_1+m_2)/2+1-(i+j)}A_{ij}.
\end{align*}
If \( i+j \leq \max\{m_1, m_2\} \), since \(A_{ij} = 0 \), we have \((\tilde{A}^{(k)})_{i,j}= 0 \). 
If \( i+j \geq \max\{m_1, m_2\} + 1 \) and \( m_1 \neq m_2 \), we have \( (m_1+m_2)/2+1-(i+j) < 0 \), and thus \( \lim_{k \to \infty} (\tilde{A}^{(k)})_{i,j} =0 \).
If \( i+j > \max\{m_1, m_2\} + 1 \) and \( m_1 = m_2 \), we still have \( (m_1+m_2)/2 + 1 - (i+j) < 0 \), and thus \( \lim_{k \to \infty} (\tilde{A}^{(k)})_{i,j} =0 \).
If \( i+j = \max\{m_1, m_2\} + 1 \) and \( m_1 = m_2 \), then $(\tilde{A}^{(k)})_{i,j}=A_{ij}$. The proof is complete.
\end{proof}

\begin{remark}
In \cref{lem:R-lower-stripped-R}, if \(m_1 = m_2 = m\), then \eqref{eq:lemm_Abar} can be shown more clearly as
\[
\lim_{k \to \infty} \matr{R}_k(m)^{\top}
\begin{bmatrix}
0              & \cdots  & \cdots  & 0       & A_{1,m} \\
\vdots          &         & \iddots & \iddots & A_{2,m}              \\
\vdots          & \iddots & \iddots & \iddots   & \vdots          \\
0              & \iddots &  \iddots       &         & \vdots          \\
A_{m,1} & A_{m,2}       & \cdots  & \cdots  & A_{m,m}
\end{bmatrix}
\matr{R}_k(m) =
\begin{bmatrix}
0              & \cdots  & \cdots  & 0       & A_{1,m} \\
\vdots          &         & \iddots & \iddots & 0              \\
\vdots          & \iddots & \iddots & \iddots   & \vdots          \\
0              & \iddots &  \iddots       &         & \vdots          \\
A_{m,1} & 0       & \cdots  & \cdots  & 0
\end{bmatrix}.
\]
\end{remark}

The following lemma is adapted from the proof of \cite[Lemma 1]{uhlig1976canonical} .
\begin{lemma}[{\cite[Lemma 1]{uhlig1976canonical}}]
\label{lem:matrices-commute-with-jordan-norma-form} 
Suppose that $\matr{A}\in\RR^{m\times m}$ has real eigenvalues \( \lambda_1, \ldots , \lambda_r \in \mathbb{R} \), and its Jordan normal form is
\begin{equation}\label{eq:Jord_A}
\matr{L} = \Diag{\matr{C}(\lambda_1) , \ldots , \matr{C}(\lambda_{r})},
\end{equation}
where \( \matr{C}(\lambda_v) = \Diag{\matr{J}(\lambda_v, m^{(v)}_1) , \ldots , \matr{J}(\lambda_v, m^{(v)}_{s_v})}\) is the \emph{full chain of Jordan blocks} associated with the eigenvalue \( \lambda_v \) , \(s_v\) is the number of Jordan blocks associated with \(\lambda_v\) for $1\leq v\leq r$.
Then we have\\
(i) The ring of matrices commuting with \( \matr{L} \) in \eqref{eq:Jord_A} is the direct sum of the \(r\) rings of matrices commuting with \( \matr{C}(\lambda_v) \) for \( 1 \leq v \leq r \), respectively.
In other words, if \( \matr{X}\matr{L} = \matr{L}\matr{X} \), then \( \matr{X} = \Diag{\matr{X}_1 , \ldots , \matr{X}_r} \) with \( \matr{X}_v \in \mathbb{R}^{M_v \times M_v}\) satisfying \( \matr{X}_v \matr{C}(\lambda_v) = \matr{C}(\lambda_v)\matr{X}_v \), where \(M_v = \sum_{g=1}^{s_v} m_g^{(v)}\) is the size of \( \matr{C}(\lambda_v) \) for \( 1 \leq v \leq r \).\\
(ii) Let \( \matr{X}_v \) be a matrix commuting with \( \matr{C}(\lambda_v)\). Then we can partition \( \matr{X}_v \) into \( s_v^2 \) blocks in the same way as \(\Diag{\matr{J}(\lambda_v, m^{(v)}_1) , \ldots , \matr{J}(\lambda_v, m^{(v)}_{s_v})}\). Denote \( \matr{X}_v \) by \( \matr{X}_v = [\matr{\hat{X}}_v^{(g,h)}]_{1 \leq g,h \leq s_v} \), where \( \matr{\hat{X}}_v^{(g,h)} \in \RR^{m^{(v)}_g \times m^{(v)}_h} \). Then each block \( \matr{\hat{X}}_v^{(g,h)}\) is a triangularly striped matrix, that is, \( \matr{\hat{X}}_i^{(g,h)}\) is of the form
      \begin{equation}\label{def:triangle-stripped-matrix}{\small
      \left[\begin{array}{ccccc}
              (\hat{x}^{(g,h)}_v)_1 & \cdot & \cdot & \cdot & (\hat{x}^{(g,h)}_v)_{m_h} \\
              0 & \ddots & & & \cdot \\
              \vdots &\ddots & \ddots & & \cdot \\
              \vdots & &\ddots& \ddots & \cdot \\
              \vdots & & &\ddots & (\hat{x}^{(g,h)}_v)_1 \\
              \vdots & & & & 0 \\
              \vdots & & & & \vdots \\
              0 & \cdots & \cdots & \cdots & 0
            \end{array}\right]
      \mbox{ or }
      \left[\begin{array}{llllllllll}
              0 & \cdots & \cdots & \cdots & 0 &  (\hat{x}^{(g,h)}_v)_1& \cdot & \cdot & \cdot & (\hat{x}^{(g,h)}_v)_{m_g} \\
              \vdots & & & & &\ddots & \ddots & & & \cdot \\
              \vdots & & & & & & \ddots& \ddots && \cdot \\
              \vdots & & & & & & &\ddots& \ddots & \cdot \\
              0&\cdots &\cdots &\cdots &\cdots &\cdots &\cdots &\cdots & 0 &(\hat{x}^{(g,h)}_v)_1
            \end{array}\right].}
\end{equation}
\end{lemma}

\begin{lemma}[{\cite[Lemma 8]{wang2021new}}]\label{lem-tri-st-eigen}
  Suppose \( \matr{X} = [ \matr{X}_{p,q}]_{1 \leq p,q \leq c} \) and each block \( \matr{X} \) is triangularly striped as \eqref{def:triangle-stripped-matrix}. Define another matrix \( \matr{\tilde{X}} \in \RR^{c \times c} \) by letting \( \tilde{X}_{p,q} \) be the diagonal elements of \( \matr{X}_{p,q} \) if it is square and 0 otherwise. Then \( \matr{X} \) and \( \matr{\tilde{X}} \) have the same eigenvalues.
\end{lemma}
It is easy to check the following lemmas by straightforward computation.
\begin{lemma}\label{lem-perserve-comm}
  Let \( \matr{A}_0, \matr{A}_1, \matr{A}_2 \in \RR^{m \times m} \) satisfying \( [ \matr{A}_1, \matr{A}_2]_{\matr{A}_0} = \matr{0}_{m \times m} \).\\
  (i) Let \( \matr{P} \) be any nonsingualr matrix and \( \matr{\tilde{A}}_i = \matr{P}^{\top} \matr{A}_i \matr{P} \) for \(i=0,1,2\), then \( [\matr{ \tilde{A}}_1, \matr{ \tilde{A}}_2]_{ \matr{ \tilde{A}}_0} = \matr{0}_{m \times m} \).\\
  (ii) Let \(\{ \matr{P}_k \}_{k \geq 1} \subseteq \RR^{m \times m}\) be a sequence of nonsingular matrices such that \( \matr{ \bar{A}}_i = \lim_{k\to \infty}  \matr{P}_k^{\top} \matr{A}_i \matr{P}_ k \) exists for \( i=0,1,2 \). Then \( [\matr{ \bar{A}}_1, \matr{ \bar{A}}_2]_{ \matr{ \bar{A}}_0} = \matr{0}_{m \times m} \).
\end{lemma}

% \begin{lemma}\label{lem:commues-triangle-block-matrices}
%   Let \( \{ \matr{X}_1 , \ldots , \matr{X}_L \} \subseteq \RR^{m \times m} \) be a set of communting matrices with \( \matr{X}_{i} = [\matr{\hat{X}}_{i}^{(g,h)}]_{1\leq g, h\leq s} \) and \( \matr{\hat{X}}_{i}^{(g,h)} \in \mathbb{R}^{m_g \times m_h} \) is a upper triangularly stiped matrix as in \eqref{def:triangle-stripped-matrix} for \( 1 \leq g,h \leq s \). Let \( \matr{Y}_{i} \in \RR^{s \times s} \) to be the matrix by only reserving one diagonal element of each square blocks of \( \matr{X}_{i} \), i.e.
%   \[
%     (\matr{Y}_{i})_{g,h} = \left\{
%       \begin{array}{cl}
%         (\matr{\hat{X}}_{i}^{(g,h)})_{11} & m_{g} = m_h \\
%         0 & m_g \neq m_h
%       \end{array}
%     \right.
%     .
%   \]
%   Then \( \{ \matr{Y}_1 , \ldots , \matr{Y}_{L} \} \) is also a set of communting matrices.
% \end{lemma}

\begin{lemma}\label{lem:product-of-triangularly-matrices-diagonal}
  Let \( \matr{X}_1 \in \RR^{m_1 \times m_2}, \matr{X}_2 \in \RR^{m_2 \times m_1} \) be two triangularly striped matrices. Then \(\matr{X}_1 \matr{X}_2\) is also a triangularly striped matrix. If \(m_1 = m_2\), then the diagonal elements of \(\matr{X}_1\matr{X}_2\) is \(x_1x_2\), where \(x_1\) and \(x_2\) are diagonal elements of \(\matr{X}_1\) and \(\matr{X}_2\) respectively. If \(m_1 \neq m_2\), then the diagonal elements of \(\matr{X}_1\matr{X}_2\) is 0.
\end{lemma}

\begin{proof}[Proof of \cref{thm:TWSD-B-multiple-sufficient-condition}]
  Suppose the Jordan normal form of \( \matr{S}^{-1}\matr{A}_1\) is \(\Diag{\matr{C}(\lambda_1) , \ldots , \matr{C}(\lambda_r)}\), where \( \matr{C}(\lambda_v) \) is the chain of Jordan blocks associated with eigenvalue $\lambda_v$ for $1\leq v\leq r$, as in \cref{lem:matrices-commute-with-jordan-norma-form}. By assumption, \(\lambda_v \in \mathbb{R}\) for \(1 \leq v \leq r\). Here we use the same notations as in  \cref{lem:matrices-commute-with-jordan-norma-form}, that is, for \( 1 \leq v \leq r \), \( s_v \) is the number of Jordan blocks in the Jordan normal form associated with \(\lambda_v\), and the sizes of these blocks are \( m_1^{(v)} , \ldots , m_{s_v}^{(v)} \).

  For the convenience of readers, we start with the simple case where each Jordan chain has only one block, that is, \(s_v = 1\) and \( \matr{C}(\lambda_v) \) is a Jordan block for \( 1 \leq v \leq r \). Then we consider the general case where there may exist multiple Jordan blocks with the same eigenvalue.

\textbf{Case 1:} When \(s_v = 1\) for \( 1 \leq v \leq r \), let \(m_v\) be the size of the Jordan blocks associated with eigenvalue \(\lambda_v\). So \(\matr{C}(\lambda_v) = \matr{J}(\lambda_v, m_v)\) for \(1 \leq v \leq r\).
By \cref{lem:Uhlig-canonical-nonsingular}, there exists a nonsingular matrix \( \matr{\bar{P}} \) such that
\begin{equation}
  \label{eq-Sbar}
 \begin{aligned}
  & \matr{\bar{P}}^{\T}\matr{S} \matr{\bar{P}} = \Diag{\sigma_1 \matr{E}(m_1) , \ldots , \sigma_r \matr{E}(m_r)},\\
  &\matr{\bar{P}}^{\T}\matr{A}_1 \matr{\bar{P}} =  \Diag{\sigma_1 \matr{E}(m_1)\matr{J}(\lambda_1,m_1) , \ldots , \sigma_r\matr{E}(m_r)\matr{J}(\lambda_r, m_r)}.
\end{aligned}
\end{equation}
Let \( \matr{\bar{S}} = \matr{\bar{P}}^{\T} \matr{S} \matr{\bar{P}} \) and \( \matr{\bar{A}}_{i} = \matr{\bar{P}}^{\T} \matr{A}_{i} \matr{\bar{P}} \) for all \( 1 \leq i \leq L \).
Then \( \matr{\bar{S}}^{-1} \matr{\bar{A}}_1 = \Diag{\matr{C}(\lambda_1) , \ldots , \matr{C}(\lambda_r)} \) is in Jordan normal form.
For any \( i \neq j \), \( \matr{\bar{S}}^{-1}\matr{\bar{A}}_i \) commutes with \( \matr{\bar{S}}^{-1}\matr{\bar{A}}_j \) by \cref{lem-perserve-comm}.
Combining it with the fact that
\( \matr{\bar{S}}^{-1}\matr{\bar{A}}_1 \) is in Jordan normal form and there is only one Jordan block for each eigenvalue, \( \matr{\bar{S}}^{-1}\matr{\bar{A}}_i \) is also a block diagonal matrix consisting of \( r \) blocks, and each block is triangularly striped as in \eqref{def:triangle-stripped-matrix} by \cref{lem:matrices-commute-with-jordan-norma-form}.
For all \( 1 \leq i \leq L \), we assume that 
\begin{equation}\label{eq-Sbar-Abar}
  \matr{\bar{S}}^{-1} \matr{\bar{A}}_i = \textbf{Diag}\left\{\matr{\bar{X}}^{(1,i)} , \ldots , \matr{\bar{X}}^{(r,i)}\right\},
\end{equation}
where \(\matr{ \bar{X}}^{(v,i)} \in \RR^{m_{v} \times m_v}\) is triangularly striped for \( 1 \leq v \leq r \). Since the diagonal elements of triangularly striped matrix are the same, we denote the diagonal elements of \( \hat{\matr{X}}^{(i,v)} \) by \(x^{(i,v)}\) for \(1 \leq v \leq r \).

For all \( 1 \leq i \leq L \), combining \eqref{eq-Sbar} and \eqref{eq-Sbar-Abar}, we have
\begin{equation*}\label{eq-Abar}
 \matr{\bar{A}}_i = \matr{\bar{S}} (\matr{\bar{S}}^{-1} \matr{\bar{A}}_i) = \textbf{Diag}\left\{\sigma_1 \matr{E}(m_1) \matr{\bar{X}}^{(1,i)}, \ldots , \sigma_r \matr{E}(m_r) \matr{\bar{X}}^{(r,i)}\right\}.
\end{equation*}
For \( 1 \leq i \leq L\) and \(1 \leq r \leq v \), note that each block \( \sigma_v\matr{E}(m_{v}) \matr{\bar{X}}^{(v,i)} \) satisfies the assumption in \cref{lem:R-lower-stripped-R}.
We have 
\begin{equation}\label{eq:limit-REXR}
 \lim_{k \to \infty} \matr{R}_k(m_v)^{\top}  (\sigma_v\matr{E}(m_v) \matr{\bar{X}}^{(v,i)}) \matr{R}_k(m_v) =  \sigma_v x^{(v,i)} \matr{E}(m_v).
\end{equation}
Suppose that \( \matr{Q}_v \) is the orthogonal matrix such that \( \matr{Q}_v^{\top} \matr{E}(m_v) \matr{Q}_v = \matr{G}(m_v) \), where \( \matr{G} \) is defined as in \eqref{def:matrix-G-R}. Combining it with \eqref{eq:limit-REXR}, we have
\begin{equation}\label{eq-limit-QREX}
 \lim_{k \to \infty} \matr{Q}_v^{\top} \matr{R}_k(m_v)^{\top}  (\sigma_v\matr{E}(m_v) \matr{\bar{X}}^{(v,i)}) \matr{R}_k(m_v) \matr{Q}_v =  \sigma_v x^{(v,i)} \matr{G}(m_v),
\end{equation}
which is diagonal.

Finally, for all \( k \geq 1 \), define \( \matr{P}_k = \matr{\bar{P}} \Diag{\matr{R}_k(m_1) \matr{Q}_1 , \ldots , \matr{R}_k(m_r) \matr{Q}_r} \). By  \eqref{eq-limit-QREX}, we have
\begin{equation}\label{eq-limit-PA}
  \lim_{k \to \infty} \matr{P}_k^{\top} \matr{A}_i \matr{P}_k = \Diag{\sigma_1 x^{(1,i)} \matr{G}(m_1), \ldots , \sigma_r x^{(r,i)} \matr{G}(m_r)},
\end{equation}
for all $ 1 \leq i \leq L$. 
Since \( \det(\matr{R}_k) = 1 \) for \( 1 \leq k \leq r \), \( \det(\matr{P}_k) \) is a constant. 
Then \eqref{eq-limit-PA} implies that \( \mathcal{C} \) is TWSD-B.

\textbf{Case 2:} Now we consider the general case where there may exist multiple Jordan blocks with the same eigenvalue. 
Let \( \matr{C}(\lambda_v) = \Diag{\matr{J}(\lambda_v, m^{(v)}_1) , \ldots , \matr{J}(\lambda_v, m^{(v)}_{s_v})}\) be the Jordan chain associated with eigenvalue \(\lambda_v\) for \( 1 \leq v \leq r \). Without loss of generality, we also assume \(m_1^{(v)} >  m_2^{(v)} > \ldots > m_{s_v}^{(v)} \). Let \( M_v = \sum_{p=1}^{s_v} m_p^{(v)} \).
Then \( \matr{C}(\lambda_v) \in \RR^{M_v \times M_v} \).

By \cref{lem:Uhlig-canonical-nonsingular}, there exists a nonsingular matrix \( \matr{\bar{P}} \) such that
\begin{equation}\label{eq:PSP-PAP}
\begin{aligned}
  \matr{\bar{P}}^{\T}\matr{S} \matr{\bar{P}} & = \textbf{Diag}\{\underbrace{\sigma^{(1)}_1 \matr{E}(m_{1}^{(1)}), \ldots , \sigma^{(1)}_{s_1} \matr{E}(m_{s_1}^{(1)})}_{{\small \mbox{blocks correpsonding to } \matr{C}(\lambda_1)}} , \ldots , \underbrace{\sigma_1^{(r)} \matr{E}(m_1^{(r)}), \ldots, \sigma_{s_r}^{(r)} \matr{E}(m^{(r)}_{s_r})}_{{\small \mbox{blocks corresponding to } \matr{C}(\lambda_r)}}\},\\
  \matr{\bar{P}}^{\T}\matr{A}_1 \matr{\bar{P}} &= \Diag{\underbrace{\sigma^{(1)}_1 \matr{E}(m_1^{(1)}) \matr{J}(\lambda_1, m^{(1)}_1), \ldots , \sigma^{(1)}_{s_1} \matr{E}(m_{s_1}^{(1)}) \matr{J}(\lambda_1, m^{(1)}_{s_1})}_{{\small \mbox{blocks correpsonding to } \matr{C}(\lambda_1)}}, \ldots , \\
  & \ \ \ \ \ \ \ \ \ \ \ \ \ \underbrace{\sigma_{1}^{(r)} \matr{E}(m^{(r)}_{1}) \matr{J}(\lambda_r, m^{(r)}_{1}), \ldots , \sigma_{s_r}^{(r)} \matr{E}(m^{(r)}_{s_r}) \matr{J}(\lambda_r, m^{(r)}_{s_r})}_{{\small \mbox{blocks correpsonding to } \matr{C}(\lambda_r)}}}.
\end{aligned}
\end{equation}
Let \( \matr{\bar{S}} = \matr{\bar{P}}^{\T} \matr{S} \matr{\bar{P}} \) and \( \matr{\bar{A}}_{i} = \matr{\bar{P}}^{\T} \matr{A}_{i} \matr{\bar{P}} \) for all \( 1 \leq i \leq L \). 
Let \( \matr{S}^{(v)} \) and \( \matr{A}^{(v,1)} \) be the diagonal blocks of \( \matr{\bar{S}} \) and \( \matr{\bar{A}_1} \) corresponding to \( \matr{C}(\lambda_v) \), respectively.
In other words, 
\begin{equation}
  \label{eq-Sv-A1v}
\begin{aligned}
 \matr{S}^{(v)} = & \textbf{Diag}\left\{ \sigma_1^{(v)} \matr{E}(m_1^{(v)}) , \ldots , \sigma_{s_v}^{(v)} \matr{E}(m_{s_v}^{(v)})\right\}, \\
 \matr{A}^{(v,1)} = & \textbf{Diag}\left\{ \sigma_1^{(v)} \matr{E}(m_1^{(v)} \matr{J}(m_1^{(v)}, \lambda_v)) , \ldots , \sigma_{s_v}^{(v)} \matr{E}(m_{s_v}^{(v)}) \matr{J}(m_{s_v}^{(v)}, \lambda_v)\right\},
\end{aligned}
\end{equation}
for \( 1 \leq v \leq r \).
Then \( \matr{\bar{S}} = \Diag{\matr{S}^{(1)} , \ldots , \matr{S}^{(r)}} \), \( \matr{\bar{A}}_1 = \Diag{\matr{S}^{(1)} \matr{C}(\lambda_1) , \ldots , \matr{S}^ {(r)}\matr{C}(\lambda_r)} \) and \( \matr{\bar{S}}^{-1} \matr{\bar{A}}_1 = \Diag{\matr{C}(\lambda_1) , \ldots , \matr{C}(\lambda_r)} \).

For any \( i \neq j \), similar to the above Case 1, \( \matr{\bar{S}}^{-1} \matr{\bar{A}}_i \) commutes with \( \matr{\bar{S}}^{-1} \matr{\bar{A}}_j \).
For  \(1 \leq i \leq L \), since \( \matr{\bar{S}}^{-1}  \matr{\bar{A}}_1\) is in Jordan normal form, we have
\begin{equation}\label{eq-Sbar-inv-Abar-2}
 \matr{\bar{S}}^{-1} \matr{\bar{A}}_i = \textbf{Diag}\left\{\matr{\bar{X}}^{(1,i)} , \ldots , \matr{\bar{X}}^{(r,i)}\right\},
\end{equation}
by \cref{lem:matrices-commute-with-jordan-norma-form}.
Moreover, each diagonal block matrix \( \matr{\bar{X}}^{(v,i)} \in \RR^{M_v \times M_v} \) can be partitioned into \(s_v^2\) blocks
\begin{equation}\label{eq-X-partition}
 \matr{\bar{X}}^{(v,i)} = \left[\matr{\bar{X}}_{(p,q)}^{(v,i)}\right]_{1 \leq p,q \leq s_v},
\end{equation}
where \( \matr{\bar{X}}_{(p,q)}^{(v,i)} \in \RR^{m_p^{(v)} \times m_q^{(v)}} \) is triangularly striped as in \eqref{def:triangle-stripped-matrix}.
Denote its diagonal elements by \(x_{(p,q)}^{(v,i)}\).

For \( 1 \leq i \leq L\) and \(1 \leq v \leq r \), define matrix
\begin{equation}\label{eq-A-SX}
 \matr{\bar{A}}^{(v,i)} \eqdef \matr{S}^{(v)} \matr{\bar{X}}^{(v,i)} \in \mathbb{R}^{M_v \times M_v}.
\end{equation}
Then, by equations \eqref{eq-Sbar-inv-Abar-2} and \eqref{eq-Sv-A1v}, we have 
\begin{equation}\label{eq-Abar-i-block}
 \matr{\bar{A}}_{i} = \matr{\bar{S}} (\matr{\bar{S}}^{-1} \matr{\bar{A}}_i) = \textbf{Diag}\left\{\matr{S}^{(1)} \matr{\bar{X}}^{(1,i)} , \ldots , \matr{S}^{(r)} \matr{\bar{X}}^{(r,i)}\right\} = \textbf{Diag}\left\{\matr{\bar{A}}^{(1,i)} , \ldots , \matr{\bar{A}}^{(r,i)}\right\}.
\end{equation}
We also partition \( \matr{\bar{A}}^{(v,i)} \) into \(s_v^2\) blocks, \( \matr{\bar{A}}^{(v,i)} = [ \matr{\bar{A}}_{(p,q)}^{(v,i)}]_{1 \leq g,h \leq s_v} \) in the same way as in \eqref{eq-X-partition}.
Combining \eqref{eq-Sv-A1v}, \eqref{eq-X-partition} with \eqref{eq-A-SX}, we have
 \begin{equation}
 \matr{\bar{A}}_{(p,q)}^{(v,i)} = \sigma_p^{(v)} \matr{E}(m_p^{(v)}) \matr{\bar{X}}_{(p,q)}^{(v,i)},
\end{equation}
which satisfies the assumption in \cref{lem:R-lower-stripped-R}. Thus,
\begin{equation}\label{eq:3-25-RAR}
  \lim_{k \to \infty} \matr{R}_k(m_p^{(v)})^{\top} \matr{\bar{A}}_{(p,q)}^{(v,i)} \matr{R}_k(m_q^{(v)}) = \left\{
    \begin{array}{cl}
      \matr{0}_{m_p^{(v)} \times m_q^{(v)}} & m_p^{(v)} \neq m_q^{(v)}, \\
      \sigma_p^{v}x_{(p,q)}^{(v,i)} \matr{E}(m_p^{(v)}) &  m_p^{(v)} = m_q^{(v)}.
    \end{array}
  \right.
\end{equation}
Let \( \matr{U}_k^{(v)} = \Diag{\matr{R}_k(m_1^{(v)}) , \ldots , \matr{R}_k(m_{s_v}^{(v)})} \). Then
\begin{equation}\label{eq-U-A0U}
  (\matr{U}_k^{(v)})^{\top} \matr{\bar{A}}^{(v,i)} \matr{U}_k^{(v)} =
  \begin{bmatrix}
    \matr{R}_k(m_1^{(v)}) \matr{\bar{A}}^{(v,i)}_{(1,1)} \matr{R}_k(m_1^{(v)}) & \cdots & \matr{R}_k(m_1^{(v)}) \matr{\bar{A}}^{(v,i)}_{(1,s_{v})} \matr{R}_k(m_{s_v}^{(v)}) \\
    \vdots & \ddots & \vdots\\
    \matr{R}_k(m_{s_v}^{(v)}) \matr{\bar{A}}^{(v,i)}_{(s_v,1)} \matr{R}_k(m_1^{(v)}) & \cdots & \matr{R}_k(m_{s_v}^{(v)}) \matr{\bar{A}}^{(v,i)}_{(s_{v},s_{v})} \matr{R}_k(m_{s_v}^{(v)}) \\
  \end{bmatrix}.
\end{equation}
Combining it with \eqref{eq:3-25-RAR}, we have \( \lim_{k\to \infty} (\matr{U}_k^{(v)})^{\top} \matr{\bar{A}}^{(v,i)} \matr{U}_k^{(v)} \) exists, and nonsquare blocks will converge to zero matrices. 
Since the sizes of Jordan blocks with the same blocks are different, the limit is a block diagonal matrix:
\begin{equation}\label{eq:limit-RAR}
 \lim_{k\to \infty} \matr{U}_k^{(v)} \matr{\bar{A}}^{(v,i)} \matr{U}_k^{(v)} = \textbf{Diag}\left\{ \sigma_1^{(v)} x_{(1,1)}^{(v,i)} \matr{E}(m_1^{(v)}), \ldots , \sigma_{s_v}^{(v)} x_{(s_v, s_v)}^{(v,i)} \matr{E}(m_{s_v}^{(v)})\right\}.
\end{equation}

Suppose \( \matr{Q}_p^{(v)} \) is the orthogonal matrix such that \( (\matr{Q}_p^{(v)})^{\top} \matr{E}(m_p^{(v)}) \matr{Q}_p^{(v)} \) is diagonal, and
\[
  \matr{V} = \textbf{Diag}\left\{ \matr{Q}_1^{(1)} , \ldots , \matr{Q}_{s_1}^{(1)} , \ldots , \matr{Q}_1^{(r)} , \ldots , \matr{Q}_{s_r}^{(r)}\right\}.
\]
Then by equation \eqref{eq:limit-RAR}, we know
\[
  \lim_{k \to \infty} (\matr{\bar{P}} \matr{U}_k \matr{V})^{\top} \matr{A}_i \matr{\bar{P}} \matr{U}_k \matr{V} = \matr{V}^{\top} (\lim_{k\to \infty} \matr{U}^{\top}_k \matr{\bar{A}} \matr{U}_k) \matr{V}
\]
exists, and it is diagonal for \( 1 \leq i \leq L \). The proof is complete.
\end{proof}

\begin{proof}[Proof of \cref{thm:TWSD-B-three-multiple-sufficient-condition}]
  Here we use the same notations as in the proof of \cref{thm:TWSD-B-multiple-sufficient-condition}. For \( 1 \leq v \leq r \), since there may exist blocks with the same size and eigenvalue, we assume \(m_1^{(v)} \geq m_2^{(v)} \geq \ldots \geq m_{s_v}^{(v)}\), and there are \(b_v\) kinds of sizes of the Jordan blocks with eigenvalue \(\lambda_v\), and there are \(c_{(t,v)}\) blocks with the \( t \)-th size associated with eigenvalue \(\lambda_v\). Then we have
\begin{align*}
    \matr{\bar{S}}^{(v)}&= \textbf{Diag}\left\{\sigma_1^{(v)} \matr{E}(m_1^{(v)}) , \ldots , \sigma_{s_v}^{(v)} \matr{E}(m^{(v)}_{s_v})\right\}, \\
     \matr{\bar{A}}^{(v,1)} &= \textbf{Diag}\left\{\sigma_1^{(v)} \matr{E}(m_1^{(v)}) \matr{J}(\lambda_v, m_1^{(v)}) , \ldots , \sigma_{s_v}^{(v)} \lambda_v \matr{E}(m^{(v)}_{s_v}) \matr{J}(\lambda_v, m_{s_v})\right\}, \\
    \matr{\bar{A}}^{(v,2)}&= \left [ \sigma_p^{(v)}\matr{E}(m_p^{(v)}) \matr{\bar{X}}_{(p,q)}^{(v,i)}\right]_{1 \leq p,q \leq s_v}.
\end{align*}

Let \(x_{p,q}^{(v,i)}\) be the diagonal-elements of \( \matr{\bar{X}}_{(p,q)}^{(v,i)} \) if it is square, and \( 0 \) otherwise. Define \( \matr{\Sigma} = \Diag{\sigma_1 , \ldots , \sigma_c}, \matr{\tilde{X}} = [ x_{p,q}]_{1 \leq p, q \leq s_v}\in \RR^{s_v \times s_v} \). Since each block of \( (\matr{\bar{S}}^{(v)})^{-1} \matr{\bar{A}}^{(v,2)}  \) is \(  \matr{\bar{X}}^{(v,2)} \) by equation \eqref{eq-X-partition}, \( \matr{\Sigma}^{-1} \matr{\tilde{X}} \) has only real eigenvalues by \cref{lem-tri-st-eigen} and \( \matr{S}^{-1} \matr{A}_2 \) has only real eigenvalues.

  Similar to \eqref{eq:limit-RAR}, for \( i=1,2\) and \(1 \leq v \leq r \), the nonsquare blocks of \( \matr{U}_k^{(v)} \matr{\bar{A}}_{(i,v)} \matr{U}_k^{(v)}\) converge to zero matrices. Moreover, the order of convergence is \( O(\frac{1}{k}) \), and the anti-diagonal elements of each block are equal to the anti-diagonal elements of the corresponding block in \( \matr{\bar{A}}_{(i,v)} \) from the proof of \cref{lem:R-lower-stripped-R}. Thus, we have
  \begin{align*}
    (\matr{U}^{(v)}_k)^{\top}\matr{\bar{S}}^{(v)}\matr{U}^{(v)}_k &= \Diag{\sigma_1^{(v)} \matr{E}(m_1^{(v)}) , \ldots , \sigma_{s_v}^{(v)} \matr{E}(m^{(v)}_{s_v})}, \\
     (\matr{U}_k^{(v)})^{\top} \matr{\bar{A}}^{(v,1)} \matr{U}_k^{(v)} &= \Diag{\sigma_1^{(v)} \lambda_v \matr{E}(m_1^{(v)}) , \ldots , \sigma_{s_v}^{(v)} \lambda_v \matr{E}(m^{(v)}_{s_v})} + O(\frac{1}{k}), \\
    (\matr{U}_k^{(v)})^{\top} \matr{\bar{A}}^{(v,2)} \matr{U}_k^{(v)} &= \Diag{\matr{\hat{X}}^{(1,v)} , \ldots , \matr{\hat{X}}^{(b_v,v)}} + O(\frac{1}{k}),
\end{align*}
where \( \matr{\hat{X}}^{(t,v)} \) can be partitioned into \(c_{(t,v)}^2\) square blocks of size \(m_{(t,v)}\), and each block \( \matr{\hat{X}}^{(t,v)}_{(p,q)} = x_{(p,q)}^{(t,v)} \matr{E}(m_{(t,v)}) \in \RR^{m_{(t,v)} \times m_{(t,v)}} \) for \( 1 \leq t \leq b_v \) and \( 1 \leq u,w\leq c_{(t,v)} \).

Obviously, it is sufficient to consider the blocks correpsonding to \( \matr{\hat{X}}^{(t,v)} \) for fixed \( t \). Without loss of generality, we assume \(b_v = 1\), and ignore the indices \( v, t \) to simplify our notations. Let \( \bar{m} \) be the size of each Jordan block. Consider
\begin{equation}
  \label{eq-simple-U-S-U}
\begin{aligned}
    \matr{U}^{\top}_k\matr{\bar{S}}\matr{U}_k &= \Diag{\sigma_1 \matr{E}(\bar{m}) , \ldots , \sigma_c \matr{E}(\bar{m})}, \\
     \matr{U}^{\top}_k \matr{\bar{A}}^{(1)} \matr{U}_k &= \Diag{\sigma_1 \lambda \matr{E}(\bar{m}) , \ldots , \sigma_{c} \lambda \matr{E}(\bar{m})} + O(\frac{1}{k})  = \lambda \matr{U}^{\top}_k\matr{\bar{S}}\matr{U}_k + O(\frac{1}{k}),\\
    \matr{U}^{\top}_k \matr{\bar{A}}^{(2)} \matr{U}_k &= \left [x_{u,w} \matr{E}(\bar{m})\right]_{1 \leq p,q \leq c} + O(\frac{1}{k}).
\end{aligned}
\end{equation}

 By \cref{thm:better-decomposition}, there exists a sequence \( \{\matr{\tilde{V}}_k\}_{k \geq 1} \) such that
\(  \matr{\tilde{V}}_k^{\top} \matr{\Sigma} \matr{\tilde{V}}_k\) is diagonal. Suppose
\begin{align*}
  \matr{ \tilde{V}}_k^{\top} \matr{\Sigma} \matr{\tilde{V}}_k = & \Diag{\tilde{\sigma} , \ldots , \tilde{\sigma}_c},\\
 \lim_{k \to \infty} \matr{ \tilde{V}}_k^{\top} \matr{\tilde{X}} \matr{\tilde{V}}_{k} = & \Diag{\tilde{x}_1 , \ldots , \tilde{x}_c}.
\end{align*}
Moreover, we can replace \( \delta \) in \eqref{def:matrix-G-R} by \(\frac{\delta}{20c}\) so that the maximum order of the elements of \( \matr{\tilde{V}}_k \) is no larger than \( k^{\frac{1}{5}} \). Define \( \matr{V} = [\tilde{V}_{p,q} \matr{I}_{\bar{m}}]_{1 \leq p,q \leq c} \in \RR^{c\bar{m} \times c\bar{m}}  \). Thus, combining with \eqref{eq-simple-U-S-U}, we have
\begin{equation}
  \label{eq-Vk-Uk-S}
 \begin{aligned}
  \matr{V}^{\top}_k \matr{U}^{\top}_k \matr{ \bar{S}} \matr{U}_k \matr{V}_{k} &= \textbf{Diag}\left\{\tilde{\sigma}_1 \matr{E}(\bar{m}) , \ldots , \tilde{\sigma}_c \matr{E}(\bar{m})\right\},\\
  \matr{V}^{\top}_k \matr{U}^{\top}_k \matr{ \bar{A}}^{(1)} \matr{U}_k \matr{V}_{k} &= \lambda \matr{V}^{\top}_k \matr{U}^{\top}_k\matr{\bar{S}}\matr{U}_k \matr{V}_k + O(\frac{1}{k^{1/5}}),\\
  \matr{V}^{\top}_k \matr{U}^{\top}_k \matr{ \bar{A}}^{(2)} \matr{U}_k \matr{V}_{k} &= \Diag{\tilde{x}_1 \matr{E}(\bar{m}) , \ldots , \tilde{x}_c \matr{E}(\bar{m})} + O(\frac{1}{k^{1/5}}).\\
\end{aligned}
\end{equation}

Let \( \matr{Q} \) be the orthogonal matrix such that \( \matr{Q}^{\top} \matr{E}(\bar{m}) \matr{Q} = \matr{G}(\bar{m})\) and \( \matr{W} = \Diag{ \matr{Q} , \ldots , \matr{Q}} \). By equation \eqref{eq-Vk-Uk-S}, we have
\begin{equation}
  \label{eq-Vk-Uk-S}
 \begin{aligned}
 \matr{W}^{\top} \matr{V}^{\top}_k \matr{U}^{\top}_k \matr{ \bar{S}} \matr{U}_k \matr{V}_{k} \matr{W} &= \Diag{\tilde{\sigma}_1 \matr{G}(\bar{m}) , \ldots , \tilde{\sigma}_c \matr{G}(\bar{m})},\\
\lim_{k \to \infty} \matr{W}^{\top} \matr{V}^{\top}_k \matr{U}^{\top}_k \matr{ \bar{A}}^{(1)} \matr{U}_k \matr{V}_{k} \matr{W}&= \lambda \matr{V}^{\top}_k \matr{U}^{\top}_k\matr{\bar{S}}\matr{U}_k \matr{V}_k, \\
 \lim_{k \to \infty} \matr{W}^{\top} \matr{V}^{\top}_k \matr{U}^{\top}_k \matr{ \bar{A}}^{(2)} \matr{U}_k \matr{V}_{k} \matr{W} &= \Diag{\tilde{x}_1 \matr{G}(\bar{m}) , \ldots , \tilde{x}_c \matr{G}(\bar{m})}.\\
\end{aligned}
\end{equation}
The proof is complete.
\end{proof}

\section{Long proofs in  \Cref{sec:TWSD}}\label{proofs-2}

\begin{proof}[Proof of \Cref{thm:PSD-sufficient}]
Suppose that $\alpha\not=0$ without loss of generality. 
If $\alpha \matr{A} + \beta \matr{B}=\matr{0}_{m \times m}$, then $\matr{A}=-\frac{\beta}{\alpha}\matr{B}$, and thus the set $\{\matr{A}, \matr{B}\}$ is SD.
In general, if $\matr{C} = \alpha \matr{A} + \beta \matr{B} \succeq (\not=)\matr{0}_{m \times m}$, we let $\matr{\tilde{C}}_{k}=(\matr{C}+\frac{1}{k} \matr{I}_{m})^{-1/2}$ for $k\geq 1$, and $\matr{Q}_{k}\in\SON_{m}$ satisfying that $\matr{Q}_{k}^\T \matr{\tilde{C}}_{k} \matr{B} \matr{\tilde{C}}_{k} \matr{Q}_{k}$ is diagonal.
Then 
\[
\matr{I}_{m} = \matr{\tilde{C}}_{k}(\matr{C} + \frac{1}{k} \matr{I}_m) \matr{\tilde{C}}_{k} = \alpha \matr{\tilde{C}}_{k} \matr{A} \matr{\tilde{C}}_{k} + \beta \matr{\tilde{C}}_{k} \matr{B} \matr{\tilde{C}}_{k} +\frac{1}{k} ( \matr{C}+\frac{1}{k} \matr{I}_{m})^{-1}, 
\]
which implies that 
\[
\matr{I}_{m} = \matr{Q}_{k}^\T\matr{I}_{m}\matr{Q}_{k}=  \alpha \matr{Q}_{k}^\T \matr{\tilde{C}}_{k} \matr{A} \matr{\tilde{C}}_{k} \matr{Q}_{k} + \beta \matr{Q}_{k}^\T \matr{\tilde{C}}_{k} \matr{B} \matr{\tilde{C}}_{k} \matr{Q}_{k} +\frac{1}{k} \matr{Q}_{k}^\T ( \matr{C}+\frac{1}{k} \matr{I}_{m})^{-1} \matr{Q}_{k}.
\]
Let $\matr{P}_{k} =  \matr{\tilde{C}}_{k} \matr{Q}_{k}/\det(\matr{\tilde{C}}_{k})^{1/m}$ such that $\det (\matr{P}_{k}) =1$.
We have 
\begin{equation}\label{eq:PAP-I-B-C}
\matr{P}^{\T}_{k} \matr{A} \matr{P}_{k} = \frac{1}{\alpha \det(\matr{\tilde{C}}_{k})^{2/m}} \matr{I}_{m} - \frac{\beta}{\alpha } \matr{P}_{k}^\T \matr{B} \matr{P}_{k} -
\frac{\frac{1}{k}}{\alpha \det(\matr{\tilde{C}}_{k})^{2/m}} \matr{Q}_{k}^\T ( \matr{C}+\frac{1}{k} \matr{I}_{m})^{-1} \matr{Q}_{k}.
\end{equation}
Let a spectral decomposition of $\matr{C}$ be
\[
\matr{C}=\matr{Q}^\T \mbox{\textbf{Diag}} \{\underbrace{\lambda_1,\cdots,\lambda_{m-r}}_{\mbox{\small positive}},\underbrace{0,\cdots,0}_{r}\} \matr{Q},
\]
where $r\ge 1$ is an integer, and $\matr{Q}\in\SON_{m}$.
Then 
\begin{equation}\label{eq:C_1k_inver}
(\matr{C}+\frac{1}{k} \matr{I}_{m})^{-1} = \matr{Q}^\T \mbox{\textbf{Diag}} \{\underbrace{(\lambda_1+\frac{1}{k})^{-1},\cdots,(\lambda_{m-r}+\frac{1}{k})^{-1}}_{\mbox{\small constant order}},\underbrace{k,\cdots,k}_{r}\} \matr{Q},
\end{equation}
and
\begin{equation}\label{eq:C_k_tilde_det}
\det(\matr{\tilde{C}}_{k}) = k^{\frac{r}{2}}\Pi_{s=1}^{m-r}(\lambda_s+\frac{1}{k})^{-\frac{1}{2}}. 
\end{equation}
By equation \eqref{eq:C_k_tilde_det}, the first term in \eqref{eq:PAP-I-B-C} satisfies
\begin{equation*}
\lim_{k\rightarrow\infty}\frac{1}{\alpha \det(\matr{\tilde{C}}_{k})^{2/m}} \matr{I}_{m} = \matr{0}_{m\times m}.
\end{equation*}
By the construction of $\matr{Q}_k$, 
the second term \(-\frac{\beta}{\alpha}\matr{P}_{k}^{\top} \matr{B} \matr{P}_{k} \) is diagonal in \eqref{eq:PAP-I-B-C}.
By equation \eqref{eq:C_1k_inver}, the third term in \eqref{eq:PAP-I-B-C} satisfies
\begin{align*}
{\small\lim_{k\rightarrow\infty}
\frac{\frac{1}{k}}{\alpha \det(\matr{\tilde{C}}_{k})^{2/m}} \matr{Q}_{k}^\T ( \matr{C}+\frac{1}{k} \matr{I}_{m})^{-1} \matr{Q}_{k}
=\lim_{k\rightarrow\infty}
\frac{1}{\alpha \Pi_{s=1}^{m-r}(\lambda_s+\frac{1}{k})^{-\frac{1}{m}}} \matr{Q}_{k}^\T \frac{1}{k^{1+\frac{r}{m}}}( \matr{C}+\frac{1}{k} \matr{I}_{m})^{-1} \matr{Q}_{k}
=\matr{0}_{m\times m}.}
\end{align*}
Above all, we see that $$\lim_{k\rightarrow\infty}\textbf{offdiag}(\matr{P}^{\T}_{k} \matr{A} \matr{P}_{k})=\matr{0}_{m\times m},$$
and thus the set \( \{ \matr{A}, \matr{B} \} \) is TWSD. The proof is complete.
\end{proof}

Before the proof of \Cref{thm:equi_A1_bounded}, we give a lemma to show that under a sequence of similar transformations, if the off-diagonal elements converge to 0, then the diagonal elements are bounded.

\begin{lemma}\label{lem:weak-sim-diag-bounded}
Let \( \matr{A} \in \textbf{symm}(\mathbb{R}^{m \times m}) \). 
If there exists a sequence \( \{\matr{P}_k\}_{k \geq 1} \subseteq  \mathbf{GL}_m(\RR) \) such that the off-diagonal elements of \( \matr{P}_k^{-1}\matr{A}\matr{P}_k \) converge to 0, then the diagonal elements of \( \matr{P}_k^{-1}\matr{A}\matr{P}_k \) are uniformly bounded.
%, and thus there exist a subsequence $\{k_{l}\}_{l\geq 1}\subseteq\NN$ and \( \matr{D}\in\textbf{D}_{m} \) such that \( \lim_{l \to \infty} \matr{P}_{k_l}^{-1}\matr{A}\matr{P}_{k_{l}} = \matr{D} \).
\end{lemma}

\begin{proof}
We prove this result by contradiction. 
Without loss of generality, we assume that there exists a sequence \( \{\matr{P}_k\}_{k \geq 1} \subseteq  \mathbf{GL}_m(\RR) \) such that the first \( r\) \( (1 \leq r \leq m)\) diagonal elements of $\matr{P}_k^{-1}\matr{A}\matr{P}_k$ are unbounded, while the rest elements are bounded. Note that, for any unbounded sequence, we can always find a subsequence which goes to \( \pm \infty \). So without loss of generality, we assume that all of the unbounded diagonal elements go to \( \pm \infty \), that is,
\begin{equation}\label{eq:lim_block_diag}
  \lim_{k\to \infty} \matr{P}_k^{-1}\matr{A}\matr{P}_k = \Diag{\underbrace{\pm \infty, \pm \infty , \ldots, \pm \infty}_{\mbox{\small \( r \) elements}}, \alpha_{r+1} , \ldots , \alpha_m}
,
\end{equation}
where \(\alpha_s \in \mathbb{R}\), for \( r+1 \leq s \leq m \). Now we consider the \( r \times r \) principal minors of \( \matr{P}_k^{-1} \matr{A} \matr{P}_k\). Let \(M_{k}(i_1,i_2 , \ldots , i_r)\) be the determinant of the submatrix of \( \matr{P}_k^{-1} \matr{A} \matr{P} \) consisting of the \( (i_1, i_2 , \ldots , i_r) \)-th rows and columns, where \(1\leq i_1 < i_2  <  \ldots <  i_r\leq m\).
Then \( M_k(1,2 , \ldots , r)\) goes to \(\pm \infty \).
By equation \eqref{eq:lim_block_diag}, we also know that \( \lim_{k\to \infty} M_k(1,2 , \ldots , r) / M_k(i_1, i_2 , \ldots , i_r) = \pm \infty \) for all \( (i_1, i_2 , \ldots , i_r) \neq (1,2 , \ldots , r) \).
% In fact, its order is equal to the product of the first \( m \) diagonal elements. And for all other \( r \times r \) principal minors, they might tend to infty, but their order must less then \(M_k(1,2 , \ldots , r)\) since the submatrix contain less then \( r \) elements which tend to infty. Thus, \( \lim_{k\to \infty} M_k(1,2 , \ldots , r) / M_k(i_1, i_2 , \ldots , i_r) = 0 \) for all \( (i_1, i_2 , \ldots , i_r) \neq (1,2 , \ldots , r) \). 
Let 
%\( M_k \) be the sum of all \( r \times r \) principal minors. Since there are finitly-many principal minors which may tend to infty and there is one term whose order is strictly than others, we have
\[
M_k \eqdef \sum_{1 \leq i_1 < i_2 < \ldots  < i_r \leq m} M_k(i_1, i_2 , \ldots , i_r).
\]
It follows that $\lim_{k\to\infty}M_k=\pm \infty$, since there are finitely-many principal minors which may tend to infinity and there is one term whose order is strictly than others.
However, \(|M_k|\) is equal to the absolute value of the coefficient of the \(\lambda^{m-r}\) in the characteristic polynomial of \( \matr{P}^{-1}_k \matr{A} \matr{P} \), which is always the same as the characteristic polynomial of \( \matr{A} \). Therefore, \(M_k\) is invariant for all \( k \), which contradicts $\lim_{k\to\infty}M_k=\pm\infty$.
The proof is complete. 
\end{proof}

\begin{proof}[Proof of \cref{thm:equi_A1_bounded}]
Since \( \det(\matr{P}_{k}) \) is a constant and \( \matr{P}_k^{\top} \matr{A}_1 \matr{P}_k \) is diagonal and uniformly bounded for all \( k \geq 1 \) , its inverse \( \matr{P}_k^{-1} \matr{A}_1^{-1} \matr{P}_k^{-\top} \) is also diagonal and uniformly bounded by \cref{lem:bounded-inverse}(i). Combining it with the fact that the off-diagonl elements of \( \matr{P}_k^{\top} \matr{A}_i \matr{P}_k \) converge to 0 for all \( 1 \leq i \leq L \), we have the off-diagonal elements of
\begin{equation}
(\matr{P}_k^{-1} \matr{A}_1^{-1} \matr{P}_k^{-\top})(\matr{P}_k^{\top} \matr{A}_{i} \matr{P}_k)=\matr{P}_k^{-1} \matr{A}_1^{-1}\matr{A}_{i} \matr{P}_k
\end{equation}
converge to 0.
By \cref{lem:weak-sim-diag-bounded}, the diagonal elements of \(\matr{P}_k^{-1} \matr{A}_1^{-1}\matr{A}_{i} \matr{P}_k\) are bounded.
It follows that the elements of 
\begin{equation*}
(\matr{P}_k^{\top}\matr{A}_1\matr{P}_k)(\matr{P}_k^{-1} \matr{A}_1^{-1}\matr{A}_{i} \matr{P}_k)= \matr{P}_k^{\top}\matr{A}_{i} \matr{P}_k
\end{equation*}
are also bounded. Thus, the set \( \mathcal{C} \) is TWSD-B. The proof is complete.
\end{proof}

%%%%%%%%%%%%%%%%%%%%%%%%%%%%%%%%%%%%%%%%%%%%%%%%%%%%%%%%%%%%%%%%%%%%%%%

\bibliographystyle{abbrv}
\bibliography{TensorRef}

\end{document}